\documentclass[10pt,a4paper]{article}
\pdfoutput=1 
\usepackage{amsmath}
\usepackage{amsfonts}
\usepackage{enumerate}
\usepackage{amssymb}
\usepackage{accents}
\usepackage{amsthm}
\usepackage{parskip}
\usepackage{graphicx}
\usepackage{caption}
\usepackage{sidecap}
\usepackage{subcaption}
\usepackage{float}
\usepackage{epstopdf}
\usepackage[T1]{fontenc}
\usepackage[utf8]{inputenc}

\usepackage{mathrsfs}

\usepackage{booktabs}
\newtheoremstyle{examplestyletwo}  
  {5mm}       
  {3mm}       
  {\itshape}   
  {}        
  {\bfseries}  
  {}   
  {3mm}       
  {}           

\theoremstyle{examplestyletwo}
\newtheorem{theorem}{Theorem}[section]
\newtheorem{lemma}[theorem]{Lemma}

\newtheoremstyle{examplestyleone}  
  {3mm}       
  {10mm}       
  {\normalfont}   
  {}        
  {\bfseries}  
  {\quad}   
  {1mm}       
  {}           

\theoremstyle{examplestyleone}

\newtheorem{remark}{Remark}

\newtheorem{example}{Example}

\usepackage{fancyhdr}
\fancypagestyle{plain}{ %
  \fancyhf{} 
  
}

\everymath{\displaystyle}
\allowdisplaybreaks
\numberwithin{equation}{section}
\newcommand\numberthis{\addtocounter{equation}{1}\tag{\theequation}}

\graphicspath{{Graphs//}}

\title{Global minima for semilinear optimal control problems} 
\author{Ahmad Ahmad Ali\footnote{Schwerpunkt Optimierung und Approximation, 
Universit\"at Hamburg, Bundesstra{\ss}e 55, 20146 Hamburg, Germany.}, Klaus Deckelnick\footnote{Institut f\"ur Analysis und Numerik,
Otto--von--Guericke--Universit\"at Magdeburg, Universit\"atsplatz 2,
39106 Magdeburg, Germany} \& Michael Hinze\footnote{Schwerpunkt Optimierung und Approximation, 
Universit\"at Hamburg, Bundesstra{\ss}e 55, 20146 Hamburg, Germany.}}

\date{}

\begin{document}

\maketitle

\begin{center}
{\bf \LARGE  }
\end{center}  


{\small {\bf Abstract:} We consider an optimal control problem subject to a semilinear elliptic PDE together with its variational discretization. We provide a condition which allows to decide whether a solution of the necessary first order conditions is a global minimum. This condition can be explicitly evaluated 
at the discrete level. Furthermore, we prove that if the above condition holds uniformly with
respect to the discretization parameter the sequence of discrete solutions converges to a global solution of the corresponding limit problem. Numerical examples with unique global solutions are presented.}
\\[2mm]
{\small {\bf Mathematics Subject Classification (2000): 49J20, 35K20, 49M05, 49M25, 49M29, 65M12, 65M60}} \\[2mm]
{\small {\bf Keywords: Optimal control, semilinear PDE, uniqueness of global solutions, Gagliardo--Nirenberg inequality} }

\section{Introduction}
Let us consider the following optimal control problem
\[
(\mathbb{P}) \quad
 \min_{u \in U_{ad}} J(u)=\frac{1}{2} \Vert y-y_0 \Vert_{L^2(\Omega)}^2  + \dfrac{\alpha}{2} \Vert u \Vert_{L^2(\Omega)} ^2 
\]
subject to the semilinear elliptic PDE
\begin{equation}
\label{equation: 1}
 \begin{aligned}
-\Delta y +\phi(y) &=u  \quad \mbox{ in }  \Omega \subset \mathbb{R}^2, \\
y &=0  \quad \mbox{ on } \partial\Omega,
\end{aligned}
\end{equation}
and the pointwise state constraints
\begin{equation} \label{equation: 1a}
y_a(x) \leq y(x) \leq y_b(x), \quad x \in K \subset \Omega.
\end{equation}
We will formulate the precise assumptions on the data of the problem in Section 2. Since the state equation is in general nonlinear, the control
problem is nonconvex and there may be several solutions of the necessary first order conditions. These can be examined further with the help of
second order conditions but those will only give local information and usually do not allow a decision on whether the given point is
a {\it global} minimum of $(\mathbb{P})$. It is exactly this question which is the starting point of our work. Assuming that we have  an 
admissible control $\bar{u}$ which satisfies
the necessary first order conditions we will formulate a condition on the adjoint variable that guarantees that $\bar{u}$ is a solution
of the control problem $(\mathbb{P})$. This condition requires a certain $L^q$--norm to be bounded by a constant that only depends on the data and that is known
explicitly. While this approach is only of limited use at the continuous level, the situation is different when we apply our methods to a 
suitable discretisation $(\mathbb{P}_h)$ of  $(\mathbb{P})$. It turns out that we can obtain an analogous result for a discrete stationary point 
$\bar{u}_h$ and the corresponding discrete adjoint state. But since now  the discrete adjoint is available to us as a result of a
numerical computation, we can check whether our condition is satisfied. If the answer is yes, $\bar{u}_h$ is a global minimum
of $(\mathbb{P}_h)$. Moreover we are able to make the connection back to the original control problem in that we
 show that a sequence of solutions of $(\mathbb{P}_h)$, that satisfy our condition
uniformly in $h$ converge to a global solution of $(\mathbb{P})$.

To the best of the author's knowledge this is the first contribution to the study of uniqueness of solutions to semilinear elliptic optimal control problems. However, concerning the analysis, numerical treatment and implementation of semilinear optimal control problems many contributions can be found in the literature. Here we exemplarily mention the work \cite{ACT02} of Arada et al., \cite{C02} of Casas, and the work \cite{neitzel2014finite} of Neitzel et al. Further references can be found in \cite[Chapter 3]{HPUU08}, \cite{HR11}, and in \cite{CT14}, where the role of second order conditions in PDE constrained optimization is discussed.

\section{The optimal control problem  $(\mathbb{P})$}
Let $\Omega \subset \mathbb{R}^2$ be a bounded, convex and polygonal domain. We assume that $\phi:\mathbb{R}\to\mathbb{R}$ is of class $C^2$ and
monotonically increasing. 
For our analysis  we require the following structural assumption on $\phi$: \\[2mm]
{\bf Assumption 1:} There exist $r>1$ and $M \geq 0$ such that
\begin{equation} \label{assumption: 1}
\displaystyle
|\phi''(s)| \leq M \phi'(s)^{\frac{1}{r}} \quad \mbox{ for all } s \in \mathbb{R}.
\end{equation}
Let us remark for later purposes that (\ref{assumption: 1}) implies
\begin{eqnarray}
\phi'(s) & \leq & c_1 \bigl( 1+ | s |^{r_1} \bigr), \; s \in \mathbb{R}, \quad  r_1 = \frac{r}{r-1}  \label{phi1} \\
| \phi(s)|  & \leq & c_0  \bigl( 1+ | s |^{r_0} \bigr), \; s  \in \mathbb{R}, \quad r_0 = \frac{2r-1}{r-1}. \label{phi0} 
\end{eqnarray}
Note that for a power nonlinearity of the form $\phi(s)= | s |^{q-2} s$ ($q>3$) 
\[
| \phi''(s) | = (q-1)(q-2) | s |^{q-3} = (q-2)(q-1)^{\frac{1}{q-2}} [\phi'(s)]^{\frac{q-3}{q-2}},
\]
so that (\ref{assumption: 1}) is satisfied if we chose $r=\tfrac{q-2}{q-3}$. Solving this relation for $q$ yields $q = \tfrac{3r-2}{r-1}$, which
is an expression that we will encounter in our analysis below. \\[2mm]
Using the theory of monotone operators one can show that for every $u \in L^2(\Omega)$
the boundary value problem (\ref{equation: 1}) has a unique solution $y \in H^1_0(\Omega)\cap H^2(\Omega)$ which we
denote by $y=\mathcal G(u)$. Moreover, there exists a constant $c >0$ such that
\begin{equation*} 
\| y\|_{H^2(\Omega)} \leq c \bigl( 1+ \| u\|_{L^2(\Omega)} \bigr).
\end{equation*}
Next, suppose that $K$ is a (possibly empty) compact subset of $\Omega$ and define
\[
Y_{ad}:=\{ z \in C(K): y_a(x)\leq z(x)\leq y_b(x) \mbox{ for all }  x \in K \}.
\]
Here, $y_a,y_b \in C_0(\Omega)$ are given functions that satisfy $y_a(x) < y_b(x), x \in K$. \\[2mm]
We consider the semilinear optimal control problem
\[
(\mathbb{P}) \quad
\begin{array}{l}
 \min_{u \in U_{ad}} J(u):=\frac{1}{2} \Vert y-y_0 \Vert_{L^2(\Omega)}^2  + \dfrac{\alpha}{2} \Vert u \Vert_{L^2(\Omega)} ^2 \\
\mbox{ subject to } y= \mathcal G(u)  \mbox{ and } y_{|K} \in Y_{ad},
\end{array}
\]
where
\[
U_{ad}:=\{ u \in L^2(\Omega) : u_a \leq u(x) \leq u_b \mbox{ a.e. in } \Omega \}
\]
and $- \infty \leq u_a \leq u_b \leq \infty$ are given.
By classical arguments Problem~$(\mathbb{P})$ has a solution $\bar u \in U_{ad}$.

\begin{remark}
We note that the choice $K=\bar \Omega$ also is possible, if the bounds satisfy the compatibility condition $y_a < 0 < y_b$ on $\partial \Omega$, which only requires minor modifications in the analysis.
\end{remark}

\section{The variational discretization of $(\mathbb{P})$ }
In this section we approximate Problem~$(\mathbb{P})$ using the variational discretization approach introduced in \cite{hinze2005variational}. To this end, let $\mathcal{T}_h$ be an admissible triangulation of $\Omega$ so that 
\[
\bar\Omega= \bigcup_{T \in \mathcal{T}_h} \bar T. 
\]
We define the space of linear finite elements, 
\[
X_{h0}:= \{ v_h \in C(\bar\Omega) : v_h \mbox{ is a linear polynomial on each }  T \in \mathcal{T}_h \mbox{ and } {v_h}_{ | \partial\Omega}=0   \}
\]
and approximate \eqref{equation: 1} as follows: for a given $u \in L^2(\Omega)$, find $y_h \in  X_{h0}$ such that
\begin{equation}
\label{equation: 3}
 \int_\Omega \nabla y_h   \cdot  \nabla v_h +  \phi(y_h) v_h \, dx = \int_\Omega u v_h \,dx \quad \forall \, v_h \in X_{h0}.
\end{equation}
Using a fixed point argument one can show that \eqref{equation: 3} has a unique solution $y_h \in X_{h0}$ which we denote by
$\mathcal{G}_h(u)$. Finally, let us define 
\begin{displaymath}
\mathcal N_h := \lbrace x_j \, | \, x_j \mbox{ is a vertex of } T \in \mathcal T_h, \mbox{ where } T \cap K \neq \emptyset \rbrace.
\end{displaymath}
The variational discretization of Problem~$(\mathbb{P})$ now reads:
\[
(\mathbb{P}_h) \quad
\begin{array}{l}
 \min_{u \in U_{ad}} J_h(u):=\frac{1}{2} \Vert y_h-y_0 \Vert_{L^2(\Omega)}^2  + \dfrac{\alpha}{2} \Vert u \Vert_{L^2(\Omega)} ^2 \\
\mbox{ subject to } y_h = \mathcal G_h(u),  (y_h(x_j))_{x_j \in \mathcal N_h} \in Y^h_{ad}, 
\end{array}
\]
where 
\[
Y^h_{ad}:=\{ (z_j)_{x_j \in \mathcal N_h} \, | \, y_a(x_j) \leq z_j \leq y_b(x_j), x_j \in \mathcal N_h \}.
\]
We note that Problem~$(\mathbb{P}_h)$ is still an infinite dimensional optimization problem since the controls are sought in $U_{ad}$. If a feasible point for $(\mathbb{P}_h)$ exists, standard techniques yield the existence of a solution $\bar u_h \in U_{ad}$ for Problem~$(\mathbb{P}_h)$. The typical
approach in order to find an optimum of $(\mathbb{P}_h)$ consists in trying to determine solutions of the necessary first order conditions.
A formal analysis shows  that these conditions read in our case: there exist multipliers $\bar{p}_h \in X_{h0}$ and
$\bar{\mu}_j \in \mathbb{R}, x_j \in \mathcal N_h$ such that

\begin{eqnarray}
 \int_\Omega \nabla \bar y_h   \cdot  \nabla v_h +  \phi( \bar y_h) v_h \, dx & = & \int_\Omega \bar u_h v_h \,dx \qquad \forall \, v_h \in X_{h0}, \label{equation: 4}  \\
\int_\Omega \nabla \bar p_h \cdot \nabla v_h + \phi'(\bar y_h)\bar p_h v_h \, dx 
& = &  \nonumber \\
\lefteqn{ \hspace{-3cm}  \int_\Omega (\bar y_h -y_0) v_h \, dx  + \sum_{x_j \in \mathcal N_h} \bar \mu_j v_h(x_j)  \qquad \forall \, v_h \in X_{h0}, } 
\label{equation: 2}  \\
\int_\Omega (\bar p_h + \alpha \bar u_h)(u-\bar u_h) \, dx &\geq & 0 \qquad \forall \, u \in U_{ad}, \label{inequality: 1}\\
\sum_{x_j \in \mathcal N_h} \bar \mu_j (z_j-\bar y_h(x_j)) & \leq & 0 \qquad \forall \, (z_j)_{x_j \in \mathcal N_h} \in Y^h_{ad}.\label{inequality: 2}
\end{eqnarray}
Note that condition (\ref{inequality: 1}) is equivalent to the relation $\bar u_h = P_{[u_a,u_b]} \bigl(-\frac{\bar p_h}{\alpha} \bigr)$, so that
the control variable is implicitly discretized and (\ref{equation: 4})--(\ref{inequality: 2}) amounts to solving a nonlinear finite-dimensional system.
In order to state our main result of this section we introduce the following constant:

\begin{equation}  \label{etadef}
\displaystyle
\eta(\alpha,r):= \alpha^{\frac{\rho}{2}}  C^{\frac{2-2r}{r}}_q  M^{-1} \big(\frac{r-1}{2r-1} \big)^{\frac{1-r}{r}} 
q^{1/q} r^{1/r} \rho^{\rho/2} (2 - \rho)^{\frac{\rho}{2}-1}.
\end{equation}
Here, $q:=\tfrac{3r-2}{r-1}, \, \rho:=\frac{r+q}{rq}$, while $M$ and $r$ appear in  (\ref{assumption: 1}). Furthermore, $C_q$ is  an upper bound on the 
optimal constant in the Gagliardo-Nirenberg inequality $\Vert f \Vert_{L^q} \leq C \Vert f \Vert_{L^2}^{\frac{2}{q}} \Vert \nabla f \Vert_{L^2}^{\frac{q-2}{q}}$ ($q \geq 2)$.
For our purposes it will
be important  to specify a constant $C_q$ that is as sharp as possible. Lemma \ref{lemma: 2} in the Appendix will give three such bounds, two of which
can be found in the literature, while the third is new to the best of our knowledge. Let us now formulate the main result of this section.

\begin{theorem}
\label{theorem: 2}
Suppose that $\bar u_h \in U_{ad}, \bar y_h \in X_{h0}, \bar p_h \in X_{h0}, (\bar \mu_j)_{x_j \in \mathcal N_h}$ is a solution of 
(\ref{equation: 4})-(\ref{inequality: 2}).
If
\begin{align}
\label{inequality: 18}
\|\bar p_h\|_{L^q(\Omega)} \leq \eta(\alpha,r), 
\end{align}
then $\bar u_h$ is a global minimum for Problem~$(\mathbb{P}_h)$. If the inequality (\ref{inequality: 18}) is strict, then
$\bar u_h$ is the unique global minimum.
\end{theorem}

\begin{proof}
Let $u_h \in U_{ad}$ be a feasible control,  $y_h=\mathcal{G}_h(u_h)$ the associated state with 
$(y_h(x_j))_{x_j \in \mathcal N_h}  \in Y^h_{ad}$.   We have
\begin{align*}
J_h(u_h)-J_h(\bar u_h) &= \frac{1}{2} \| y_h-\bar y_h \|_{L^2(\Omega)}^2 + \dfrac{\alpha}{2} \| u_h -\bar u_h\|_{L^2(\Omega)} ^2  + \alpha \int_\Omega \bar u_h (u_h-\bar u_h) \, dx \\ 
& \quad  + \int_\Omega (\bar y_h-y_0)(y_h-\bar y_h) \, dx =: (A) 
\end{align*}
%
Using $v_h:= y_h-\bar y_h$ in \eqref{equation: 2} we get
\begin{align*}
(A) &= \frac{1}{2} \| y_h-\bar y_h \|_{L^2(\Omega)}^2 + \dfrac{\alpha}{2} \| u_h -\bar u_h\|_{L^2(\Omega)} ^2 + \alpha \int_\Omega \bar u_h (u_h-\bar u_h) \, dx \\
& \quad  +\int_\Omega \nabla \bar p_h \cdot \nabla (y_h-\bar y_h) + \phi'(\bar y_h)\bar p_h (y_h-\bar y_h) \, dx - \sum_{x_j \in \mathcal N_h}
\bar \mu_j (y_h(x_j) - \bar y_h(x_j))  \\
& \geq \frac{1}{2} \| y_h-\bar y_h \|_{L^2(\Omega)}^2 + \dfrac{\alpha}{2} \| u_h -\bar u_h\|_{L^2(\Omega)} ^2 + \alpha \int_\Omega \bar u_h (u_h-\bar u_h) \, dx \\
& \quad  +\int_\Omega \nabla \bar p_h \cdot \nabla (y_h-\bar y_h) + \phi'(\bar y_h)\bar p_h (y_h-\bar y_h) \, dx,   \numberthis  \label{side calculation: 1}\\ 
\end{align*}
%
by \eqref{inequality: 2}. Using (\ref{equation: 3}) for  $y_h$ and $\bar y_h$ with test function $\bar p_h$  we get
\begin{eqnarray*}
\lefteqn{
\int_\Omega \nabla \bar p_h \cdot \nabla (y_h-\bar y_h)\, dx  = \int_\Omega (u_h-\bar u_h ) \bar p_h  \, dx - \int_\Omega \big( \phi(y_h)-\phi(\bar y_h)   \big) \bar p_h \, dx } \\
&=& \int_\Omega (u_h-\bar u_h ) \bar p_h  \, dx - \int_\Omega \bar p_h (y_h-\bar y_h) \int_0^1 \phi'\big(t y_h+ (1-t)\bar y_h \big) \, dt \, dx.
\end{eqnarray*}
%
Using this in \eqref{side calculation: 1} and recalling \eqref{inequality: 1} we arrive at
\begin{eqnarray}
\lefteqn{
J_h(u_h)-J_h(\bar u_h)  }  \label{side calculation: 4} \\
& \geq &  \frac{1}{2} \| y_h-\bar y_h \|_{L^2(\Omega)}^2 + \dfrac{\alpha}{2} \| u_h -\bar u_h\|_{L^2(\Omega)} ^2 + \int_\Omega  (\alpha \bar u_h
+ \bar p_h) (u_h-\bar u_h) \, dx \nonumber  \\
& & - \int_\Omega \bar p_h (y_h-\bar y_h) \int_0^1 \phi'\big(t y_h+ (1-t)\bar y_h \big) - \phi'(\bar y_h) \, dt \, dx \nonumber  \\
& \geq &  \frac{1}{2} \| y_h-\bar y_h \|_{L^2(\Omega)}^2 + \dfrac{\alpha}{2} \| u_h -\bar u_h\|_{L^2(\Omega)} ^2 - R_h(u_h), \nonumber 
\end{eqnarray}
where  
\begin{equation*}
 R_h(u_h) := \int_\Omega \bar p_h (y_h-\bar y_h) \int_0^1 \phi'\big(t y_h+ (1-t)\bar y_h \big)-\phi'(\bar y_h) \, dt \, dx. 
\end{equation*}
The aim is now to estimate $ R_h(u_h) $. To begin, Lemma~\ref{lemma: 1} implies that
\begin{align*}
|  R_h(u_h)  | &\leq L_r \int_\Omega |\bar p_h | |y_h-\bar y_h|^2  \Big( \int_0^1 \phi'\big(t y_h+ (1-t)\bar y_h \big) \, dt \Big)^{\frac{1}{r}} \, dx \\
%
& = L_r \int_\Omega |\bar p_h | |y_h-\bar y_h|^{\frac{2r-2}{r}}  \Big( \int_0^1 \phi'\big(t y_h+ (1-t)\bar y_h \big) \, dt |y_h-\bar y_h|^2 \Big)^{\frac{1}{r}} \, dx, \end{align*}
where $L_r = M \bigl( \frac{r-1}{2r-1} \bigr)^{(r-1)/r}$. 
Next, H\"{o}lder's inequality with exponents 
\begin{displaymath}
q=\frac{3r-2}{r-1}, \;  \frac{r(3r-2)}{2(r-1)^2}=  \frac{qr}{2r-2} \mbox{  and }  r
\end{displaymath}
together with Lemma \ref{lemma: 2} yield
\begin{eqnarray*}
\lefteqn{  \hspace{-1.7cm} | R_h(u_h) |  \leq L_r \|\bar p_h\|_{L^q(\Omega)}  \| y_h- \bar y_h\|^{\frac{2r-2}{r}}_{L^q(\Omega)}  
\bigg ( \int_\Omega \int_0^1 \phi'\big(t y_h+ (1-t)\bar y_h \big) \, dt |y_h-\bar y_h|^2 \, dx \bigg)^{\frac{1}{r}} } \\
& \leq &  L_r C^{\frac{2r-2}{r}}_q \|\bar p_h\|_{L^q(\Omega)} \| y_h- \bar y_h\|^{\frac{4r-4}{q r}}_{L^2(\Omega)}  \| \nabla( y_h- \bar y_h) \|^{\frac{2}{q}}_{L^2(\Omega)} \\
& &  \times \bigg ( \int_\Omega \int_0^1 \phi'\big(t y_h+ (1-t)\bar y_h \big) \, dt |y_h-\bar y_h|^2 \, dx \bigg)^{\frac{1}{r}}.
\end{eqnarray*}
Here we also made use of the relation $\frac{2r-2}{r} (1-\frac{2}{q}) = \frac{2}{q}$.
Applying Lemma \ref{lemma:0} with
\[
a:= \int_\Omega | \nabla( y_h- \bar y_h) |^2 \, dx, \;
b:= \int_\Omega \int_0^1 \phi'\big(t y_h+ (1-t)\bar y_h \big) \, dt |y_h-\bar y_h|^2 \, dx, \quad  \lambda:= \frac{1}{q}, \mu:=\frac{1}{r}
\]
we obtain 
\begin{eqnarray}
|  R_h(u_h)  | &\leq & L_r C^{\frac{2r-2}{r}}_q d_r \|\bar p_h\|_{L^q(\Omega)} \| y_h- \bar y_h\|^{\frac{4r-4}{q r}}_{L^2(\Omega)} \label{side calculation: 2}\\
&  &\times \bigg( \int_\Omega | \nabla( y_h- \bar y_h) |^2 \, dx +   \int_\Omega \int_0^1 \phi'\big(t y_h+ (1-t)\bar y_h \big) \, dt |y_h-\bar y_h|^2 \, dx    \bigg)^{\rho}, \nonumber
\end{eqnarray}
where 
\[
d_r = q^{-1/q}  r^{-1/r} \rho^{-\rho}, \quad \rho= \frac{r+q}{rq}.
\]
Using again (\ref{equation: 3}) for  $y_h, \bar y_h$, this time with test function $y_h-\bar y_h$ yields 
\begin{align*}
\int_\Omega | \nabla( y_h- \bar y_h) |^2 \, dx +   \int_\Omega \int_0^1 \phi'\big(t y_h+ (1-t)\bar y_h \big) \, dt |y_h-\bar y_h|^2 \, dx   \\
\leq \| u_h- \bar u_h\|_{L^2(\Omega)} \| y_h- \bar y_h\|_{L^2(\Omega)}. 
\end{align*}
Inserting this estimate into \eqref{side calculation: 2} and observing that $\tfrac{4r-4}{qr} + \rho =2-\rho$ we deduce
\begin{align*}
|  R_h(u_h)  | & \leq L_r C^{\frac{2r-2}{r}}_q d_r \|\bar p_h\|_{L^q(\Omega)} \| y_h- \bar y_h\|^{2-\rho}_{L^2(\Omega)}  \| u_h- \bar u_h\|^{\rho}_{L^2(\Omega)} \\
& = 2 \alpha^{-\frac{\rho}{2}} L_r C^{\frac{2r-2}{r}}_q d_r \|\bar p_h\|_{L^q(\Omega)} 
 \Big( \frac{1}{2} \| y_h- \bar y_h \|_{L^2(\Omega)}^2  \Big)^{1-\frac{\rho}{2}}   \Big( \frac{\alpha}{2} \| u_h- \bar u_h \|_{L^2(\Omega)}^2  \Big)^{\frac{\rho}{2}}.  
\end{align*}
Applying again Lemma \ref{lemma:0},  this time with the choices 
\[
a := \frac{1}{2} \| y_h- \bar y_h \|_{L^2(\Omega)}^2, \; b := \frac{\alpha}{2} \| u_h- \bar u_h \|_{L^2(\Omega)}^2,
\quad  \lambda := 1-\frac{\rho}{2},\;  \mu:= \frac{\rho}{2},
\]
we obtain 
\begin{equation}  \label{side calculation: 3}
\displaystyle
|  R_h(u_h)  |   \leq 2 \alpha^{-\frac{\rho}{2}} L_r C^{\frac{2r-2}{r}}_q d_r e_r  \|\bar p_h\|_{L^q(\Omega)} 
 \Big(  \frac{1}{2} \| y_h- \bar y_h \|_{L^2(\Omega)}^2 +   \frac{\alpha}{2} \| u_h- \bar u_h \|_{L^2(\Omega)}^2   \Big),   
\end{equation}
where 
\[
e_r = \bigl( 1 - \frac{\rho}{2} \bigr)^{1- \frac{\rho}{2}} \bigl( \frac{\rho}{2} \bigr)^{\frac{\rho}{2}}.
\]
Using \eqref{side calculation: 3} in \eqref{side calculation: 4} we get
\begin{eqnarray*}
\lefteqn{ J_h(u_h)-J_h(\bar u_h) } \\
&\geq &  \Big(  \frac{1}{2} \| y_h- \bar y_h \|_{L^2(\Omega)}^2 +   \frac{\alpha}{2} \| u_h- \bar u_h \|_{L^2(\Omega)}^2   \Big)  \Big(  1-2 \alpha^{-\frac{\rho}{2}} L_r C^{\frac{2r-2}{r}}_q d_r e_r  \|\bar p_h\|_{L^q(\Omega)} \Big)   
\end{eqnarray*}
so that $J_h(u_h) \geq J_h(\bar u_h)$ provided that
\begin{align}
\label{side calculation: 5}
\|\bar p_h\|_{L^q(\Omega)} \leq \big( 2 \alpha^{-\frac{\rho}{2}} L_r C^{\frac{2r-2}{r}}_q d_r e_r \big)^{-1}.
\end{align} 
By direct calculations, we have
\begin{displaymath}
2 d_r e_r  = q^{-1/q} r^{-1/r} \rho^{- \rho/2} (2 - \rho)^{1-\frac{\rho}{2}}.
\end{displaymath}
Hence, using the above result and the value of $L_r$ from Lemma~\ref{lemma: 1} we can rewrite \eqref{side calculation: 5} as
\begin{align*}
\|\bar p_h\|_{L^q(\Omega)} &\leq \alpha^{\frac{\rho}{2}}  C^{\frac{2-2r}{r}}_q  M^{-1} \big(\frac{r-1}{2r-1} \big)^{\frac{1-r}{r}} 
q^{1/q} r^{1/r} \rho^{ \rho/2} (2 - \rho)^{\frac{\rho}{2}-1}
\end{align*}
which is the desired result.
\end{proof}

\noindent
Since  the adjoint state $\bar p_h$ and the quantity $\eta(\alpha,r)$ can be computed {\it explicitly},  Theorem \ref{theorem: 2} allows
us to decide whether a function $\bar u_h$ which satisfies the necessary conditions of first order is a global minimum of 
$(\mathbb{P}_h)$. A natural question then is, whether a sequence $(\bar u_h)_{0<h \leq h_0}$ of minima satisfying
(\ref{inequality: 18}) uniformly in $h$ converges to a global minimum of  $(\mathbb{P})$. We shall address this problem in the
next section and  it will be useful to have a continuous analogue of Theorem~\ref{theorem: 2}.  A function $\bar u \in U_{ad}$
satisfies the necessary first order conditions for problem $(\mathbb{P})$ if there exist $\bar p \in L^2(\Omega)$ and a 
measure $\bar \mu \in \mathcal M(K)$ such that 
\begin{eqnarray}
 \int_\Omega \nabla \bar y   \cdot  \nabla v +  \phi( \bar y) v \, dx & = & \int_\Omega \bar u v \,dx \qquad \forall \, v \in H^1_0(\Omega), \label{equation: 4c}  \\
\int_\Omega  \bar p (- \Delta v)  + \phi'(\bar y)\bar p v \, dx 
& = &  \nonumber \\
\lefteqn{ \hspace{-2cm}  \int_\Omega (\bar y -y_0) v \, dx  +   \int_K v \, d\bar\mu  \quad \forall v \in 
H^1_0(\Omega) \cap H^2(\Omega), } 
\label{equation: 2c}  \\
\int_\Omega (\bar p + \alpha \bar u)(u-\bar u) \, dx &\geq & 0 \qquad \forall \, u \in U_{ad}, \label{inequality: 1c}\\
\int_K (z-\bar y)\, d\bar\mu & \leq & 0 \qquad \forall \, z \in Y_{ad}. \label{inequality: 2c}
\end{eqnarray}
It is well--known that the function $\bar p$ then belongs to $W^{1,s}_0(\Omega)$ for any $1<s<2$ and hence to
$L^q(\Omega)$ for any $q<\infty$ (recall that $\Omega \subset \mathbb{R}^2$). Arguing in almost the same way as in the proof of Theorem \ref{theorem: 2} we obtain:

\begin{theorem}
\label{theorem: 5}
Suppose that $\bar u \in U_{ad}, \bar y \in H^1_0(\Omega), \bar p \in L^2(\Omega), \bar \mu \in \mathcal M(K)$ is a solution of 
(\ref{equation: 4c})-(\ref{inequality: 2c}).
If
\begin{align}
\label{18strich}
\|\bar p \|_{L^q(\Omega)} \leq \eta(\alpha,r), 
\end{align}
then $\bar u$ is a global minimum for Problem~$(\mathbb{P})$. If the inequality (\ref{18strich}) is strict, then
$\bar u$ is the unique global minimum.
\end{theorem}

\section{Convergence analysis}

Let $(\mathcal T_h)_{0 < h \leq h_0}$ be a quasiuniform sequence of triangulations of $\bar{\Omega}$. We consider the corresponding
sequence of control problems $(\mathbb{P}_h)$ and suppose that $\bar u_h \in U_{ad}$ satisfies the hypotheses of Theorem~\ref{theorem: 2}
uniformly in $0< h \leq h_0$. Thus there exist $\bar{p}_h \in X_{h0}$ and
$(\bar{\mu}_j)_{x_j \in \mathcal N_h}$ satisfying (\ref{equation: 4})-(\ref{inequality: 2}) as well as
\begin{align}
\label{inequality: 18u}
\|\bar p_h\|_{L^q(\Omega)} \leq \eta(\alpha,r), \quad 0< h \leq h_0.
\end{align}
It is convenient to introduce the measure $\bar \mu_h \in \mathcal M(\Omega)$ by
\begin{displaymath}
\bar \mu_h:= \sum_{x_j \in \mathcal N_h} \mu_j \delta_{x_j},
\end{displaymath}
where $\delta_{x_j}$ is the Dirac measure at $x_j$. Since $K \subset \Omega, \mbox{dist}(x_j,K) \leq h, x_j \in \mathcal N_h$ and
$y_a(x) < y_b(x), x \in K$ there exists a compact set $\tilde{K} \subset \Omega$, $\delta>0$ and $0< h_1 \leq h_0$ such that
$K \subset \tilde{K}$ and
\begin{eqnarray}
\mathcal N_h \subset \tilde{K}, & & 0 < h \leq h_1, \label{supportmu} \\
y_a(x) + \delta \leq \frac{1}{2} ( y_a(x) + y_b(x)) \leq y_b(x) -\delta, & & x \in \tilde{K}.
\end{eqnarray}
Applying a smoothing procedure to $x \mapsto \frac{1}{2}( y_a+y_b) \in C_0(\Omega)$ we obtain a function $w \in C^{\infty}_0(\Omega)$
such that
\begin{displaymath}
y_a(x) + \frac{\delta}{2} \leq w(x) \leq y_b(x) - \frac{\delta}{2}, \quad x \in \tilde{K}.
\end{displaymath}
Let us denote by  $R_h: H^1_0(\Omega) \rightarrow X_{h0}$ the Ritz projection defined by
\begin{equation} \label{ritz}
\displaystyle 
\int_{\Omega} \nabla R_h w \cdot \nabla v_h dx = \int_{\Omega} \nabla w \cdot \nabla v_h dx \qquad \forall v_h \in X_{h0}.
\end{equation}
Since $R_h w \rightarrow w$ uniformly in $\bar{\Omega}$, we may assume after choosing $h_1$ smaller if necessary that
\begin{equation} \label{w1}
\displaystyle
y_a(x) + \frac{\delta}{4} \leq R_h w(x) \leq y_b(x) - \frac{\delta}{4}, \qquad x \in \tilde{K}.
\end{equation}
Our first step in the convergence analysis are uniform bounds on the optimal control $\bar u_h$ as well as on $\bar y_h=\mathcal G_h(\bar u_h)$ and
$\bar \mu_h$.

\begin{lemma}
\label{lemma: 7} 
Let $\bar u_h \in U_{ad}, \bar y_h, \bar p_h \in X_{h0}$ and $(\bar \mu_j)_{x_j \in \mathcal N_h}$ be a solution of (\ref{equation: 4})-(\ref{inequality: 2}).
Then there exists a constant $C>0$, which is independent of $h$, such that
\[
\|\bar u_h\|_{L^2(\Omega)},\|\bar y_h\|_{H^1(\Omega)} ,  \|\bar \mu_h\|_{\mathcal{M}(\tilde K)} \leq C. 
\]
\end{lemma}
\begin{proof}
To begin, fix a function $u_0 \in U_{ad}$. Inserting $u_0$ into (\ref{inequality: 1}) we infer
\begin{eqnarray*}
\alpha \Vert \bar u_h \Vert_{L^2(\Omega)}^2 & \leq &  \int_{\Omega} u_0 ( \alpha \bar u_h + \bar p_h) dx - \int_{\Omega} \bar u_h \bar p_h dx \\
& \leq & \Vert u_0 \Vert_{L^2(\Omega)} \bigl( \alpha \Vert \bar u_h \Vert_{L^2(\Omega)} + \Vert \bar p_h \Vert_{L^2(\Omega)} \bigr) + 
\Vert \bar u_h \Vert_{L^2(\Omega)} \Vert \bar p_h \Vert_{L^2(\Omega)}.
\end{eqnarray*}
Since $q=\frac{3r-2}{r-1} \geq 3$ we deduce with the help of (\ref{inequality: 18u})
\begin{displaymath}
\Vert \bar u_h \Vert_{L^2(\Omega)} \leq C \bigl( \Vert u_0 \Vert_{L^2(\Omega)}  + \Vert \bar p_h \Vert_{L^2(\Omega)} \bigr)
 \leq C \bigl( \Vert u_0 \Vert_{L^2(\Omega)}  + \Vert \bar p_h \Vert_{L^q(\Omega)} \bigr) \leq C.
\end{displaymath}
Testing (\ref{equation: 4}) with $\bar y_h$, using the monotonicity of $\phi$ and Poincar\'e's inequality we infer
\begin{equation} \label{h1bound}
\displaystyle
\|\bar y_h\|_{H^1(\Omega)} \leq C \bigl( 1+ \| \bar u_h\|_{L^2(\Omega)} \bigr)  \leq C.
\end{equation}
Furthermore, (\ref{phi1}), (\ref{phi0}) along with the continuous embedding $H^1(\Omega) \hookrightarrow L^t(\Omega)$ for all $1 \leq t < \infty$ yield
\begin{equation}  \label{nonlinearbound}
\displaystyle \Vert \phi(\bar y_h) \Vert_{L^2(\Omega)}, \; \Vert \phi'(\bar y_h) \Vert_{L^2(\Omega)} \leq C.
\end{equation}
In order to verify the uniform boundedness of $\|\bar\mu_h\|_{\mathcal{M}(\tilde{K})}$ we first observe that
(\ref{inequality: 2}) implies 
\begin{displaymath}
\bar y_h(x_j) =
\left\{
\begin{array}{ll}
y_b(x_j), & \mbox{ if } \bar \mu_j >0, \\
y_a(x_j), & \mbox{ if } \bar \mu_j <0.
\end{array}
\right.
\end{displaymath}
As a result we deduce with the help of (\ref{w1})
\begin{displaymath}
\frac{\delta}{4} \Vert \bar \mu_h \Vert_{\mathcal M(\tilde K)} = \frac{\delta}{4} \sum_{x_j \in \mathcal N_h} | \bar \mu_j |
\leq \sum_{x_j \in \mathcal N_h} \bar \mu_j \bigl( \bar y_h(x_j) - R_h w(x_j) \bigr).
\end{displaymath}
Using $v_h = \bar y_h - R_h w$ in (\ref{equation: 2}) we may continue
\begin{eqnarray}
\frac{\delta}{4} \Vert \bar \mu_h \Vert_{\mathcal M(\tilde K)}  & \leq &  \int_{\Omega} \nabla \bar p_h \cdot \nabla \bar y_h dx 
 - \int_{\Omega} \nabla \bar p_h \cdot \nabla R_h w dx \nonumber \\
& & + \int_{\Omega} \phi'(\bar y_h) \bar p_h (\bar y_h-R_h w) dx  - \int_{\Omega} (\bar y_h - y_0)( \bar y_h- R_h w) dx \nonumber  \\
& \equiv & \sum_{i=1}^4 S_i. \label{muest1}
\end{eqnarray}
If we let $v_h = \bar p_h$ in  (\ref{equation: 4}) we obtain with the help of (\ref{nonlinearbound}) and (\ref{inequality: 18u})
\begin{displaymath}
| S_1 |  = | \int_{\Omega} (\bar u_h - \phi(\bar y_h)) \bar p_h dx | \leq
\bigl( \Vert \bar u_h \Vert_{L^2(\Omega)} + \Vert \phi(\bar y_h) \Vert_{L^2(\Omega)} \bigr) \Vert \bar p_h \Vert_{L^2(\Omega)} \leq C.
\end{displaymath}
Next, the definition of the Ritz projection and integration by parts yields
\begin{displaymath}
S_2 =  - \int_{\Omega} \nabla \bar p_h \cdot \nabla  w dx = \int_{\Omega} \bar p_h \Delta w dx
\end{displaymath}
so that
\begin{displaymath}
| S_2 | \leq \Vert \bar p_h \Vert_{L^2(\Omega)} \Vert \Delta w \Vert_{L^2(\Omega)} \leq C.
\end{displaymath}
H\"older's inequality along with (\ref{nonlinearbound}), (\ref{inequality: 18u}) and  (\ref{h1bound}) implies that
\begin{displaymath}
| S_3 | \leq \Vert \phi'(\bar y_h) \Vert_{L^2(\Omega)} \Vert \bar p_h \Vert_{L^q(\Omega)} \Vert \bar y_h - R_h w \Vert_{L^{\frac{2q}{q-2}}(\Omega)}
\leq C \Vert \bar y_h - R_h w \Vert_{H^1(\Omega)} \leq C.
\end{displaymath}
Finally,
\begin{displaymath}
| S_4 | \leq \bigl( \Vert \bar y_h \Vert_{L^2(\Omega)} + \Vert y_0 \Vert_{L^2(\Omega)} \bigr) \bigl( \Vert \bar y_h \Vert_{L^2(\Omega)}
+ \Vert R_h w \Vert_{L^2(\Omega)} \bigr) \leq C.
\end{displaymath}
Inserting the above estimates into (\ref{muest1}) yields the bound on $\Vert \bar \mu_h \Vert_{\mathcal M(\tilde K)}$. 
\end{proof}


\noindent
Now, we are in position to formulate the main theorem in this section:
\begin{theorem}
Suppose that $(\bar u_h, \bar y_h, \bar p_h, \bar \mu_h)_{0<h \leq h_1}$ is a sequence 
satisfying (\ref{equation: 4})-(\ref{inequality: 2}) as well as (\ref{inequality: 18u}). Then
\[
\bar u_h \to \bar u \mbox{ in } L^2(\Omega) \mbox{ for a subsequence } h \to 0,
\]
where $\bar u$ is a global minimum for Problem~$(\mathbb{P})$. If 
\begin{equation} \label{strict}
\displaystyle
\|\bar p_h\|_{L^q(\Omega)} \le \kappa \eta(\alpha,r), \quad 0< h \leq h_1,
\end{equation}
for some $0 < \kappa <1$, then $\bar u$ is the unique global solution of $(\mathbb{P})$ and the whole sequence $( \bar u_h)_{0<h \leq h_1}$ converges 
to $\bar u$. 
\end{theorem}

\begin{proof}
From Lemma~\ref{lemma: 7}, we deduce the existence of a subsequence $h \rightarrow 0$ and $\bar u \in L^2(\Omega)$, $\bar y \in H^1_0(\Omega)$,
$\bar p \in L^q(\Omega)$, $\bar \mu \in \mathcal M(\tilde K)$ such that
\begin{eqnarray}
\bar u_h &\rightharpoonup & \bar u \quad \mbox{ in } L^2(\Omega), \label{uconv} \\
\bar y_h & \rightharpoonup & \bar y \quad \mbox{ in } H^1_0(\Omega) \mbox{ and } \bar y_h \rightarrow \bar y \mbox{ in } L^t(\Omega), 1 \leq t < \infty. 
\label{yconv} \\
\bar\mu_h &\rightharpoonup & \bar\mu \quad \mbox{ in } \mathcal{M}(\tilde K), \label{muconv} \\
\bar p_h & \rightharpoonup & \bar p \quad \mbox{ in } L^q(\Omega), \label{pconv} 
\end{eqnarray}
Our aim is to show that $(\bar u,\bar y, \bar p, \bar \mu)$ is a solution of (\ref{equation: 4c})-(\ref{inequality: 2c}).
It is easy to see that $\bar u \in U_{ad}$ and that $\bar y =\mathcal{G}(\bar u)$, so that (\ref{equation: 4c}) is satisfied. Furthermore, the fact that $\mbox{dist}(x_j,K) \leq h, x_j \in \mathcal N_h$ implies that
$\mbox{supp}(\bar \mu) \subset K$. Combining this with the bound $\Vert \bar \mu_h \Vert_{\mathcal M(\tilde K)} \leq C$ we infer that
\begin{equation} \label{dualconv}
\displaystyle \int_{\tilde K} z^h d \bar \mu_h \rightarrow \int_K z d \bar \mu \quad \mbox{ as } h \rightarrow 0 
\end{equation} 
for every sequence $(z^h)_{0 < h \leq h_1} \subset C(\tilde{K})$ converging uniformly to $z$ on $\tilde{K}$.
Next, we claim that
\begin{equation} \label{uniform}
\displaystyle
\bar y_h \rightarrow \bar y \qquad \mbox{ uniformly in } \bar{\Omega}.
\end{equation}
To see this, denote by $y^h \in H^2(\Omega) \cap H^1_0(\Omega)$ the solution of
\begin{displaymath}
- \Delta y^h = \bar u_h - \phi(\bar y_h) \; \mbox{ in } \Omega, \quad y^h = 0 \; \mbox{ on } \partial \Omega.
\end{displaymath}
We deduce from Lemma \ref{lemma: 7} and (\ref{nonlinearbound}) that  $(y^h)_{0<h \leq h_1}$ is bounded in $H^2(\Omega)$, so that there exists a further subsequence and a function $\hat{y} \in H^2(\Omega) \cap H^1_0(\Omega)$ with
\begin{displaymath}
y^h \rightharpoonup \hat{y} \mbox{ in }  H^2(\Omega), \quad y^h \rightarrow \hat{y} \mbox{ in } C(\bar{\Omega}).
\end{displaymath}
Since $\bar u_h - \phi(\bar y_h) \rightharpoonup \bar u - \phi(\bar y)$ in $L^2(\Omega)$ we find that $- \Delta \hat{y} = - \Delta \bar y$ a.e. in
$\Omega$. Hence $\hat{y} = \bar y$ and  $y^h \rightarrow \bar y$ in $C(\bar{\Omega})$. On the other hand, the definition of $y^h$ implies that $\bar y_h = R_h y^h$, so that
 standard interpolation and inverse estimates imply
\begin{eqnarray*}
\Vert \bar y_h - \bar y \Vert_{L^{\infty}(\Omega)} & \leq & \Vert R_h y^h - y^h \Vert_{L^{\infty}(\Omega)} +  \Vert y^h - \bar y \Vert_{L^{\infty}(\Omega)} \\
& \leq &  C h \Vert y^h \Vert_{H^2(\Omega)} + \Vert y^h - \bar y \Vert_{L^{\infty}(\Omega)} \rightarrow 0 \quad \mbox{ as } h \rightarrow 0,
\end{eqnarray*}
since $\Vert y^h \Vert_{H^2(\Omega)} \leq C$. This proves (\ref{uniform}). \\
Let us check that $\bar y_{|K} \in Y_{ad}$. For a fixed point $x \in K$ we can choose a sequence $(x_{j_h})_{0 <h \leq h_1}$ such that
$x_{j_h} \in \mathcal N_h$ and $| x_{j_h}-x| \leq h$. Since $y_a(x_{j_h}) \leq \bar y_h(x_{j_h}) \leq y_b(x_{j_h})$ we obtain 
$y_a(x) \leq \bar y(x) \leq y_b(x)$ by passing to the limit $h \rightarrow 0$ and
using (\ref{uniform}). \\
Next, let us fix  $z \in Y_{ad}$ and extend $z$ to a function $\tilde z \in C(\tilde K)$ satisfying $y_a(x) \leq \tilde z(x) \leq y_b(x), x \in \tilde K$. We obtain from (\ref{inequality: 2}), (\ref{dualconv}) and (\ref{uniform})
\begin{displaymath}
0 \geq \sum_{x_j \in \mathcal N_h} \bar \mu_j (\tilde z(x_j) - \bar y_h(x_j)) = \int_{\tilde K} (\tilde z - \bar y_h) d \bar \mu_h \rightarrow
\int_K (z - \bar y) d \bar \mu,
\end{displaymath}
which yields (\ref{inequality: 2c}). \\
In order to  derive (\ref{equation: 2c}) we fix $v \in H^2(\Omega)\cap H^1_0(\Omega)$ and insert $v_h = R_h v$ into (\ref{equation: 2}), i.e.
\begin{displaymath}
\int_\Omega \nabla \bar p_h \cdot \nabla R_h v + \phi'(\bar y_h)\bar p_h R_h v \, dx =  \int_\Omega (\bar y_h -y_0) R_h v \, dx  + 
\int_{\tilde K} R_h v \, d \bar \mu_h. 
\end{displaymath}
Using the definition of $R_h$ and integration by parts we may write
\begin{displaymath}
\int_\Omega \nabla \bar p_h \cdot \nabla R_h v \, dx = \int_{\Omega} \nabla \bar p_h \cdot \nabla v \, dx = \int_{\Omega} \bar p_h (- \Delta v) \, dx
\end{displaymath}
so that (\ref{equation: 2c}) follows from passing to the limit $h \rightarrow 0$ taking into account (\ref{pconv}), (\ref{uniform}) and (\ref{dualconv}). \\
Our next goal is to show that $\bar u_h \rightarrow \bar u$ in $L^2(\Omega)$. Inserting $\bar u$ into (\ref{inequality: 1}) and rearranging we infer
\begin{equation}  \label{uhbound}
\displaystyle
\alpha \Vert \bar u_h \Vert_{L^2(\Omega)}^2  \leq   \int_{\Omega} \bar u ( \alpha \bar u_h + \bar p_h) dx - \int_{\Omega} \bar u_h \bar p_h dx.
\end{equation}
The second integral can be rewritten with the help of (\ref{equation: 4}) and (\ref{equation: 2}), namely
\begin{eqnarray*}
\lefteqn{
\int_{\Omega} \bar u_h \bar p_h dx  =  \int_{\Omega} \nabla \bar y_h \cdot \nabla \bar p_h dx + \int_{\Omega} \phi(\bar y_h) \bar p_h dx }  \\
& = & - \int_{\Omega} \phi'(\bar y_h) \bar p_h \bar y_h dx + \int_{\Omega} (\bar y_h - y_0) \bar y_h dx + \int_{\tilde K} \bar y_h d \bar \mu_h  
+ \int_{\Omega} \phi(\bar y_h) \bar p_h dx.
\end{eqnarray*}
This relation allows us to pass to the limit in a similar way as above to give
\begin{eqnarray*}
\lefteqn{ \hspace{-2cm}  \int_{\Omega} \bar u_h \bar p_h dx 
 \rightarrow   - \int_{\Omega} \phi'(\bar y) \bar p \bar y dx + \int_{\Omega} (\bar y - y_0) \bar y dx + \int_K \bar y d \bar \mu  
+ \int_{\Omega} \phi(\bar y) \bar p dx  } \\
& = & \int_{\Omega} (- \Delta \bar y) \bar p dx  + \int_{\Omega} \phi(\bar y) \bar p dx = \int_{\Omega} \bar u \bar p dx,
\end{eqnarray*}
where we used (\ref{equation: 2c}) and the fact that $\bar y = \mathcal G(\bar u)$. We can now pass to the limit
in (\ref{uhbound}) and deduce that
\begin{displaymath}
\limsup_{h \rightarrow 0} \Vert \bar u_h \Vert_{L^2(\Omega)}^2 \leq \Vert \bar u \Vert_{L^2(\Omega)}^2.
\end{displaymath}
Since $\Vert \bar u \Vert_{L^2(\Omega)}^2 \leq \liminf_{h \rightarrow 0} \Vert \bar u_h \Vert_{L^2(\Omega)}^2$ we infer that
$\Vert \bar u_h \Vert_{L^2(\Omega)} \rightarrow \Vert \bar u \Vert_{L^2(\Omega)}$, 
which together with the fact that $\bar u_h \rightharpoonup \bar u$ in $L^2(\Omega)$ implies that 
$\bar u_h \rightarrow \bar u$ in $L^2(\Omega)$. \\
Combining this with the  weak  convergence $\bar p_h \rightharpoonup  \bar p$ in $L^2(\Omega)$, one can pass to the limit in \eqref{inequality: 1} to obtain 
\begin{equation}
\label{inequality: 17}
\int_\Omega (\bar p + \alpha \bar u)(u-\bar u) \, dx \geq 0 \quad \forall u \in U_{ad},
\end{equation} 
which is (\ref{inequality: 1c}).
In conclusion we see that $(\bar u, \bar y, \bar p, \bar \mu)$ is a solution of 
(\ref{equation: 4c})-(\ref{inequality: 2c}). Furthermore, the lower semicontinuity of the $L^q$--norm implies that 
\begin{displaymath}
\|\bar p \|_{L^q(\Omega)} \leq \liminf_{h \rightarrow 0} \|\bar p_h\|_{L^q(\Omega)} \le  \eta(\alpha,r)
\end{displaymath}
and we infer from Theorem~\ref{theorem: 5}, that $\bar u$ is a global minimum of Problem~$(\mathbb{P})$. If (\ref{strict}) holds,
then $\bar p$ satisfies $\Vert \bar p \Vert_{L^q(\Omega)} \leq \kappa \eta(\alpha,r) < \eta(\alpha,r)$ and $\bar u$ is the unique
global minimum of  $(\mathbb{P})$. A standard argument then shows that the whole sequence $(\bar u_h)_{0 < h \leq h_1}$ converges
to $\bar u$.
\end{proof}
\noindent
Before we go to the numerical examples, we make the following general remarks.
\begin{remark} \mbox{}

\noindent
1. We do not require a constraint qualification  such as a Slater condition to deduce that $(\bar u, \bar y, \bar p, \bar \mu)$ satisfies the system 
(\ref{equation: 4c})-(\ref{inequality: 2c}), which represents the first order necessary optimality conditions for Problem~$(\mathbb{P})$.\\[2mm]
2. 
It is well known that  (\ref{equation: 4})-(\ref{inequality: 2}) can be rewritten equivalently as a system of semi-smooth equations and thus can be solved by a semi-smooth Newton method, see for instance \cite{deckelnick2007finite}, \cite{hintermuller2002primal}, \cite{ulbrich2002semismooth}. In particular, we can
avoid the  use of relaxation methods such as Moreau-Yosida relaxation, interior point methods, or Lavrentiev-type regularization. \\[2mm]
3.
Since we solve (\ref{equation: 4})-(\ref{inequality: 2}) in practice on the computer, we consider $\bar u_h$ a global minimum if the inequality \eqref{inequality: 18} is satisfied up to machine precision. Here, the integral $\|\bar p_h\|_{L^q}$ on the left hand side of this inequality is assumed to be calculated exactly. However, this assumption can be achieved easily whenever $q$ is an integer because in this case the function $|\bar p_h|^q$ restricted to every triangle in the mesh is a  (possibly piecewise) polynomial of degree $q$. Hence, one can use an appropriate quadrature rule to evaluate such an integral exactly. \\[2mm]
\end{remark}

\section{Numerical Examples}
In this section we consider variational discretization of the optimal control problem $(\mathbb{P})$ for different choices of the nonlinearity $\phi$ and the data
$y_0,u_a,u_b,y_a,y_b,\alpha$, while $\Omega:=(0,1) \times (0,1)$ is kept fixed in all considered examples.
For the desired state $y_0$ we consider the following two choices
\begin{align*}
\textbf{A1 :} \quad y_0(x)&:= 2\sin(2 \pi x_1)\sin(2\pi x_2), \\
\textbf{A2 :} \quad y_0(x)&:=60+160(x_1(x_1-1)+x_2(x_2-1)).
\end{align*}
We note that in choice~\textbf{A1} the desired state $y_0$ vanishes on the boundary $\partial\Omega$ of the domain, while in choice~\textbf{A2} it doesn't, see Figure~\ref{figure: desired state choiceA1 and choiceA2}.  The numerical solution of the corresponding systems (\ref{equation: 4})-(\ref{inequality: 2}) is performed with the semismooth Newton method proposed in \cite{deckelnick2007finite}, whose extension to the treatment of finite element approximations of semilinear PDEs ist straightforward. All the computations are performed on a uniform triangulation of $\bar\Omega$ with mesh size $h=2^{-5}\sqrt{2}$.

\begin{figure}[p]
        \centering
        \begin{subfigure}[h!]{0.5\textwidth}
                \includegraphics[trim = 40mm 80mm 30mm 70mm, clip, width=\textwidth]{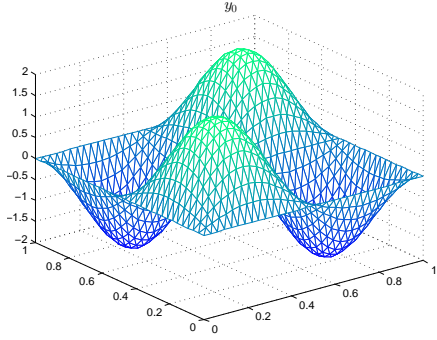}
                \caption{$y_0$ choice \textbf{A1}.}
        \end{subfigure}%
        ~ 
        \begin{subfigure}[h!]{0.5\textwidth}
                \includegraphics[trim = 40mm 80mm 30mm 70mm, clip, width=\textwidth]{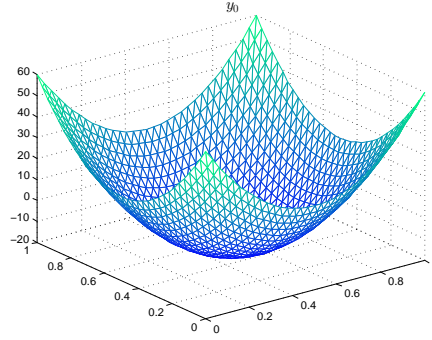}
                \caption{$y_0$ choice \textbf{A2}.}
        \end{subfigure}
        
        \caption{The desired state $y_0$ choices  \textbf{A1} and \textbf{A2}.}
        \label{figure: desired state choiceA1 and choiceA2}
\end{figure}

\begin{example}
\label{example: y power 3}
In this example we define $\phi(s):=s^3$. It is easy to see that this nonlinearity satisfies Assumption 1 with $r=2$ and $M=2\sqrt{3}$. Hence, in view of Theorem~\ref{theorem: 2} we have $q=4$ and a control $\bar u_h$ obtained from solving (\ref{equation: 4})-(\ref{inequality: 2}) is a global minimum if the associated adjoint state $\bar p_h$ satisfies 
\[
\|\bar p_h\|_{L^4(\Omega)} \leq  5^{-\frac{5}{8}} 3^{\frac{3}{8}} \sqrt{2} C^{-1}_4 \alpha^{\frac{3}{8}},
\]
where $C^{-1}_4 \approx 1.543145399297809$ is the constant from Lemma~\ref{lemma: 2}. For this example we consider the following three cases. Let us
abbreviate
\[
\eta(\alpha):= \eta(\alpha,2)= 5^{-\frac{5}{8}} 3^{\frac{3}{8}} \sqrt{2} C^{-1}_4 \alpha^{\frac{3}{8}}.
\]

\noindent
\textbf{Case~1} (unconstrained problem)
In this case we set
\begin{align*}
u_b &=-u_a=\infty, \\
y_b &=-y_a=\infty. \\
\end{align*}
In Table~\ref{table: example y3 case(unconstrained) choiceA1} we provide the values of $\|\bar p_h\|_{L^4} $, $\eta(\alpha)$ and $J(\bar u_h)$ for different values of $\alpha$ where  we consider the choice~\textbf{A1} for the desired state $y_0$. The findings are illustrated graphically in Figure~\ref{figure: example y3 case(unconstrained) choiceA1}. We see that for all values of $\alpha$ we can claim that $\bar u_h$ is a global minimum since $\|\bar p_h\|_{L^4} $ is less than its corresponding  $\eta(\alpha)$. On the other hand, if we consider the choice~\textbf{A2} for $y_0$ we can claim $\bar u_h$ is a global minimum only for approximately $\alpha$ greater than $10^{-2}$ as it  can be seen from Figure~\ref{figure: example y3 case(unconstrained) choiceA2}. The numerical values are provided in Table~\ref{table: example y3 case(unconstrained) choiceA2}.

\begin{table}[p]
\caption{Example~\ref{example: y power 3}  Case~1 with choice~\textbf{A1} for $y_0$: The values of  $\|\bar p_h\|_{L^4} $, $\eta(\alpha)$ and $J(\bar u_h)$ for different values of $\alpha$.}
\label{table: example y3 case(unconstrained) choiceA1}
\begin{tabular}{ l  c  c  c}
\toprule
$\alpha$ &	  $\|\bar p_h\|_{L^4}$&        $\eta(\alpha)$ &   $J(\bar u_h)$  \\
\midrule[1pt]

1.0e-06 &	  9.990654861172e-05 &	 	 6.776197632762e-03 &	 	 3.344560044987e-03   \\
1.0e-05 &	  9.328604940252e-04 &	 	 1.606889689070e-02 &	 	 3.128947575776e-02   \\
1.0e-04 &	  5.916313713912e-03 &	 	 3.810535956559e-02 &	 	 1.967337721757e-01   \\
1.0e-03 &	  1.322797500856e-02 &	 	 9.036204771862e-02 &	 	 4.320833160546e-01   \\
1.0e-02 &	  1.509224717529e-02 &	 	 2.142821839497e-01 &	 	 4.922544738762e-01   \\
1.0e-01 &	  1.530600543072e-02 &	 	 5.081431366100e-01 &	 	 4.992144829702e-01   \\
1.0e+00 &	  1.532768796263e-02 &	 	 1.204997272869e+00 &	 	 4.999213370332e-01   \\
1.0e+01 &	  1.532985932323e-02 &	 	 2.857498848277e+00 &	 	 4.999921325890e-01   \\
1.0e+02 &	  1.533007649041e-02 &	 	 6.776197632762e+00 &	 	 4.999992132478e-01   \\
1.0e+03 &	  1.533009820744e-02 &	 	 1.606889689070e+01 &	 	 4.999999213247e-01   \\

\bottomrule 
\end{tabular}
\end{table}

\begin{figure}[p]
        \centering
        \begin{subfigure}[h!]{0.5\textwidth}
                \includegraphics[trim = 40mm 80mm 30mm 70mm, clip, width=\textwidth]{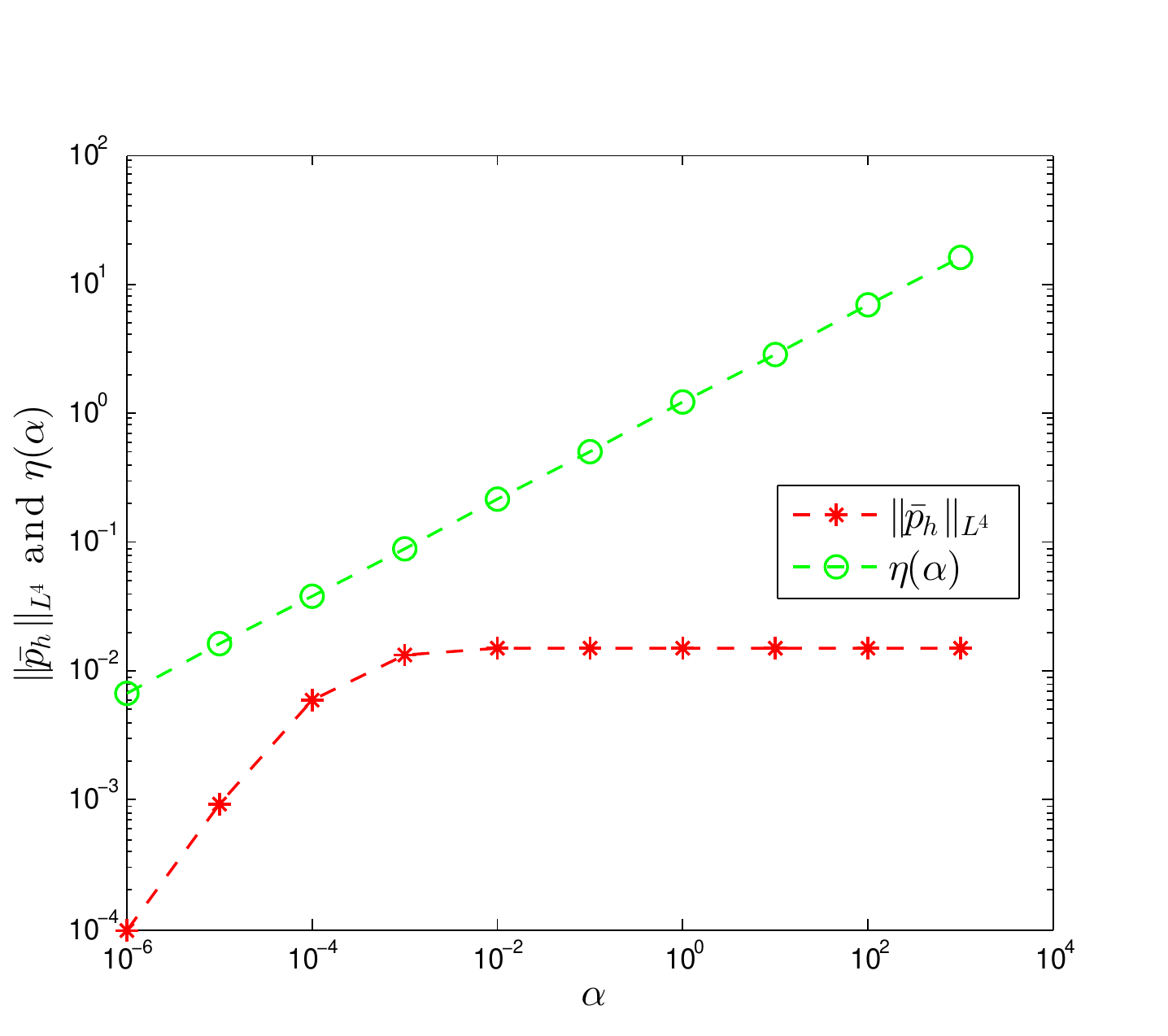}
                \caption{$\|\bar p_h\|_{L^4}$ and $\eta(\alpha)$ vs. $\alpha$.}
        \end{subfigure}%
        ~ 
        \begin{subfigure}[h!]{0.5\textwidth}
                \includegraphics[trim = 35mm 80mm 30mm 70mm, clip, width=\textwidth]{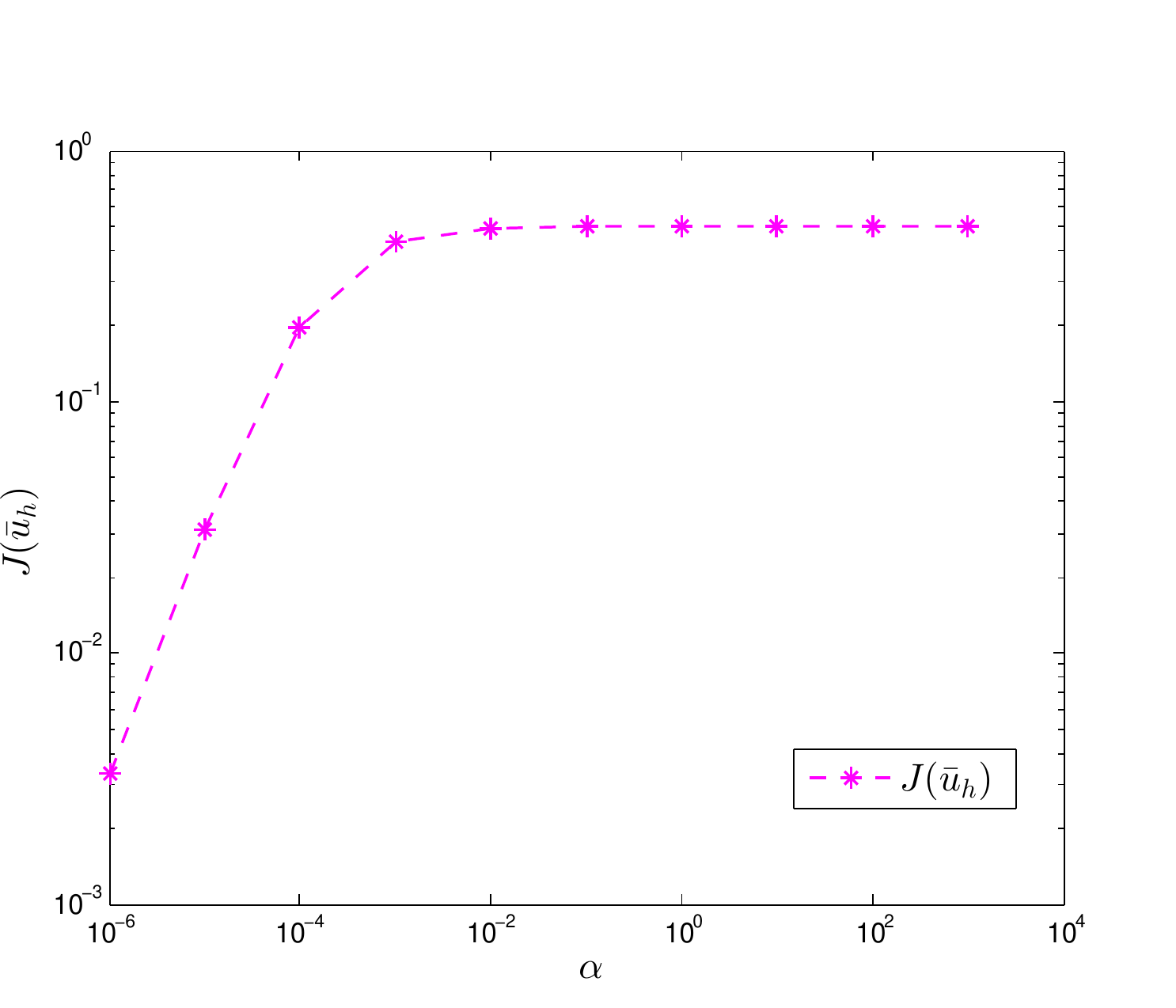}
                \caption{$J(\bar u_h)$  vs. $\alpha$.}
        \end{subfigure}
        
        \begin{subfigure}[h!]{0.5\textwidth}
                  \includegraphics[trim = 40mm 80mm 30mm 70mm, clip, width=\textwidth]{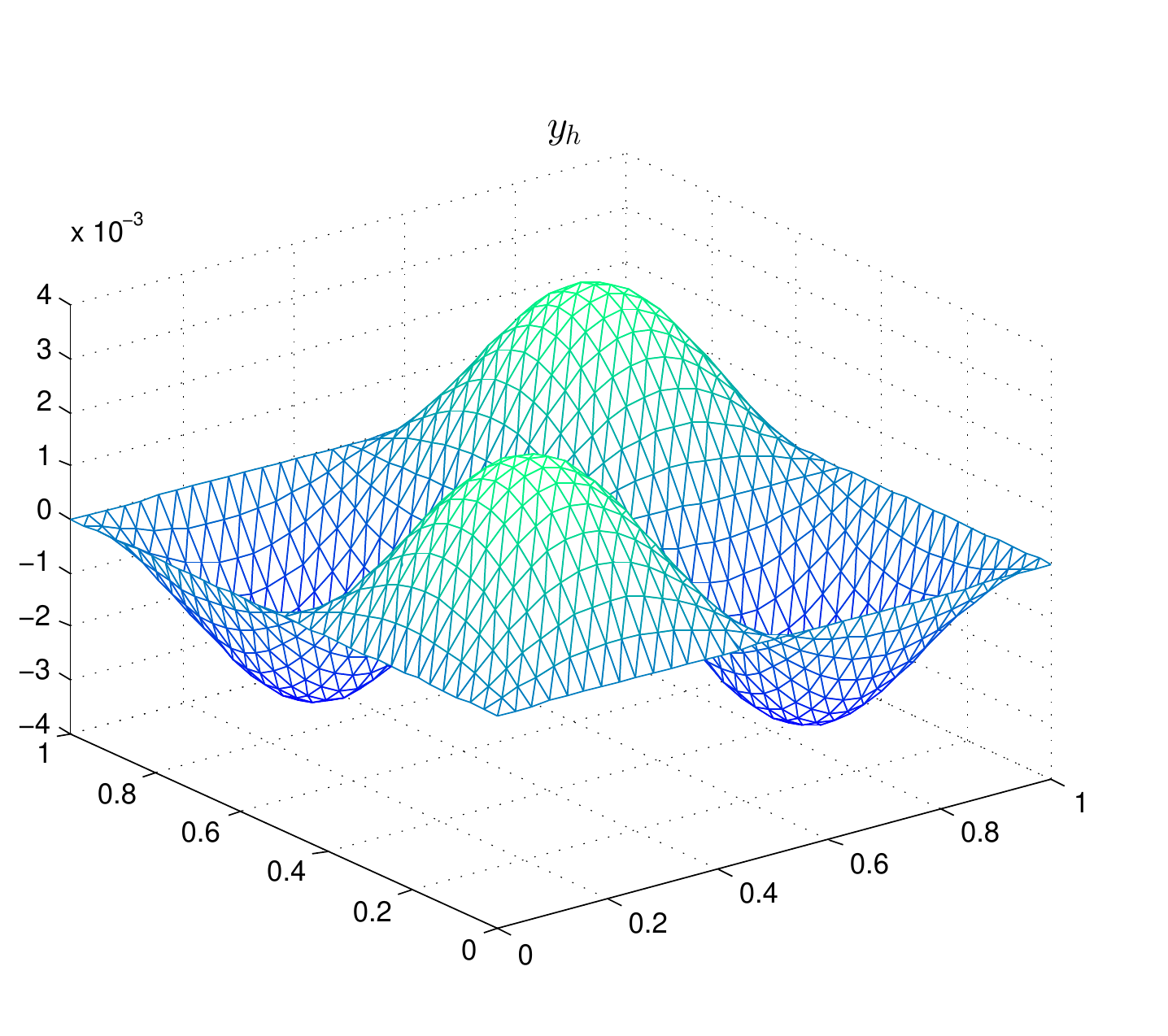}
                \caption{The optimal state $\bar y_h$.}
        \end{subfigure}~
        \begin{subfigure}[h!]{0.5\textwidth}
                  \includegraphics[trim = 40mm 80mm 30mm 70mm, clip, width=\textwidth]{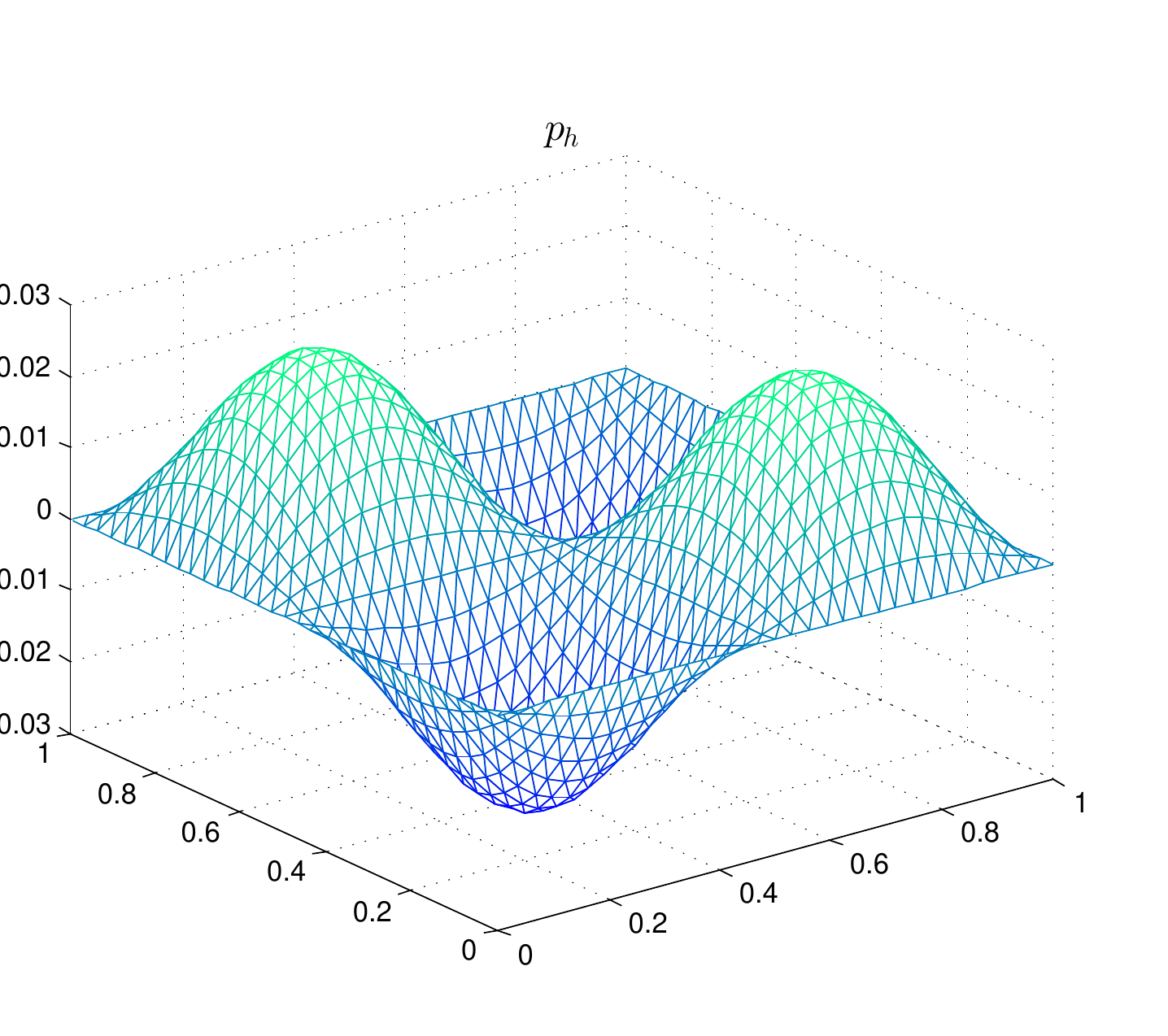}
                \caption{The adjoint state $\bar p_h$.}
        \end{subfigure}
        
        \begin{subfigure}[h!]{0.5\textwidth}
                  \includegraphics[trim = 40mm 80mm 30mm 70mm, clip, width=\textwidth]{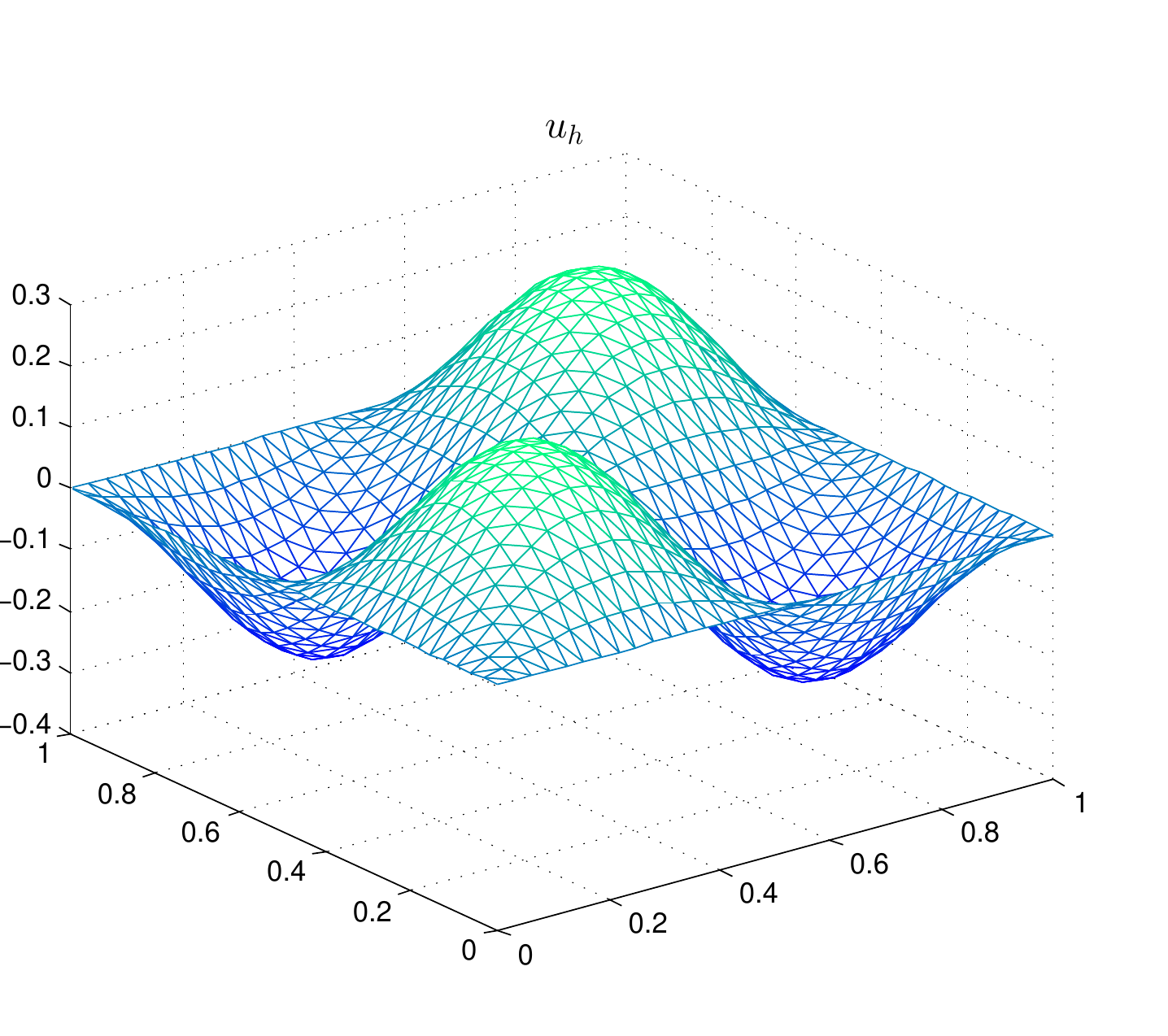}
                \caption{The optimal control $\bar u_h$.}
        \end{subfigure}

        \caption{Example~\ref{example: y power 3} Case~1 with choice~\textbf{A1} for $y_0$: The values of $\|\bar p_h\|_{L^4}$, $\eta(\alpha)$ and $J(\bar u_h)$ vs. $\alpha$. The optimal state $\bar y_h$, the optimal control $\bar u_h$ and the adjoint state $\bar p_h$ for $\alpha=10^{-1}$.}
        \label{figure: example y3 case(unconstrained) choiceA1}
\end{figure}

\begin{table}[p]
\caption{Example~\ref{example: y power 3} Case~1 with choice~\textbf{A2} for $y_0$: The values of  $\|\bar p_h\|_{L^4} $, $\eta(\alpha)$ and $J(\bar u_h)$ for different values of $\alpha$.}
\label{table: example y3 case(unconstrained) choiceA2}
\begin{tabular}{ l  c  c  c}
\toprule
$\alpha$ &	  $\|\bar p_h\|_{L^4}$&        $\eta(\alpha)$ &   $J(\bar u_h)$  \\
\midrule[1pt]

1.0e-06 &	  7.823778739727e-03 &	 	 6.776197632762e-03 &	 	 7.227759688190e+01   \\ 
1.0e-05 &	  2.234541300612e-02 &	 	 1.606889689070e-02 &	 	 1.065710637346e+02   \\
1.0e-04 &	  5.805844706415e-02 &	 	 3.810535956559e-02 &	 	 1.386316936362e+02   \\
1.0e-03 &	  1.125576598202e-01 &	 	 9.036204771862e-02 &	 	 1.568821491955e+02   \\
1.0e-02 &	  2.290137136719e-01 &	 	 2.142821839497e-01 &	 	 1.625724420922e+02   \\
1.0e-01 &	  2.997603240217e-01 &	 	 5.081431366100e-01 &	 	 1.642031427088e+02   \\
1.0e+00 &	  3.061090377257e-01 &	 	 1.204997272869e+00 &	 	 1.644198126030e+02   \\
1.0e+01 &	  3.066635772733e-01 &	 	 2.857498848277e+00 &	 	 1.644419766418e+02   \\
1.0e+02 &	  3.067181181971e-01 &	 	 6.776197632762e+00 &	 	 1.644441976184e+02   \\
1.0e+03 &	  3.067235630566e-01 &	 	 1.606889689070e+01 &	 	 1.644444197614e+02   \\

\bottomrule 
\end{tabular}
\end{table}

\begin{figure}[p]
        \centering
        \begin{subfigure}[h!]{0.5\textwidth}
                \includegraphics[trim = 40mm 80mm 30mm 70mm, clip, width=\textwidth]{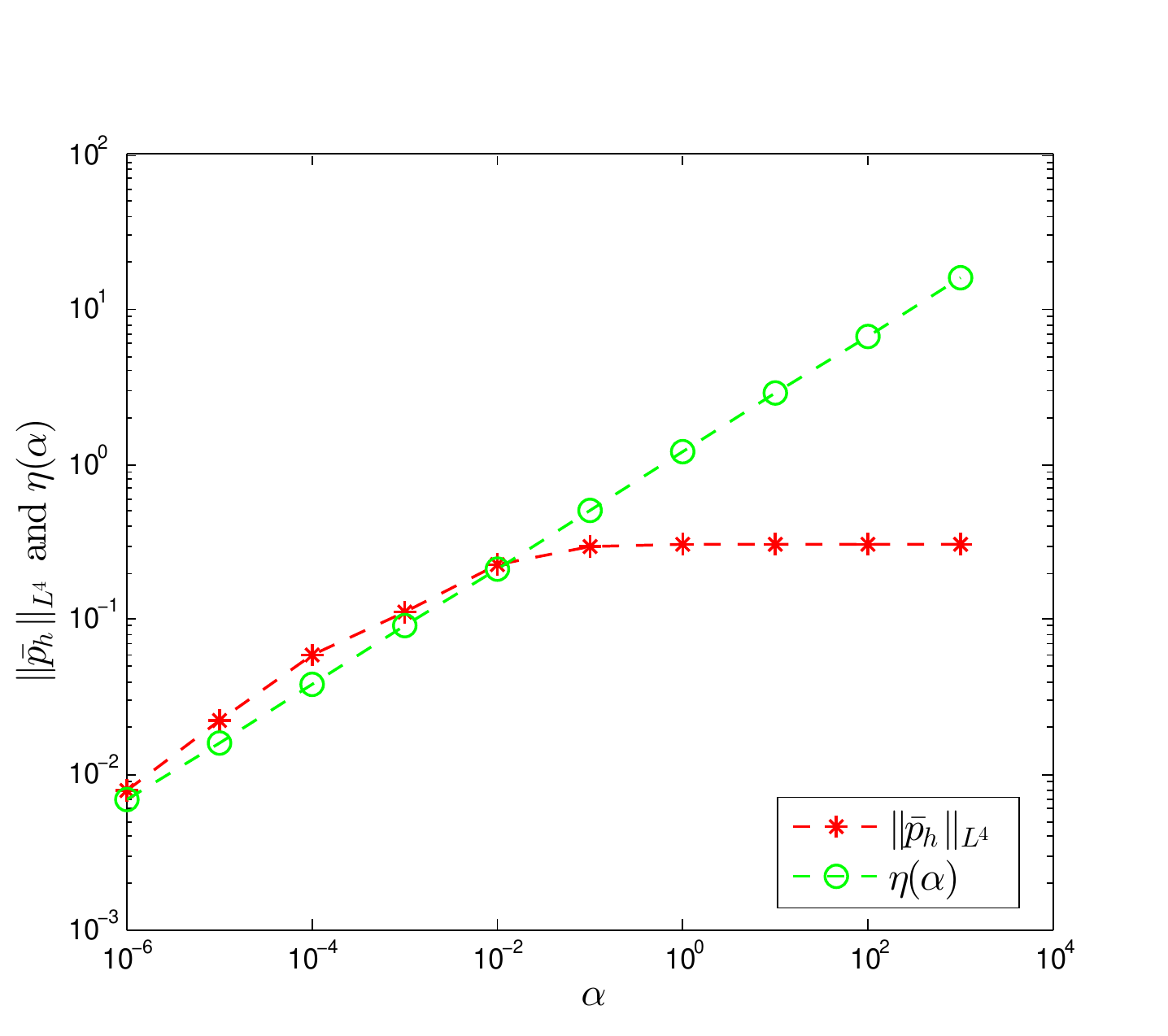}
                \caption{$\|\bar p_h\|_{L^4}$ and $\eta(\alpha)$ vs. $\alpha$.}
        \end{subfigure}%
        ~ 
        \begin{subfigure}[h!]{0.5\textwidth}
                \includegraphics[trim = 35mm 80mm 30mm 70mm, clip, width=\textwidth]{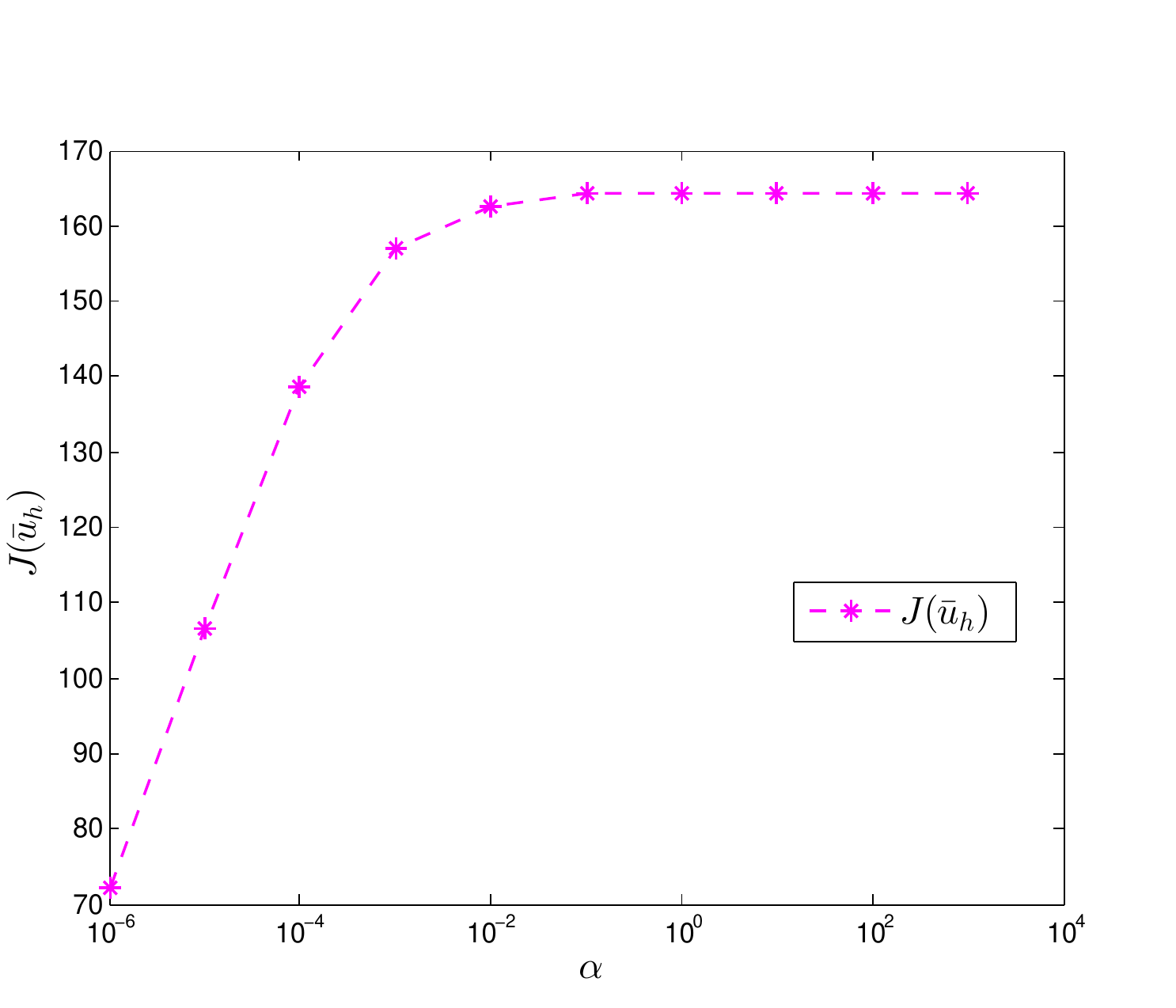}
                \caption{$J(\bar u_h)$  vs. $\alpha$.}
        \end{subfigure}
        
        \begin{subfigure}[h!]{0.5\textwidth}
                \includegraphics[trim = 40mm 80mm 30mm 70mm, clip, width=\textwidth]{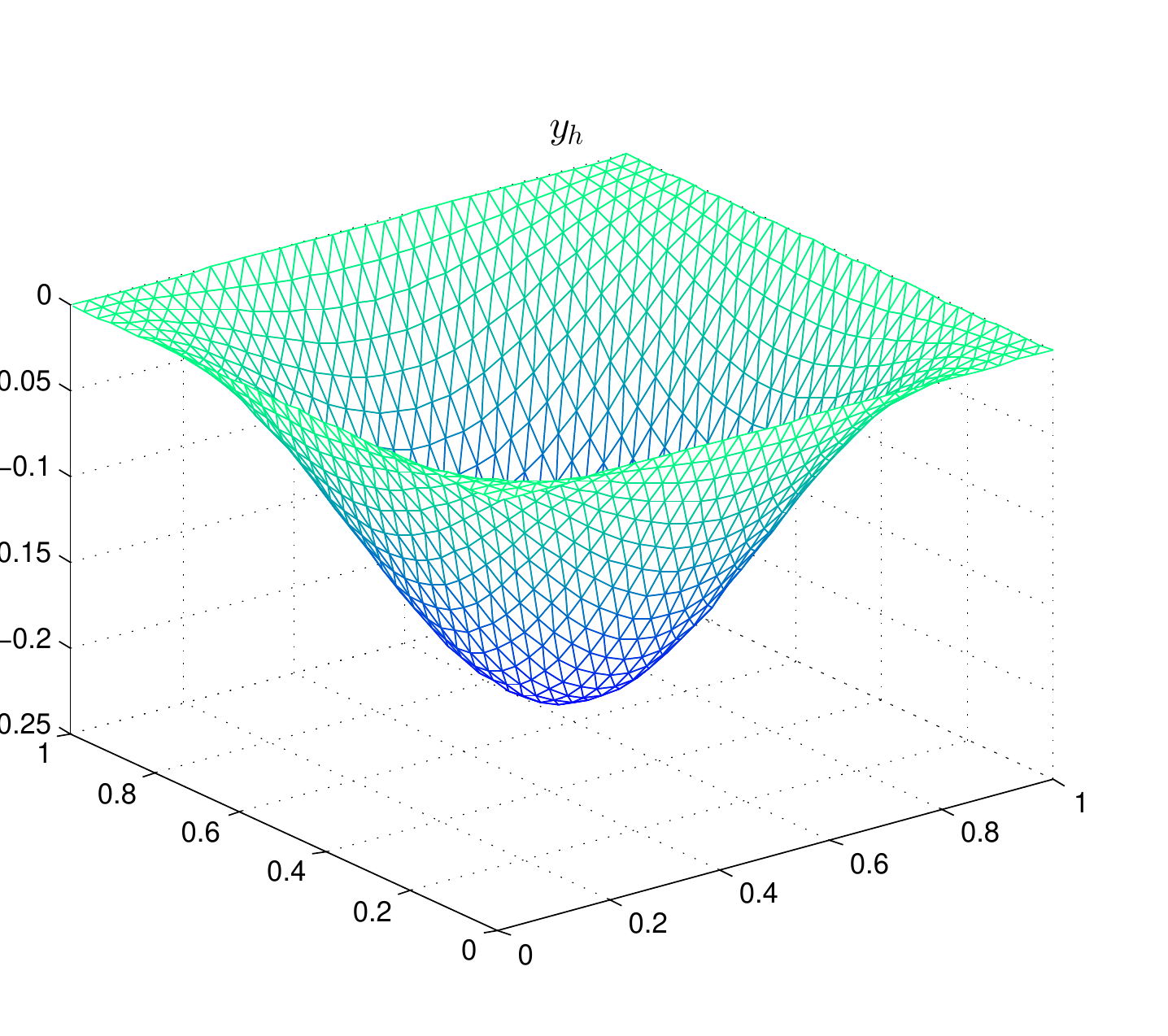}
                \caption{The optimal state $\bar y_h$.}
        \end{subfigure}~
        \begin{subfigure}[h!]{0.5\textwidth}
                \includegraphics[trim = 40mm 80mm 30mm 70mm, clip, width=\textwidth]{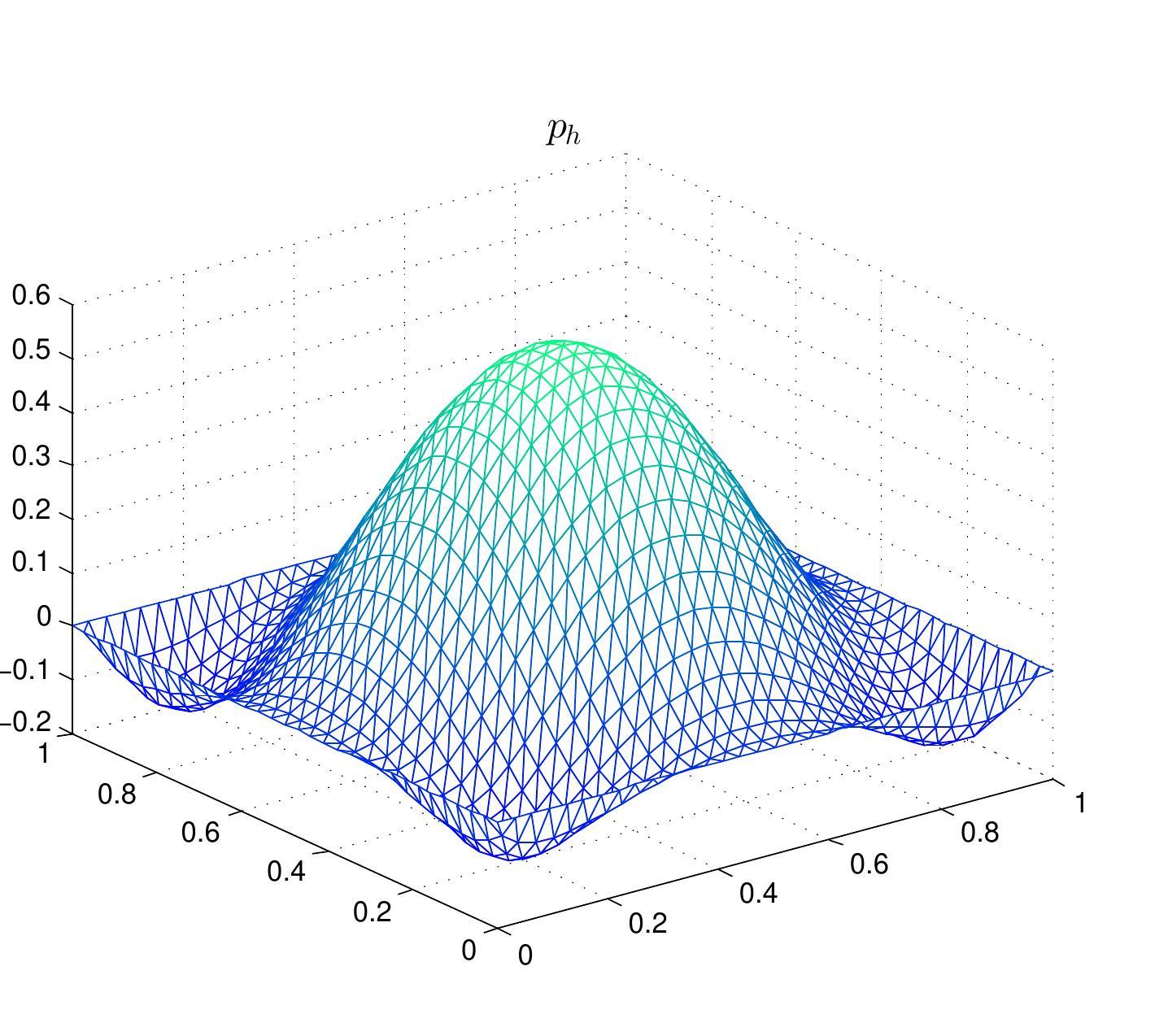}
                \caption{The adjoint state $\bar p_h$.}
        \end{subfigure}
        \begin{subfigure}[h!]{0.5\textwidth}
                \includegraphics[trim = 40mm 80mm 30mm 70mm, clip, width=\textwidth]{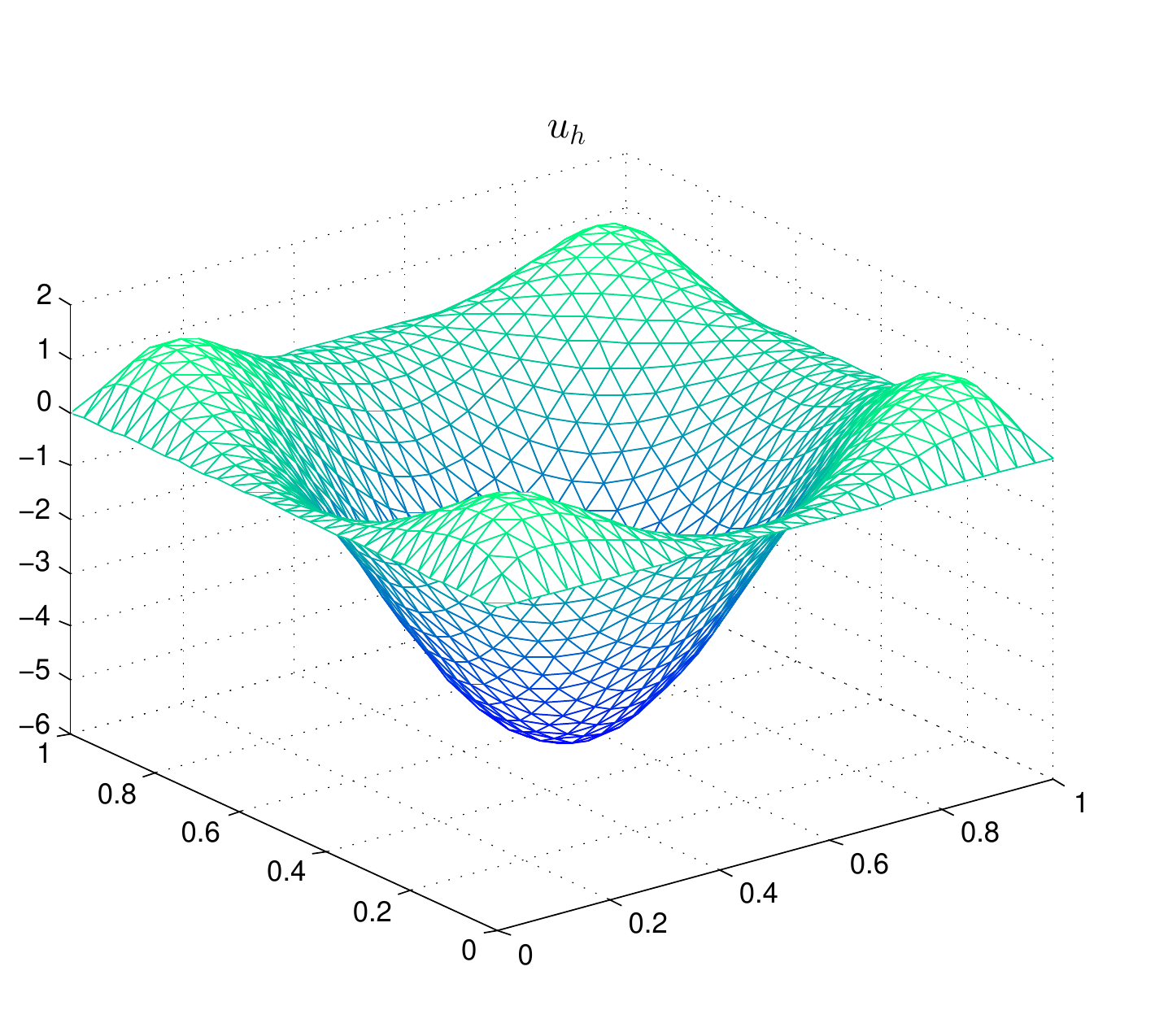}
                \caption{The optimal control $\bar u_h$.}
        \end{subfigure}
        
        \caption{Example~\ref{example: y power 3} Case~1 with choice~\textbf{A2} for $y_0$: The values of $\|\bar p_h\|_{L^4}$, $\eta(\alpha)$ and $J(\bar u_h)$ vs. $\alpha$. The optimal state $\bar y_h$, the optimal control $\bar u_h$ and the adjoint state $\bar p_h$ for $\alpha=10^{-1}$.}
        \label{figure: example y3 case(unconstrained) choiceA2}
\end{figure}


\noindent
\textbf{Case~2} (constrained control)
In this case we consider constraints only on the control, we set
\begin{align*}
u_a & = -5, \\
u_b & = 5, \\
y_b &=-y_a=\infty.
\end{align*}
Table~\ref{table: example y3 case(constrained control) choiceA1} shows the values of  $\|\bar p_h\|_{L^4} $, $\eta(\alpha)$ and $J(\bar u_h)$ computed for different values of $\alpha$ while considering the choice~\textbf{A1} for $y_0$. The graphical illustration of these findings are shown in Figure~\ref{figure: example y3 case(constrained control) choiceA1}. We see that $\bar u_h$ is a global minimum for $\alpha$ approximately greater than $10^{-5}$. The numerical results associated with the choice~\textbf{A2} are given in Table~\ref{table: example y3 case(constrained control) choiceA2} and illustrated in Figure~\ref{figure: example y3 case(constrained control) choiceA2}. In this case $\bar u_h$ is a global minimum for $\alpha$ approximately greater than $10^{-1}$.

\begin{table}[p]
\caption{Example~\ref{example: y power 3} Case~2 with choice~\textbf{A1} for $y_0$: The values of  $\|\bar p_h\|_{L^4} $, $\eta(\alpha)$ and $J(\bar u_h)$ for different values of $\alpha$.}
\label{table: example y3 case(constrained control) choiceA1}
\begin{tabular}{ l  c  c  c}
\toprule
$\alpha$ &	  $\|\bar p_h\|_{L^4}$&        $\eta(\alpha)$ &   $J(\bar u_h)$  \\
\midrule[1pt]

1.0e-06 &	  1.455724773650e-02 &	 	 6.776197632762e-03 &	 	 4.507886038196e-01   \\ 
1.0e-05 &	  1.455724403855e-02 &	 	 1.606889689070e-02 &	 	 4.508916148391e-01   \\
1.0e-04 &	  1.455717724977e-02 &	 	 3.810535956559e-02 &	 	 4.519082323790e-01   \\
1.0e-03 &	  1.457338622672e-02 &	 	 9.036204771862e-02 &	 	 4.612690393001e-01   \\
1.0e-02 &	  1.509224717529e-02 &	 	 2.142821839497e-01 &	 	 4.922544738762e-01   \\
1.0e-01 &	  1.530600543072e-02 &	 	 5.081431366100e-01 &	 	 4.992144829702e-01   \\
1.0e+00 &	  1.532768796263e-02 &	 	 1.204997272869e+00 &	 	 4.999213370332e-01   \\
1.0e+01 &	  1.532985932323e-02 &	 	 2.857498848277e+00 &	 	 4.999921325890e-01   \\
1.0e+02 &	  1.533007649041e-02 &	 	 6.776197632762e+00 &	 	 4.999992132478e-01   \\
1.0e+03 &	  1.533009820744e-02 &	 	 1.606889689070e+01 &	 	 4.999999213247e-01   \\

\bottomrule 
\end{tabular}
\end{table}

\begin{figure}[p]
        \centering
        \begin{subfigure}[h!]{0.5\textwidth}
                \includegraphics[trim = 40mm 80mm 30mm 70mm, clip, width=\textwidth]{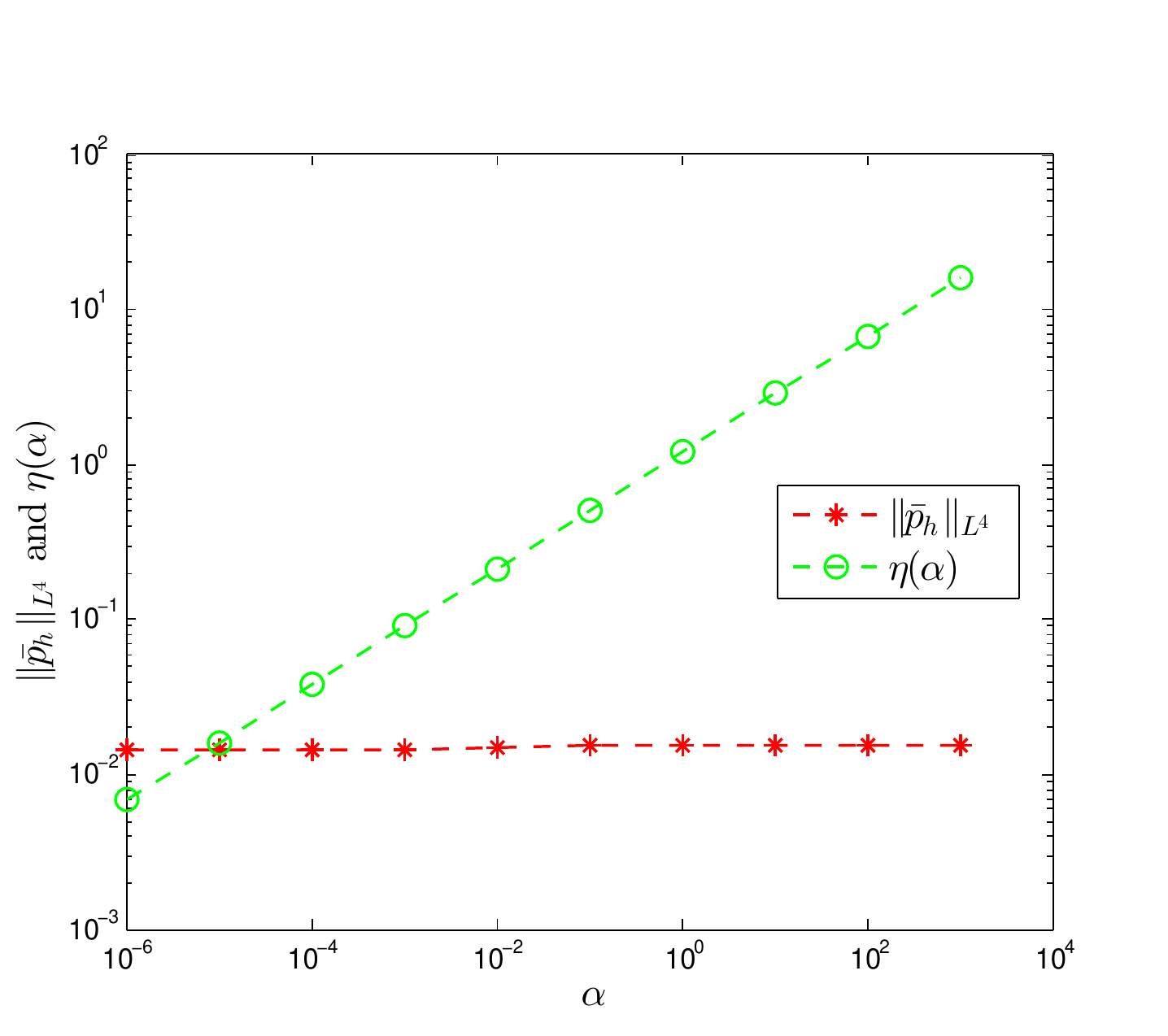}
                \caption{$\|\bar p_h\|_{L^4}$ and $\eta(\alpha)$ vs. $\alpha$.}
        \end{subfigure}%
        ~ 
        \begin{subfigure}[h!]{0.5\textwidth}
                \includegraphics[trim = 30mm 80mm 30mm 70mm, clip, width=\textwidth]{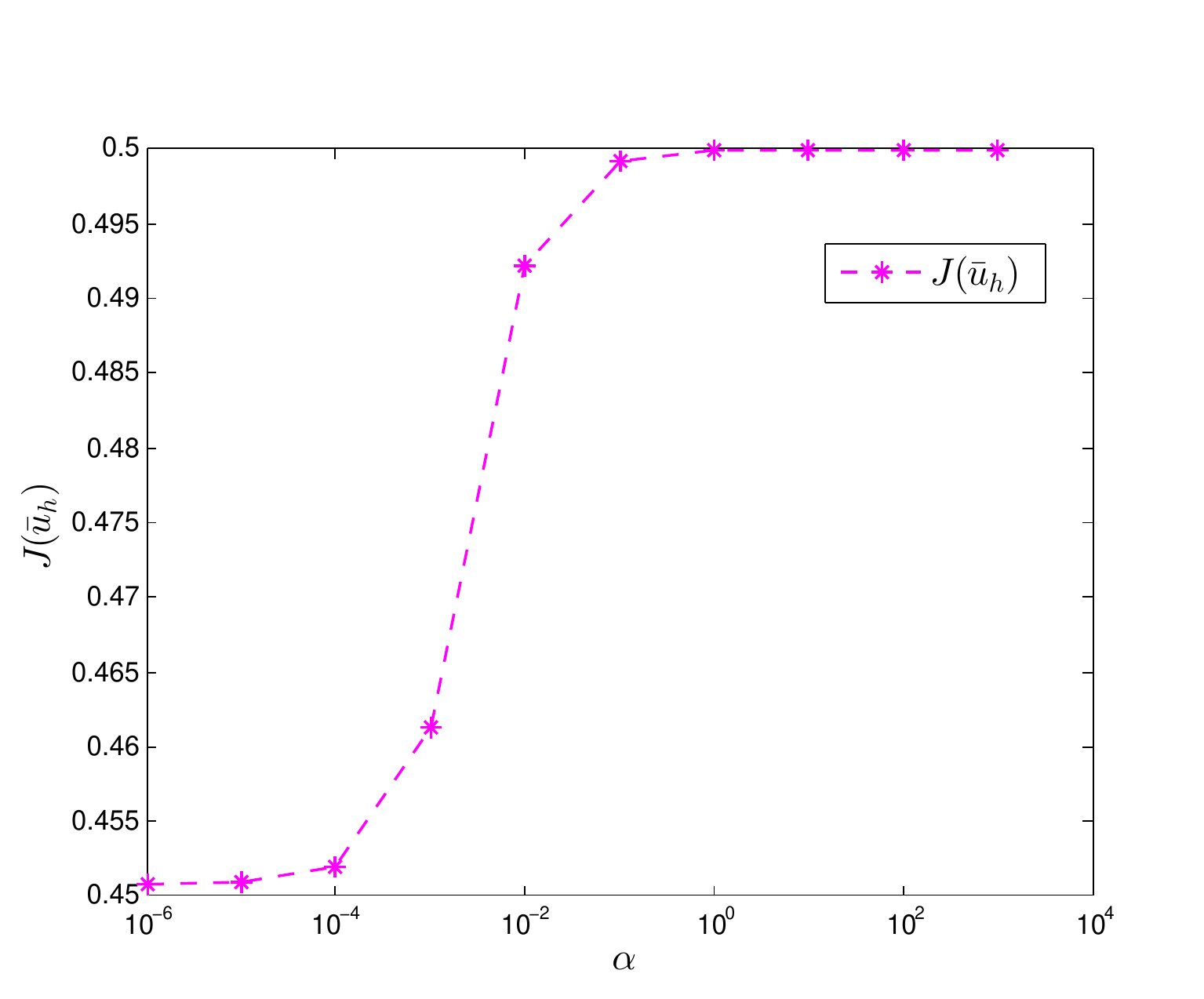}
                \caption{$J(\bar u_h)$  vs. $\alpha$.}
        \end{subfigure}
        
        \begin{subfigure}[h!]{0.5\textwidth}
                \includegraphics[trim = 40mm 80mm 30mm 70mm, clip, width=\textwidth]{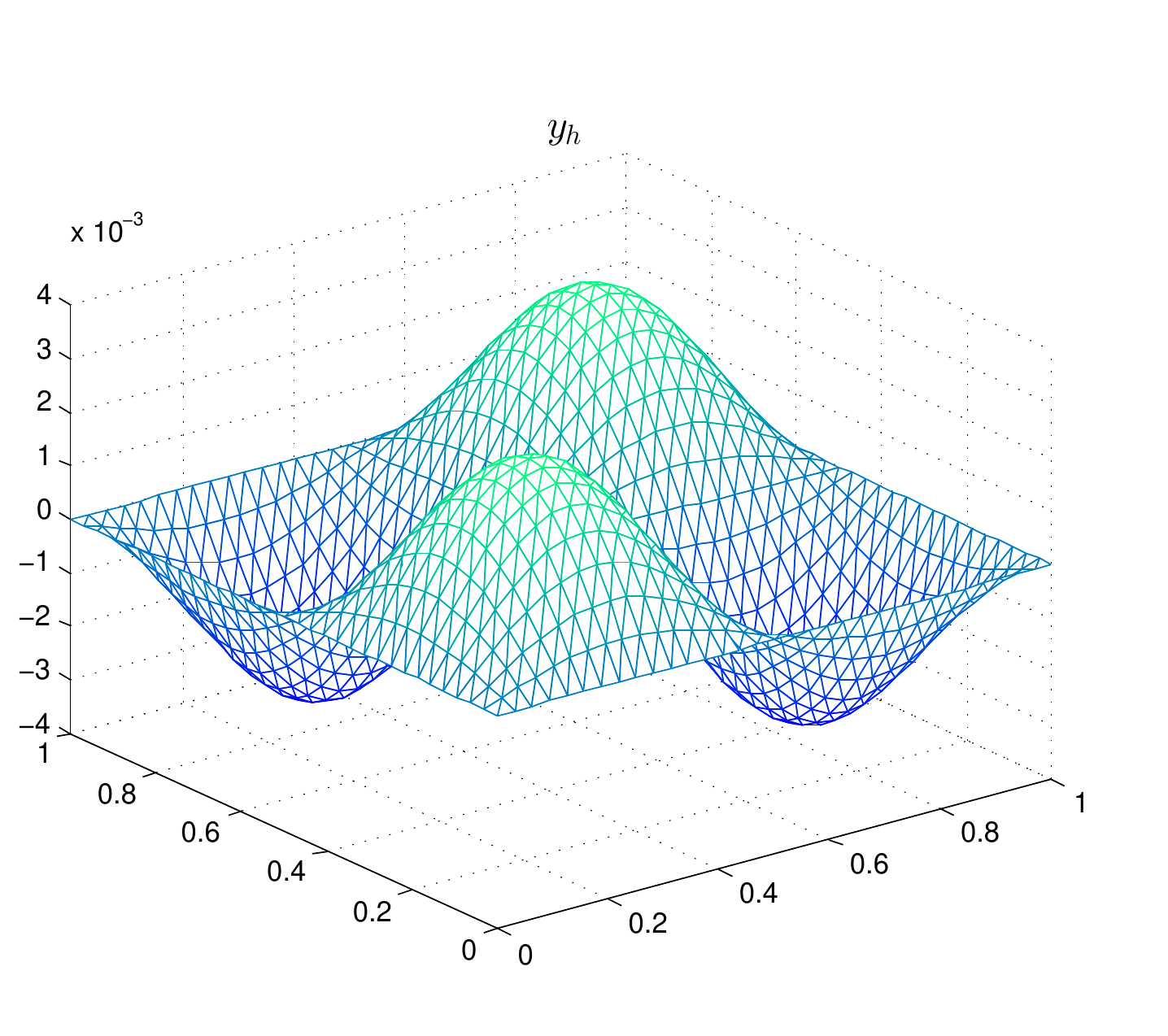}
                \caption{The optimal state $\bar y_h$.}
        \end{subfigure}~
        \begin{subfigure}[h!]{0.5\textwidth}
                \includegraphics[trim = 40mm 80mm 30mm 70mm, clip, width=\textwidth]{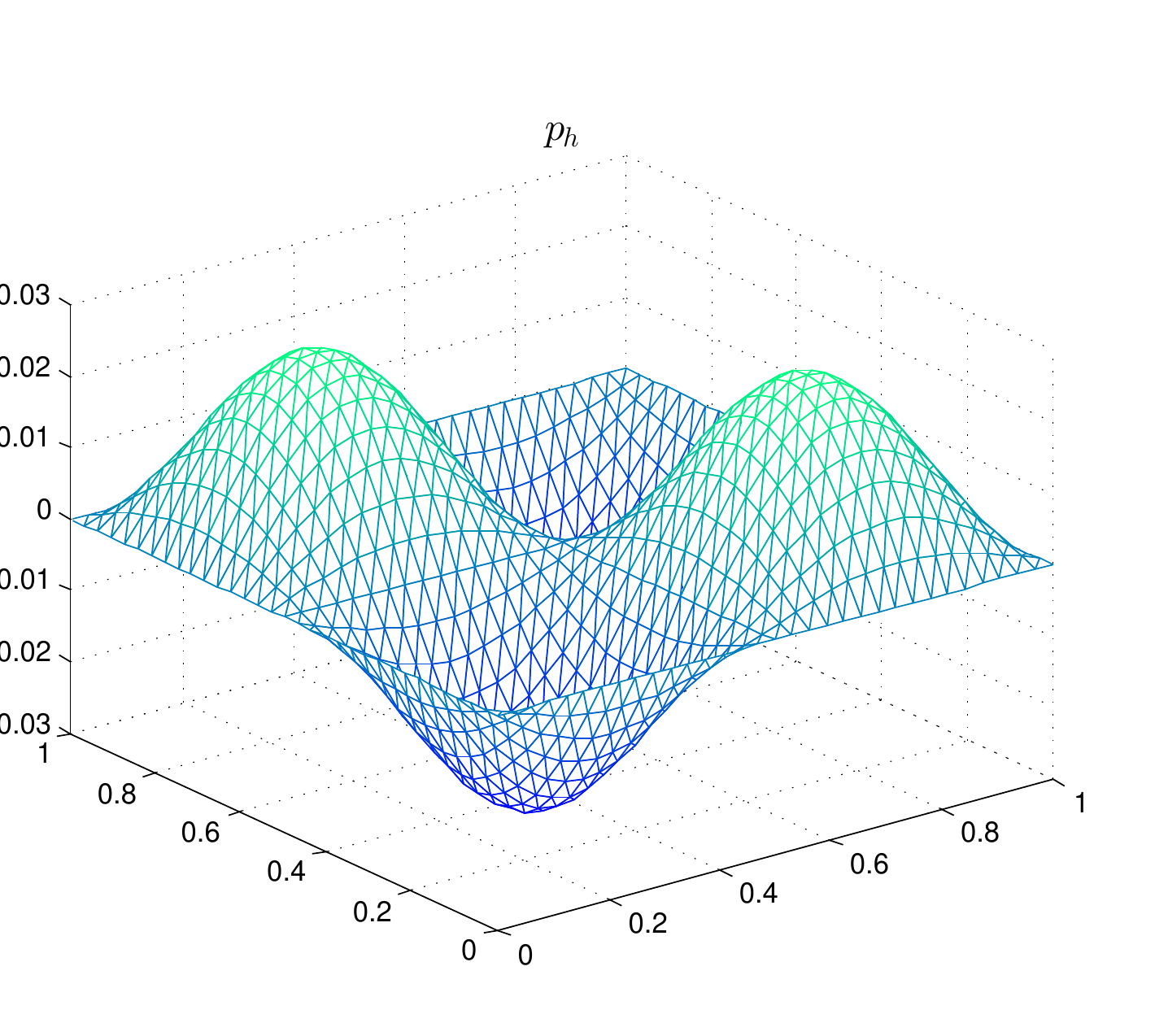}
                \caption{The adjoint state $\bar p_h$.}
        \end{subfigure}
        
        \begin{subfigure}[h!]{0.5\textwidth}
                \includegraphics[trim = 40mm 80mm 30mm 70mm, clip, width=\textwidth]{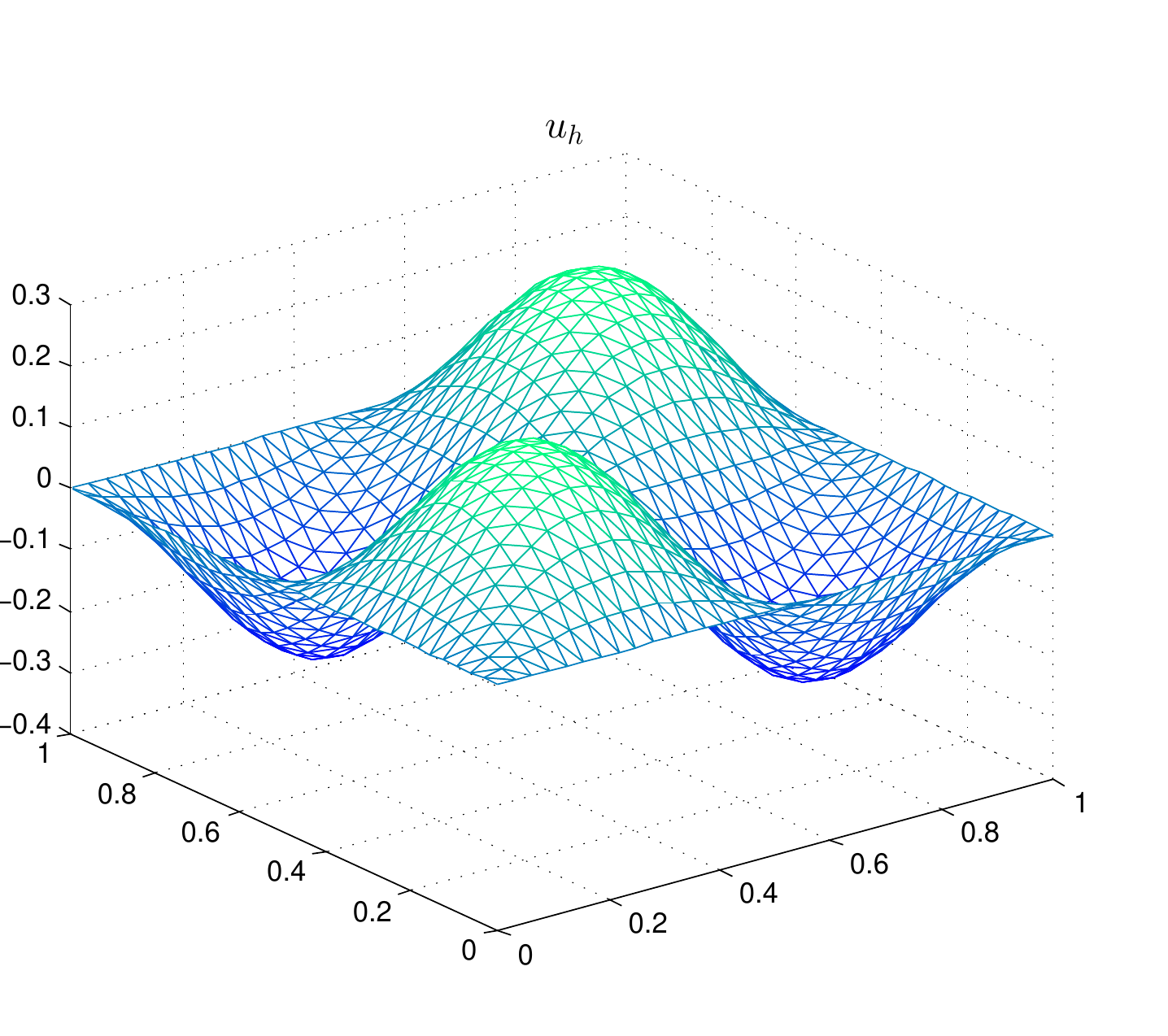}
                \caption{The optimal control $\bar u_h$.}
        \end{subfigure}

        \caption{Example~\ref{example: y power 3} Case~2 with choice~\textbf{A1} for $y_0$: The values of $\|\bar p_h\|_{L^4}$, $\eta(\alpha)$ and $J(\bar u_h)$ vs. $\alpha$. The optimal state $\bar y_h$, the optimal control $\bar u_h$ and the adjoint state $\bar p_h$ for $\alpha=10^{-1}$.}
        \label{figure: example y3 case(constrained control) choiceA1}
\end{figure}

\begin{table}[p]
\caption{Example~\ref{example: y power 3} Case~2 with choice~\textbf{A2} for $y_0$: The values of  $\|\bar p_h\|_{L^4} $, $\eta(\alpha)$ and $J(\bar u_h)$ for different values of $\alpha$.}
\label{table: example y3 case(constrained control) choiceA2}
\begin{tabular}{ l  c  c  c}
\toprule
$\alpha$ &	  $\|\bar p_h\|_{L^4}$&        $\eta(\alpha)$ &   $J(\bar u_h)$  \\
\midrule[1pt]

1.0e-06 &	  2.954513493743e-01 &	 	 6.776197632762e-03 &	 	 1.636849171437e+02   \\ 
1.0e-05 &	  2.954513728927e-01 &	 	 1.606889689070e-02 &	 	 1.636850204333e+02   \\
1.0e-04 &	  2.954526968135e-01 &	 	 3.810535956559e-02 &	 	 1.636860509190e+02   \\
1.0e-03 &	  2.954464960067e-01 &	 	 9.036204771862e-02 &	 	 1.636961799251e+02   \\
1.0e-02 &	  2.955530339094e-01 &	 	 2.142821839497e-01 &	 	 1.637871978058e+02   \\
1.0e-01 &	  2.998739300063e-01 &	 	 5.081431366100e-01 &	 	 1.642034478360e+02   \\
1.0e+00 &	  3.061090377257e-01 &	 	 1.204997272869e+00 &	 	 1.644198126030e+02   \\
1.0e+01 &	  3.066635772733e-01 &	 	 2.857498848277e+00 &	 	 1.644419766418e+02   \\
1.0e+02 &	  3.067181181971e-01 &	 	 6.776197632762e+00 &	 	 1.644441976184e+02   \\
1.0e+03 &	  3.067235630566e-01 &	 	 1.606889689070e+01 &	 	 1.644444197614e+02   \\

\bottomrule 
\end{tabular}
\end{table}

\begin{figure}[p]
        \centering
        \begin{subfigure}[h!]{0.5\textwidth}
                \includegraphics[trim = 40mm 80mm 30mm 70mm, clip, width=\textwidth]{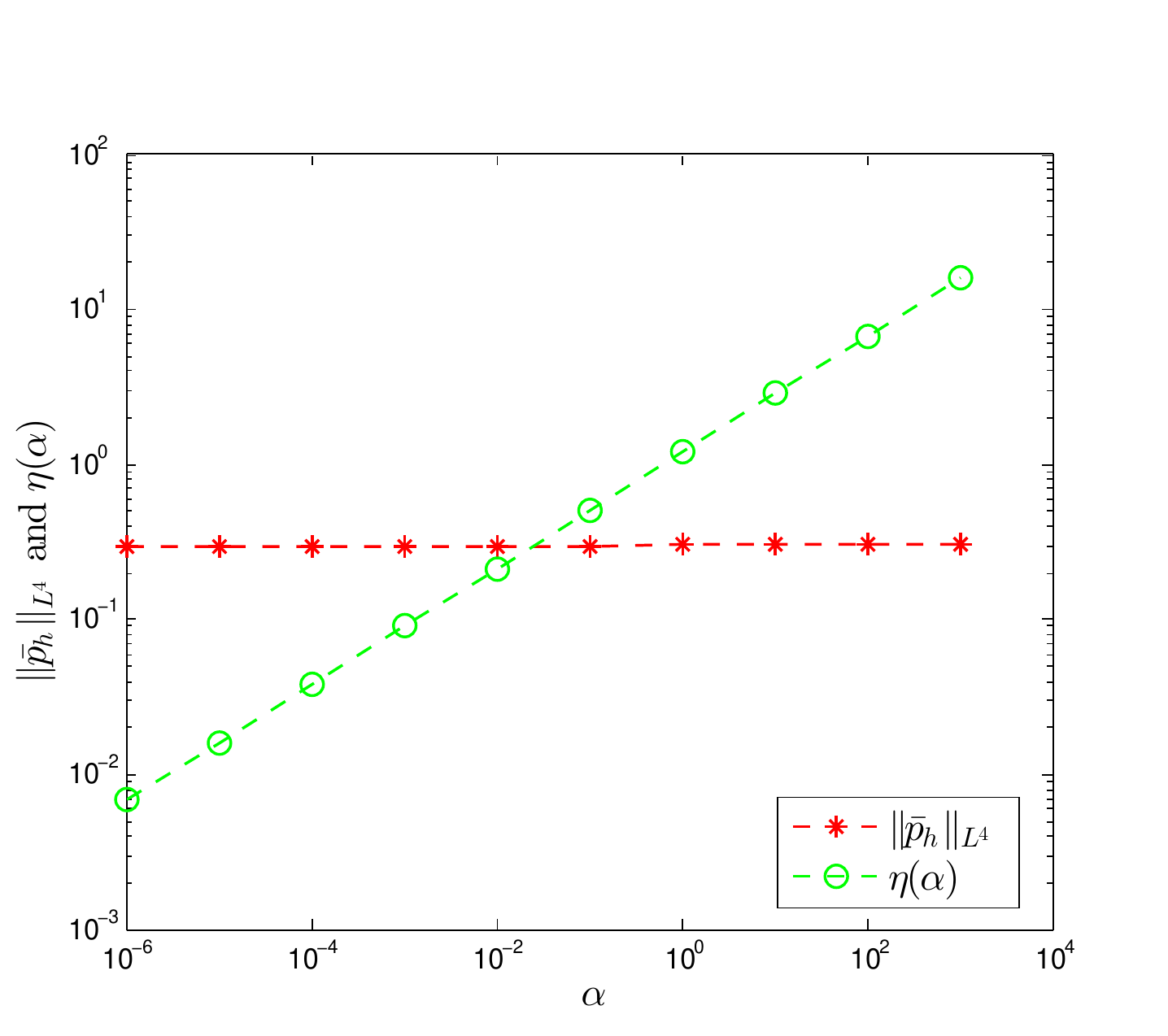}
                \caption{$\|\bar p_h\|_{L^4}$ and $\eta(\alpha)$ vs. $\alpha$.}
        \end{subfigure}%
        ~ 
        \begin{subfigure}[h!]{0.5\textwidth}
                \includegraphics[trim = 30mm 80mm 30mm 70mm, clip, width=\textwidth]{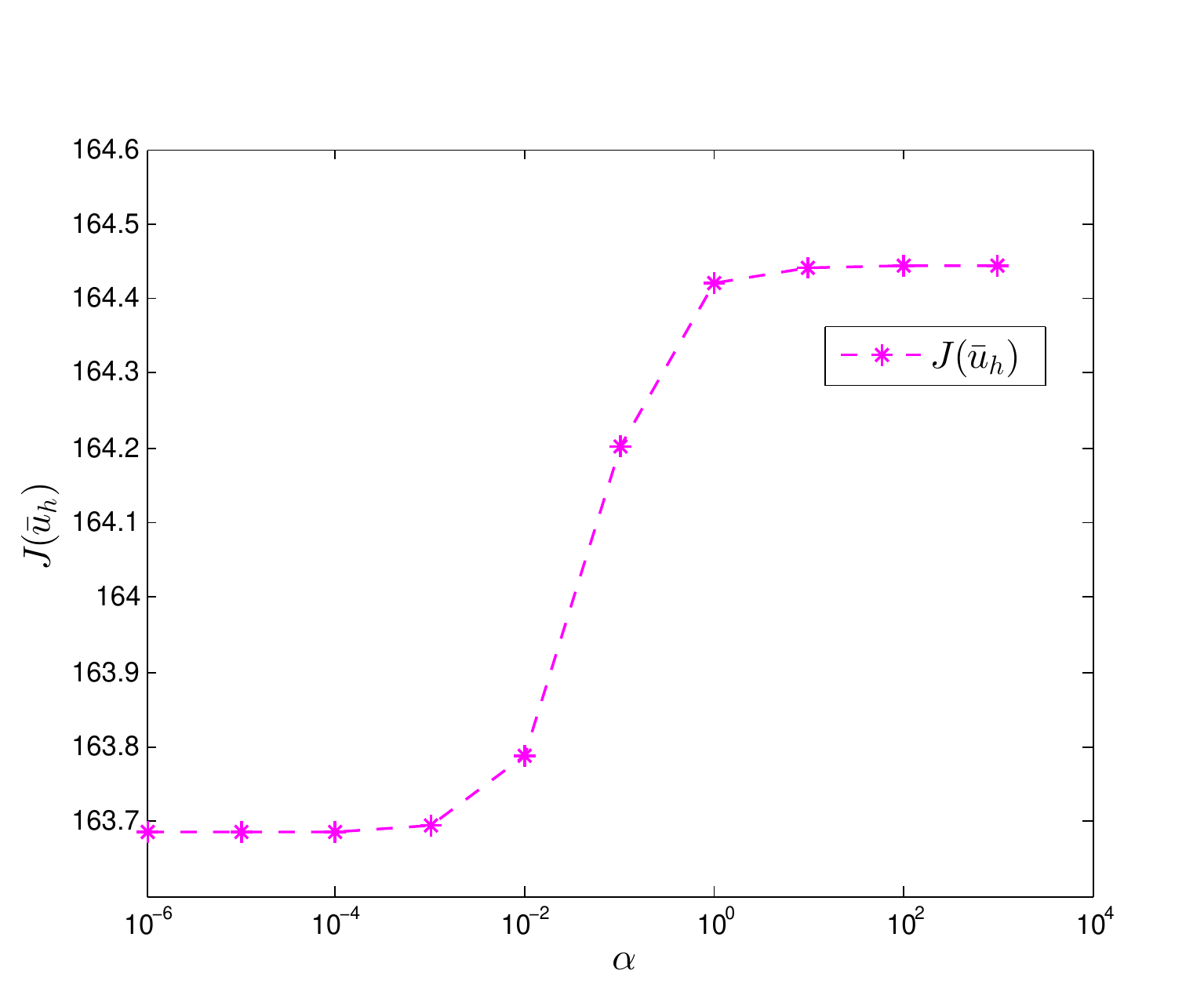}
                \caption{$J(\bar u_h)$  vs. $\alpha$.}
        \end{subfigure}
        
        \begin{subfigure}[h!]{0.5\textwidth}
                \includegraphics[trim = 40mm 80mm 30mm 70mm, clip, width=\textwidth]{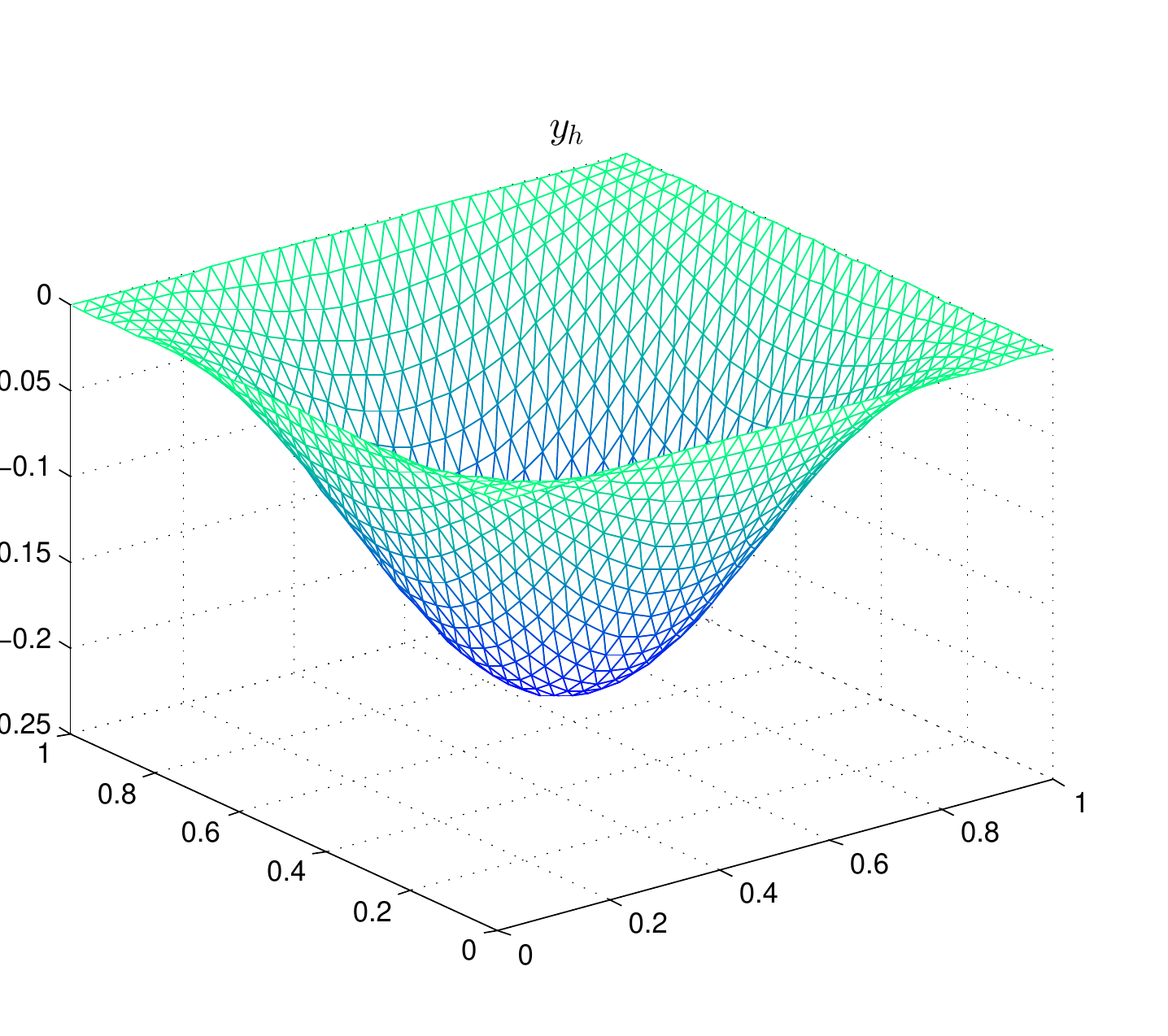}
                \caption{The optimal state $\bar y_h$.}
        \end{subfigure}~
        \begin{subfigure}[h!]{0.5\textwidth}
                \includegraphics[trim = 40mm 80mm 30mm 70mm, clip, width=\textwidth]{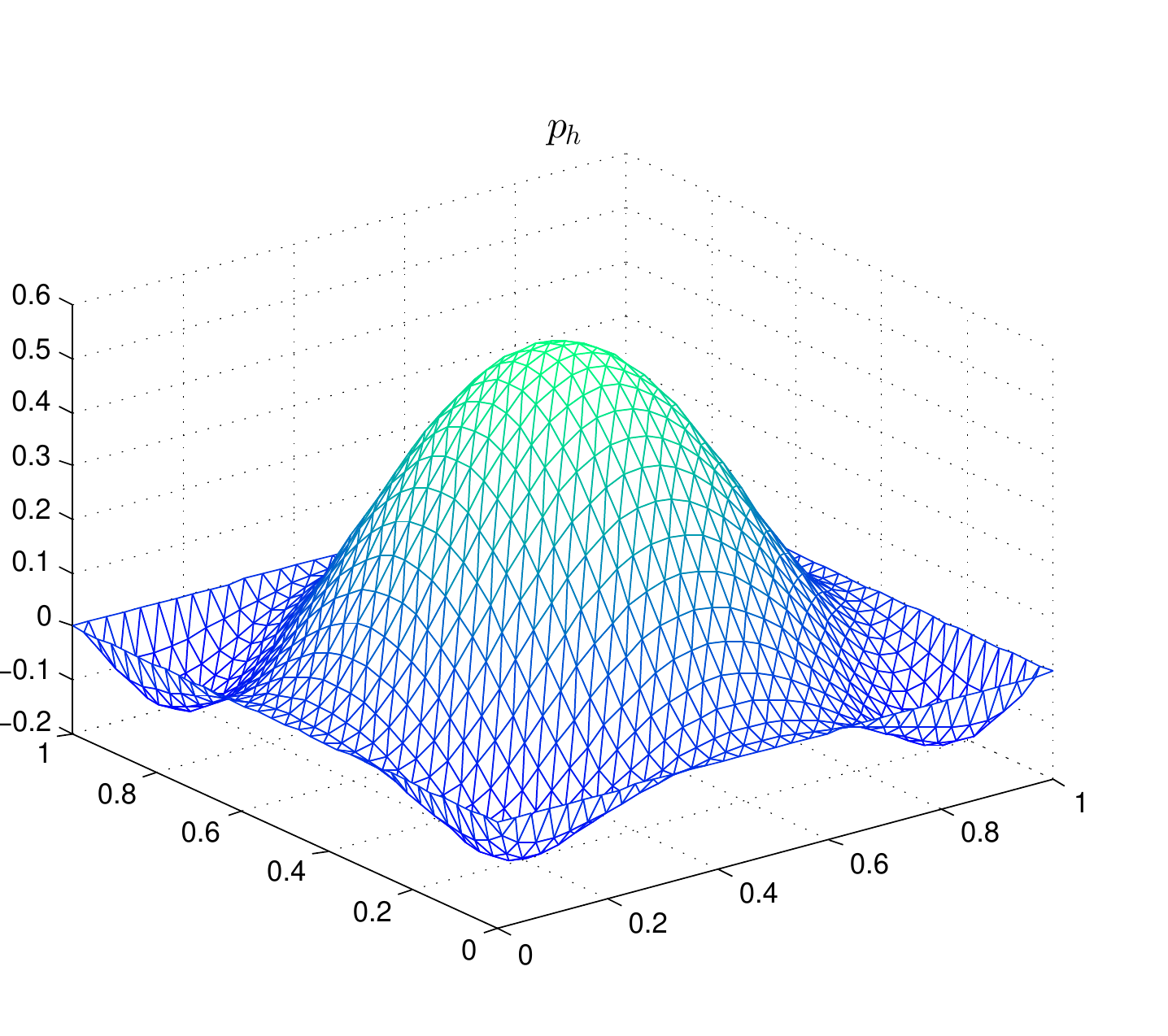}
                \caption{The adjoint state $\bar p_h$.}
        \end{subfigure}
        
        \begin{subfigure}[h!]{0.5\textwidth}
                \includegraphics[trim = 40mm 80mm 30mm 70mm, clip, width=\textwidth]{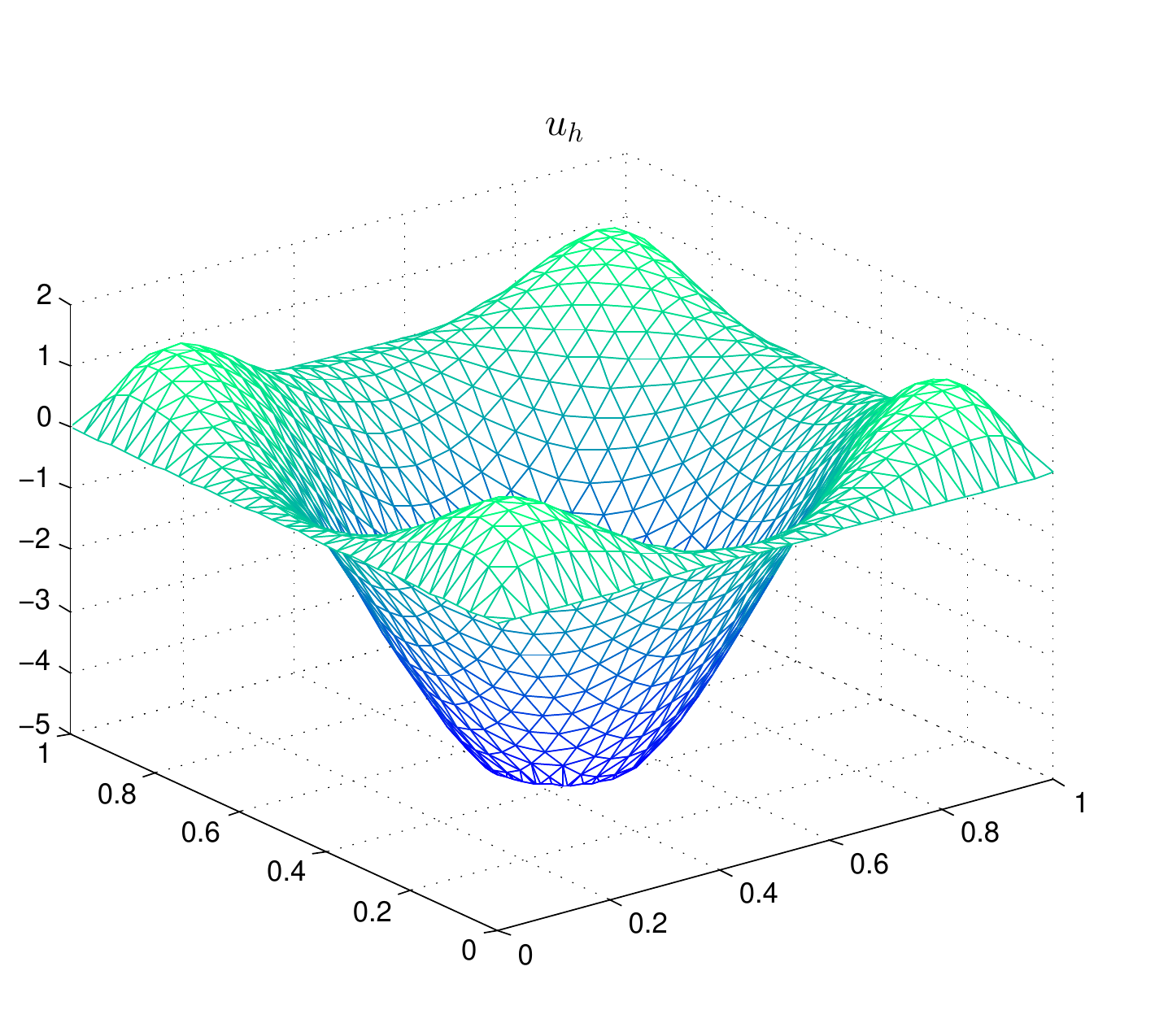}
                \caption{The optimal control $\bar u_h$.}
        \end{subfigure}~
        \begin{subfigure}[h!]{0.5\textwidth}
                \includegraphics[trim = 40mm 80mm 30mm 70mm, clip, width=\textwidth]{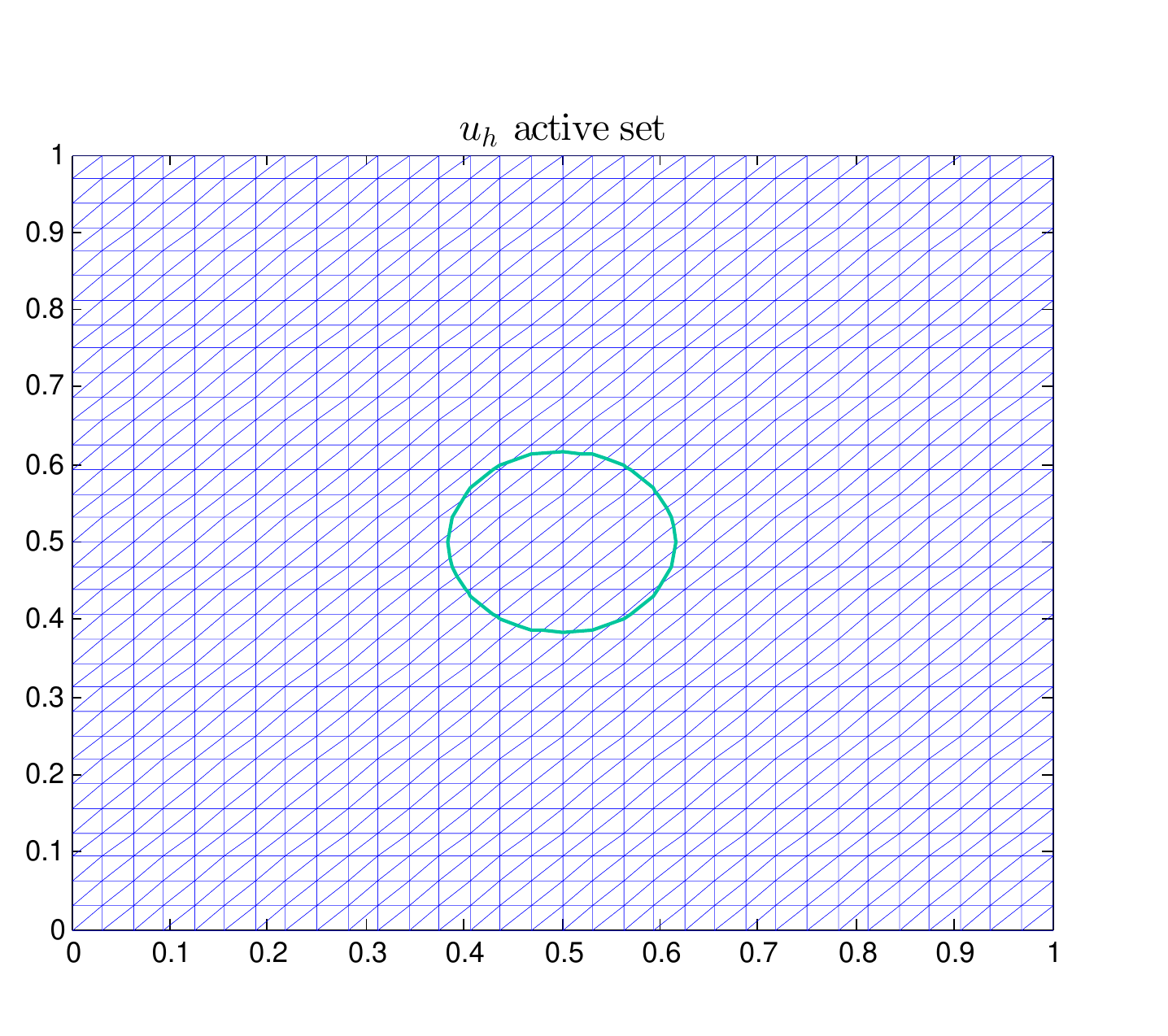}
                \caption{The control active set ($\bar u_h=-5$ inside the polygonal region).}
        \end{subfigure}
        
        \caption{Example~\ref{example: y power 3} Case~2 with choice~\textbf{A2} for $y_0$: The values of $\|\bar p_h\|_{L^4}$, $\eta(\alpha)$ and $J(\bar u_h)$ vs. $\alpha$. The optimal state $\bar y_h$, the optimal control $\bar u_h$, the control active set and the adjoint state $\bar p_h$ for $\alpha=10^{-1}$}
        \label{figure: example y3 case(constrained control) choiceA2}
\end{figure}


\noindent
\textbf{Case~3} (constrained state)
In this case we consider constrains only on the state, we set
\begin{align*}
u_b & = -u_a=\infty, \\
y_a &=-1, \\
y_b &=1.
\end{align*}
The numerical findings associated with choice~\textbf{A1} are provided in Table~\ref{table: example y3 case(constrained state) choiceA1} and illustrated in Figure~\ref{figure: example y3 case(constrained state) choiceA1}. For the choice~\textbf{A2} the results are given in Table~\ref{table: example y3 case(constrained state) choiceA2} and illustrated in Figure~\ref{figure: example y3 case(constrained state) choiceA2}. In both cases we see that $\bar u_h$ is a global minimum for all available values of $\alpha$.

\begin{table}[p]
\caption{Example~\ref{example: y power 3} Case~3 with choice~\textbf{A1} for $y_0$:  The values of  $\|\bar p_h\|_{L^4} $, $\eta(\alpha)$ and $J(\bar u_h)$ for different values of $\alpha$.}
\label{table: example y3 case(constrained state) choiceA1}
\begin{tabular}{ l  c  c  c}
\toprule
$\alpha$ &	  $\|\bar p_h\|_{L^4}$&        $\eta(\alpha)$ &   $J(\bar u_h)$  \\
\midrule[1pt]

1.0e-06 &	  1.166321621310e-04 &	 	 6.776197632762e-03 &	 	 6.248613075636e-02   \\ 
1.0e-05 &	  8.045399583166e-04 &	 	 1.606889689070e-02 &	 	 8.942494600427e-02   \\
1.0e-04 &	  5.009426247692e-03 &	 	 3.810535956559e-02 &	 	 2.037409649052e-01   \\
1.0e-03 &	  1.322797500856e-02 &	 	 9.036204771862e-02 &	 	 4.320833160546e-01   \\
1.0e-02 &	  1.509224717529e-02 &	 	 2.142821839497e-01 &	 	 4.922544738762e-01   \\
1.0e-01 &	  1.530600543072e-02 &	 	 5.081431366100e-01 &	 	 4.992144829702e-01   \\
1.0e+00 &	  1.532768796263e-02 &	 	 1.204997272869e+00 &	 	 4.999213370332e-01   \\
1.0e+01 &	  1.532985932323e-02 &	 	 2.857498848277e+00 &	 	 4.999921325890e-01   \\
1.0e+02 &	  1.533007649041e-02 &	 	 6.776197632762e+00 &	 	 4.999992132478e-01   \\
1.0e+03 &	  1.533009820744e-02 &	 	 1.606889689070e+01 &	 	 4.999999213247e-01   \\

\bottomrule 
\end{tabular}
\end{table}

\begin{figure}[p]
        \centering
        \begin{subfigure}[h!]{0.5\textwidth}
                \includegraphics[trim = 40mm 80mm 30mm 70mm, clip, width=\textwidth]{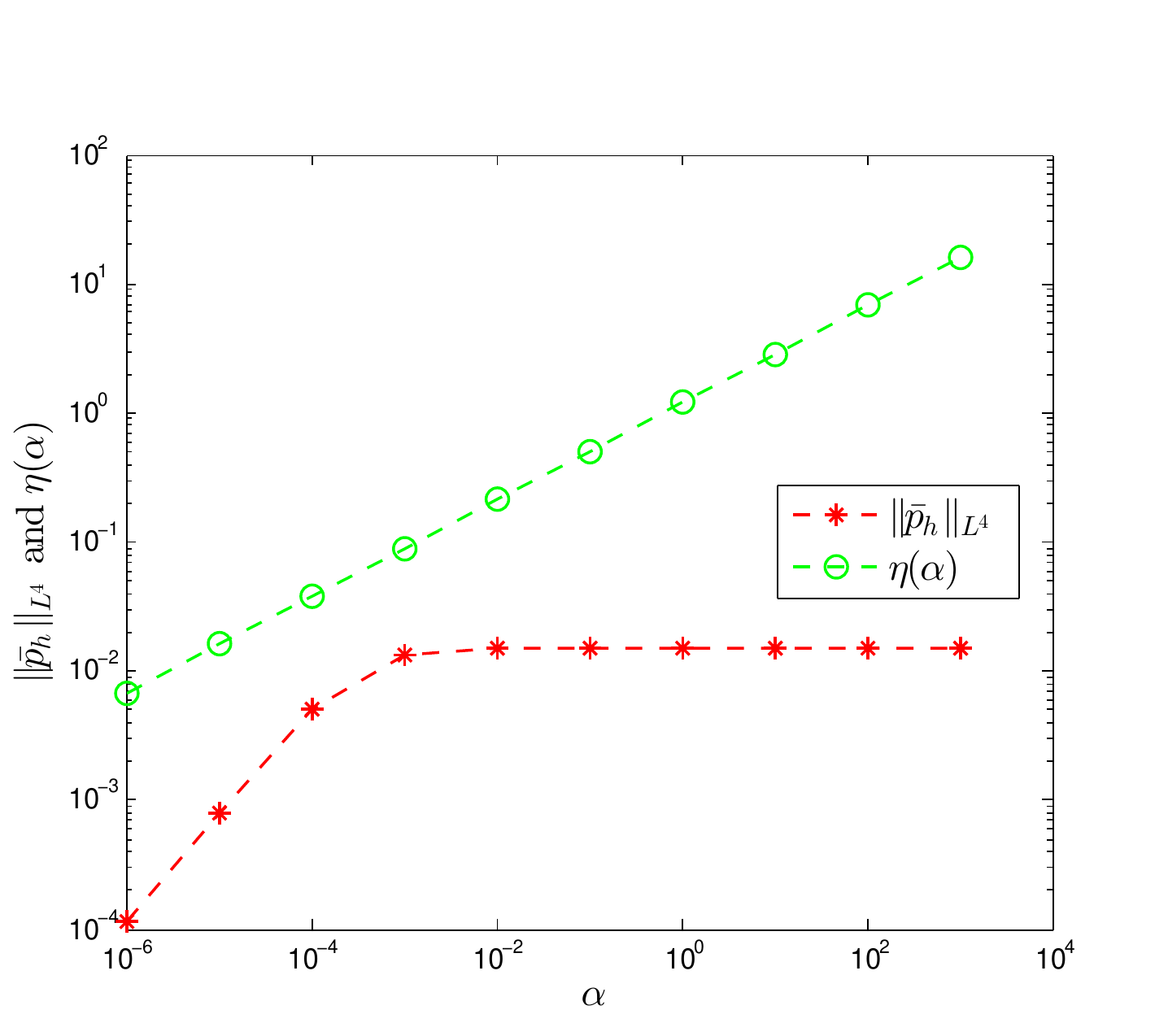}
                \caption{$\|\bar p_h\|_{L^4}$ and $\eta(\alpha)$ vs. $\alpha$.}
        \end{subfigure}%
        ~ 
        \begin{subfigure}[h!]{0.5\textwidth}
                \includegraphics[trim = 30mm 80mm 30mm 70mm, clip, width=\textwidth]{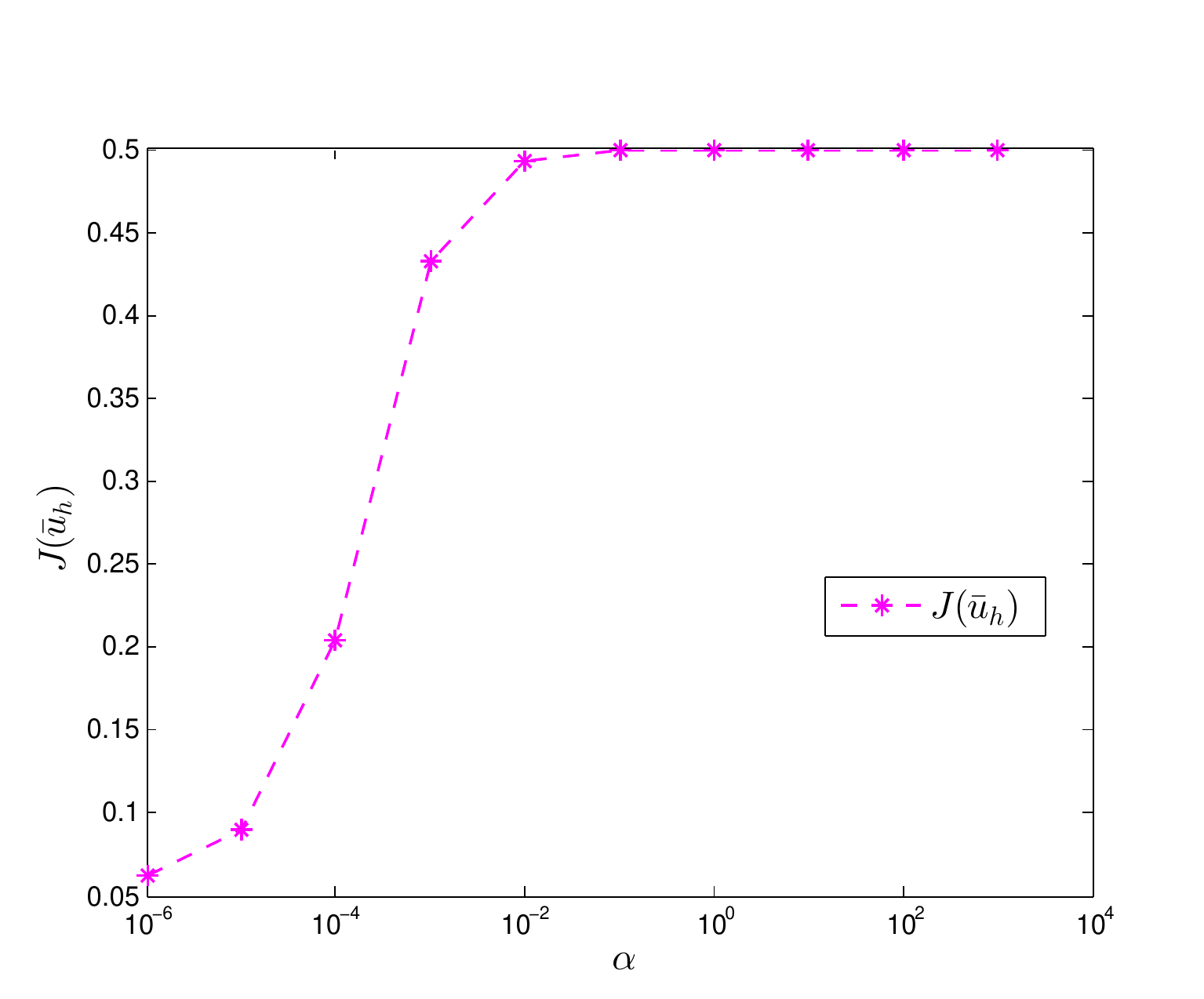}
                \caption{$J(\bar u_h)$  vs. $\alpha$.}
        \end{subfigure}
        
        \begin{subfigure}[h!]{0.5\textwidth}
                \includegraphics[trim = 40mm 80mm 30mm 70mm, clip, width=\textwidth]{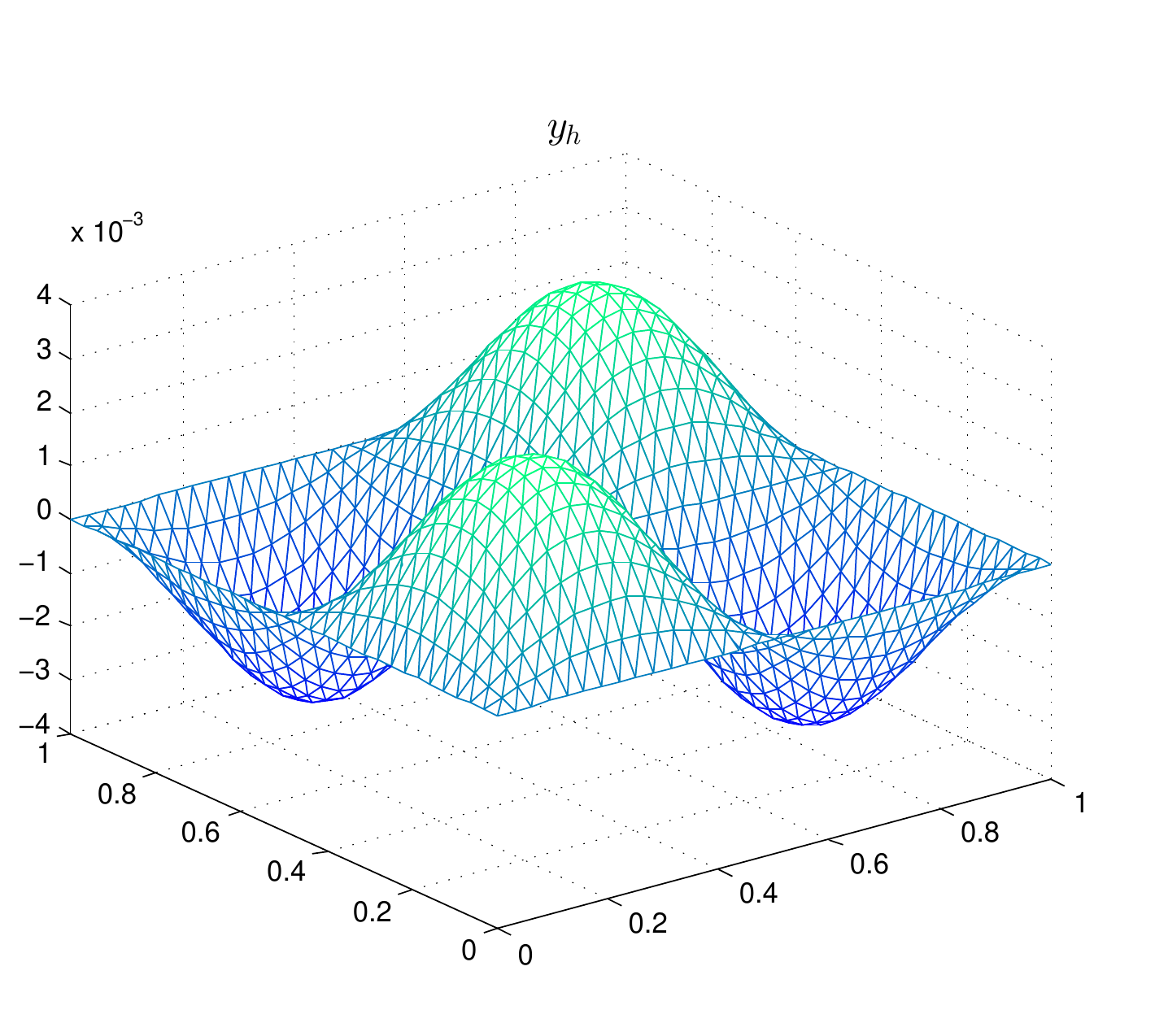}
                \caption{The optimal state $\bar y_h$.}
        \end{subfigure}~
        \begin{subfigure}[h!]{0.5\textwidth}
                \includegraphics[trim = 40mm 80mm 30mm 70mm, clip, width=\textwidth]{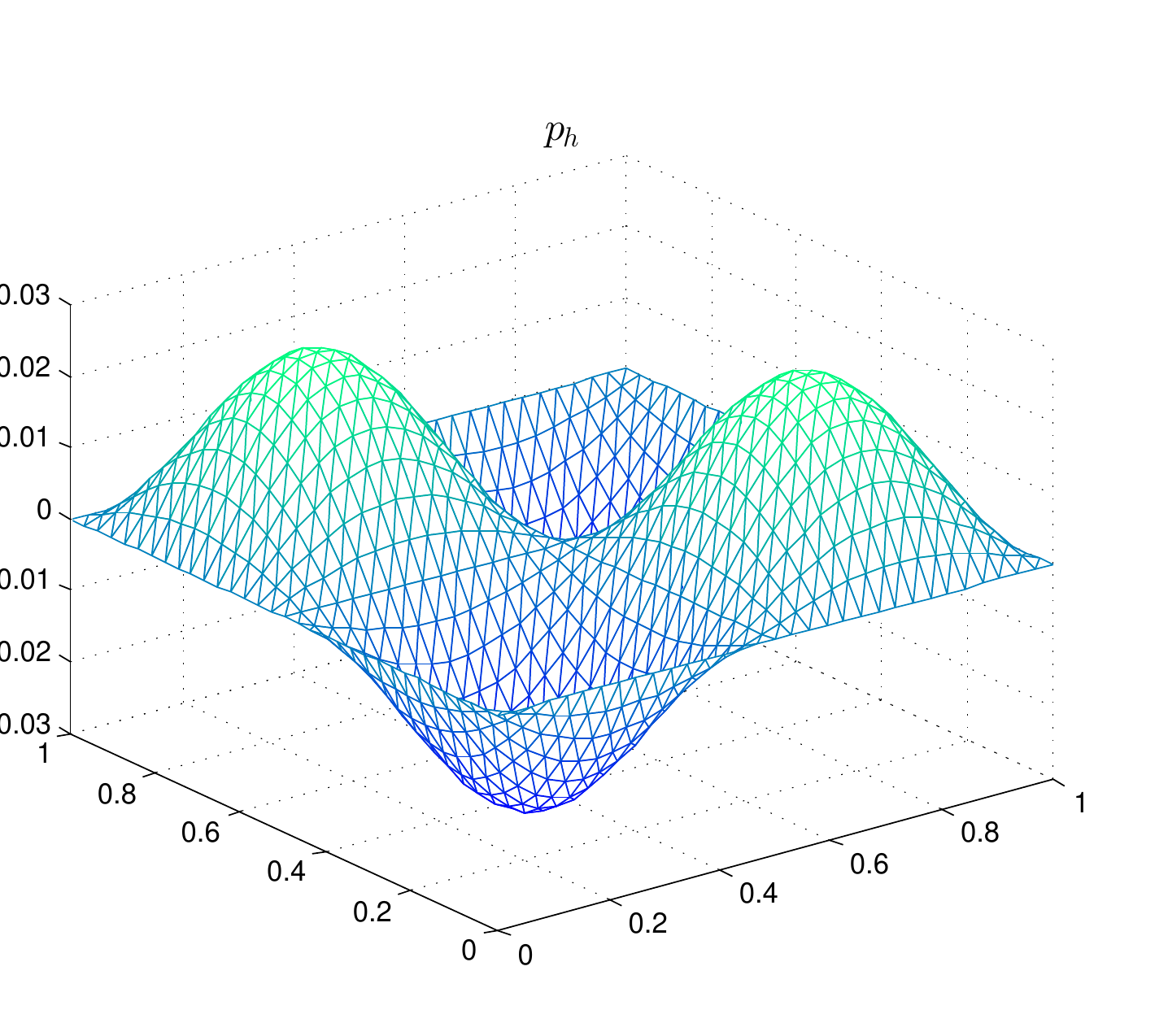}
                \caption{The adjoint state $\bar p_h$.}
        \end{subfigure}
        
        \begin{subfigure}[h!]{0.5\textwidth}
                \includegraphics[trim = 40mm 80mm 30mm 70mm, clip, width=\textwidth]{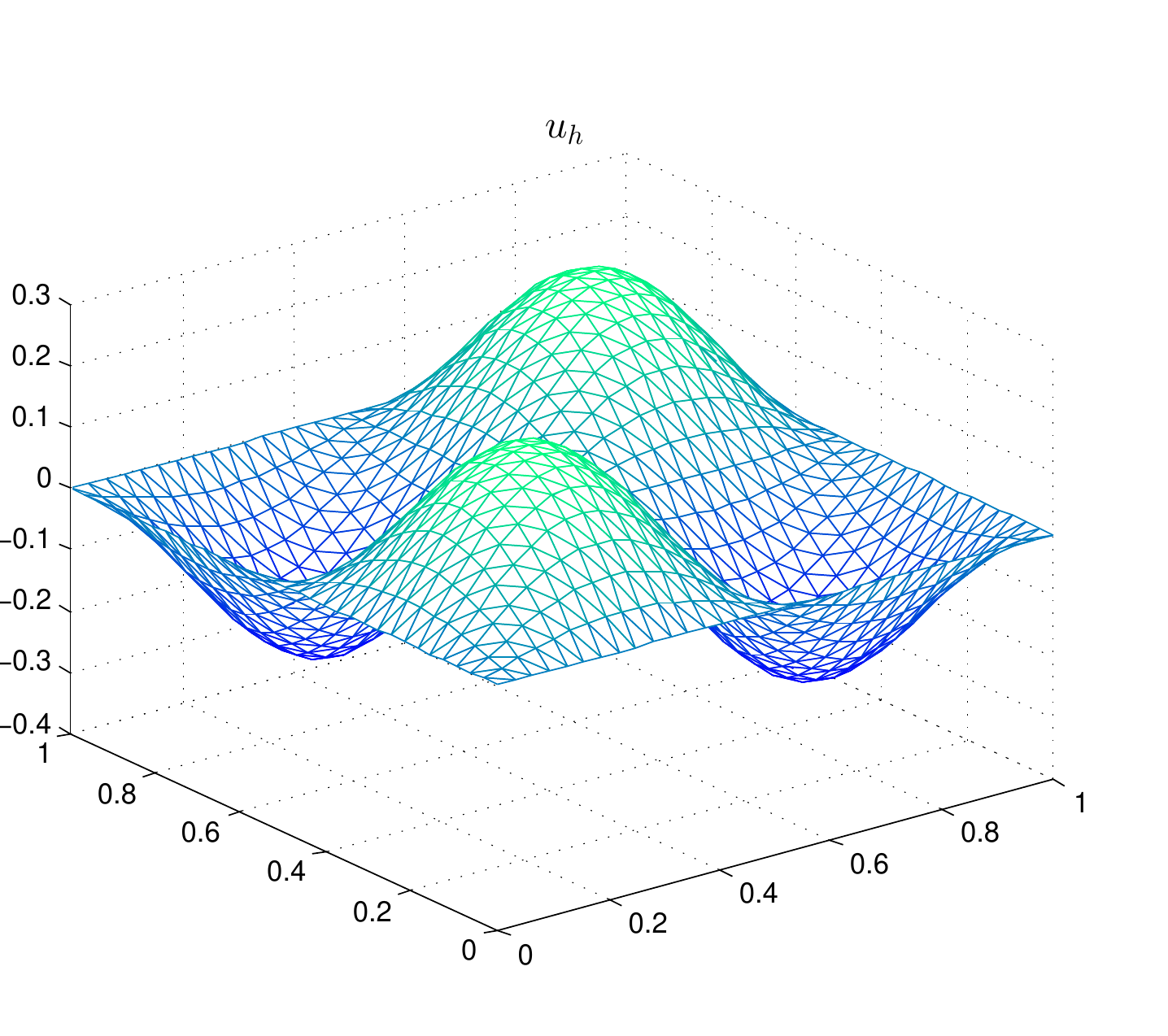}
                \caption{The optimal control $\bar u_h$.}
        \end{subfigure}
        
        \caption{Example~\ref{example: y power 3} Case~3 with choice~\textbf{A1} for $y_0$: The values of $\|\bar p_h\|_{L^4}$, $\eta(\alpha)$ and $J(\bar u_h)$ vs. $\alpha$. The optimal state $\bar y_h$, the optimal control $\bar u_h$ and the adjoint state $\bar p_h$ for $\alpha=10^{-1}$.}
        \label{figure: example y3 case(constrained state) choiceA1}
\end{figure}

\begin{table}[h!]
\caption{Example~\ref{example: y power 3} Case~3 with choice~\textbf{A2} for $y_0$: The values of  $\|\bar p_h\|_{L^4} $, $\eta(\alpha)$ and $J(\bar u_h)$ for different values of $\alpha$.}
\label{table: example y3 case(constrained state) choiceA2}
\begin{tabular}{ l  c  c  c}
\toprule
$\alpha$ &	  $\|\bar p_h\|_{L^4}$&        $\eta(\alpha)$ &   $J(\bar u_h)$  \\
\midrule[1pt]

1.0e-06 &	  8.727496956489e-04 &	 	 6.776197632762e-03 &	 	 1.525635329141e+02   \\
1.0e-05 &	  6.303449080470e-03 &	 	 1.606889689070e-02 &	 	 1.536018906075e+02   \\
1.0e-04 &	  2.143214405409e-02 &	 	 3.810535956559e-02 &	 	 1.559131574621e+02   \\
1.0e-03 &	  8.541044896637e-02 &	 	 9.036204771862e-02 &	 	 1.600259053817e+02   \\
1.0e-02 &	  1.596641521237e-01 &	 	 2.142821839497e-01 &	 	 1.627648901943e+02   \\
1.0e-01 &	  2.997603240217e-01 &	 	 5.081431366100e-01 &	 	 1.642031427088e+02   \\
1.0e+00 &	  3.061090377257e-01 &	 	 1.204997272869e+00 &	 	 1.644198126030e+02   \\
1.0e+01 &	  3.066635772733e-01 &	 	 2.857498848277e+00 &	 	 1.644419766418e+02   \\
1.0e+02 &	  3.067181181971e-01 &	 	 6.776197632762e+00 &	 	 1.644441976184e+02   \\
1.0e+03 &	  3.067235630566e-01 &	 	 1.606889689070e+01 &	 	 1.644444197614e+02   \\

\bottomrule 
\end{tabular}
\end{table}

\begin{figure}[p]
        \centering
        \begin{subfigure}[h!]{0.5\textwidth}
                \includegraphics[trim = 40mm 80mm 30mm 70mm, clip, width=\textwidth]{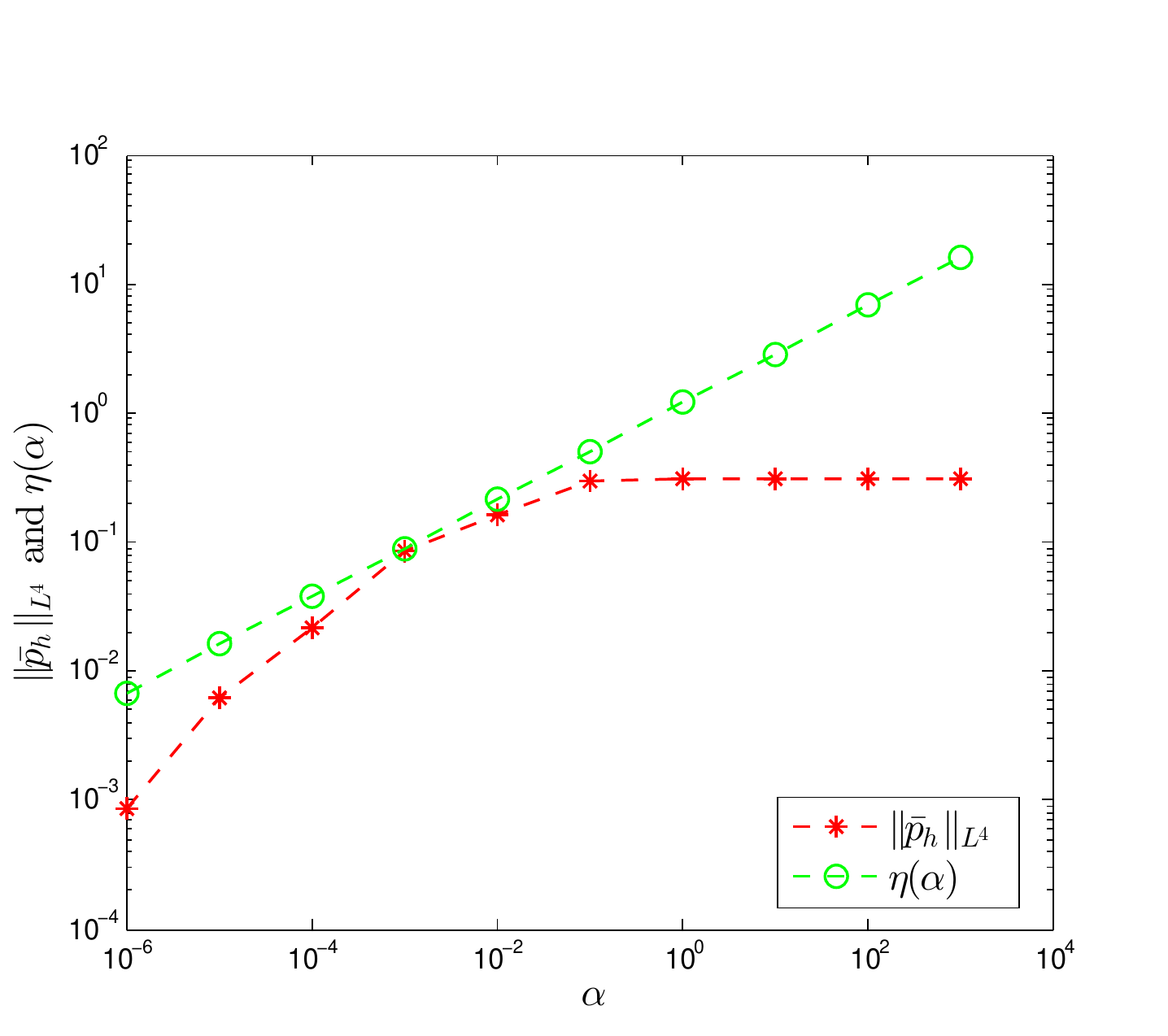}
                \caption{$\|\bar p_h\|_{L^4}$ and $\eta(\alpha)$ vs. $\alpha$.}
        \end{subfigure}%
        ~ 
        \begin{subfigure}[h!]{0.5\textwidth}
                \includegraphics[trim = 30mm 80mm 30mm 70mm, clip, width=\textwidth]{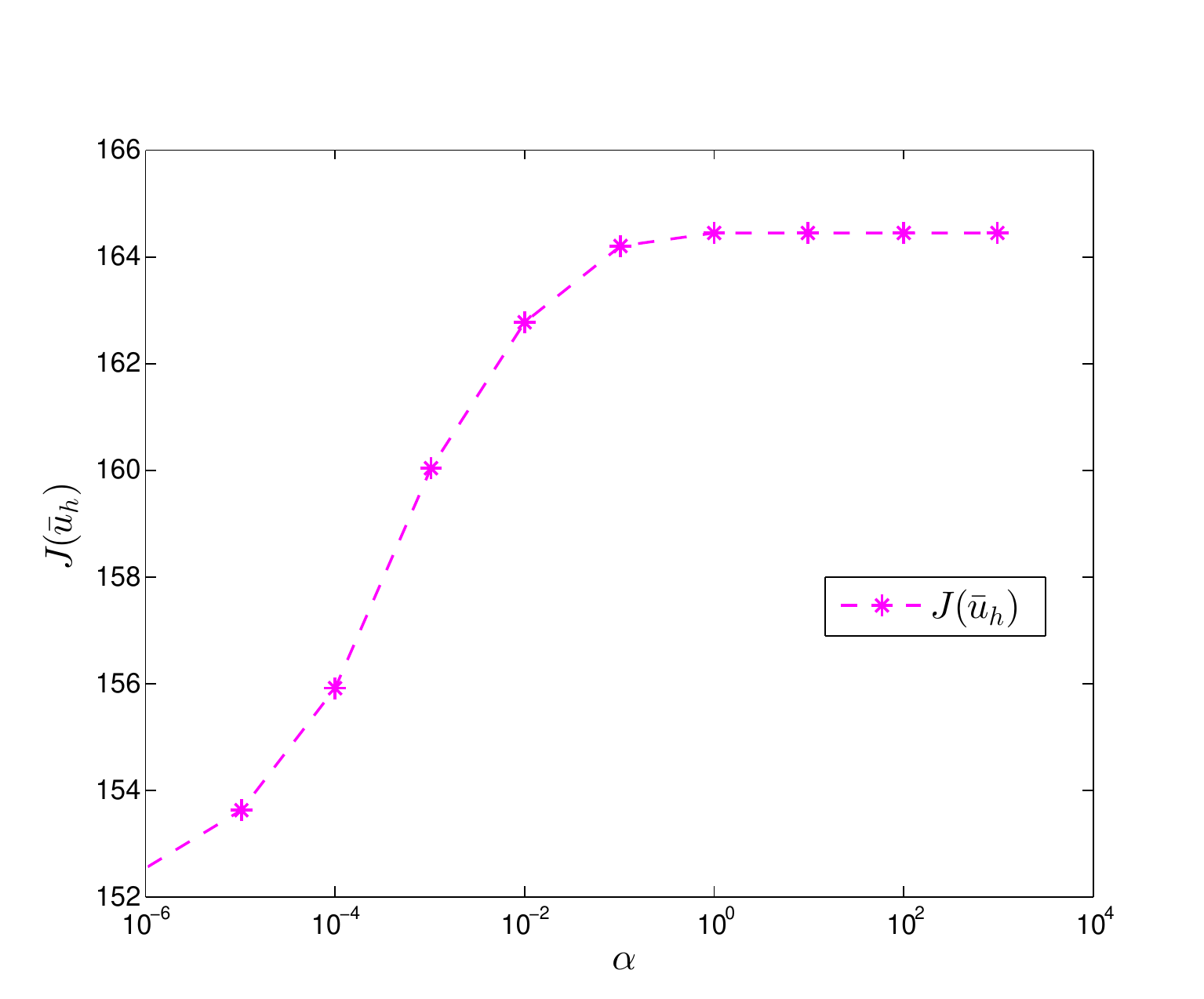}
                \caption{$J(\bar u_h)$  vs. $\alpha$.}
        \end{subfigure}
        
        \begin{subfigure}[h!]{0.5\textwidth}
                \includegraphics[trim = 40mm 80mm 30mm 70mm, clip, width=\textwidth]{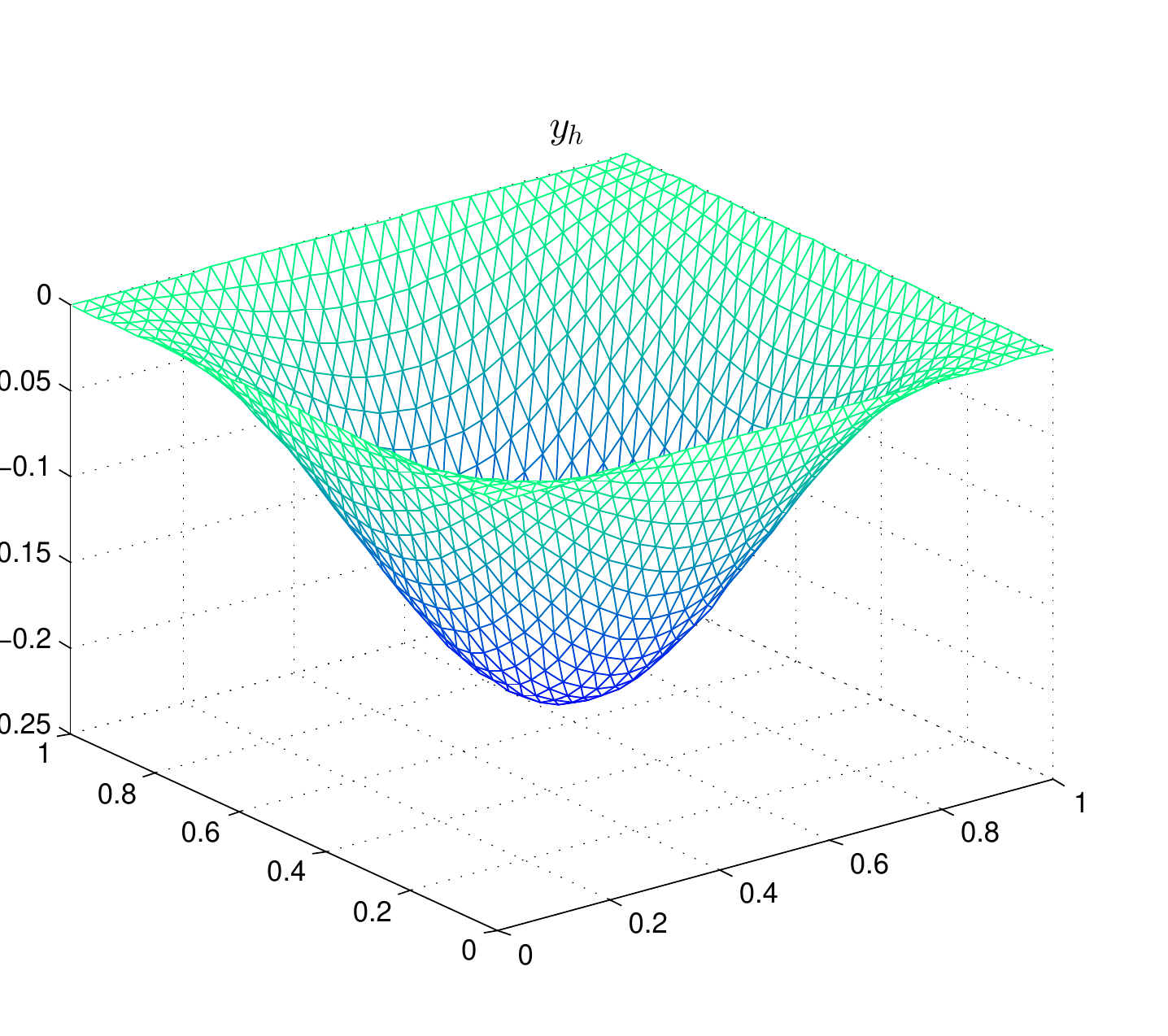}
                \caption{The optimal state $\bar y_h$.}
        \end{subfigure}~
        \begin{subfigure}[h!]{0.5\textwidth}
                \includegraphics[trim = 40mm 80mm 30mm 70mm, clip, width=\textwidth]{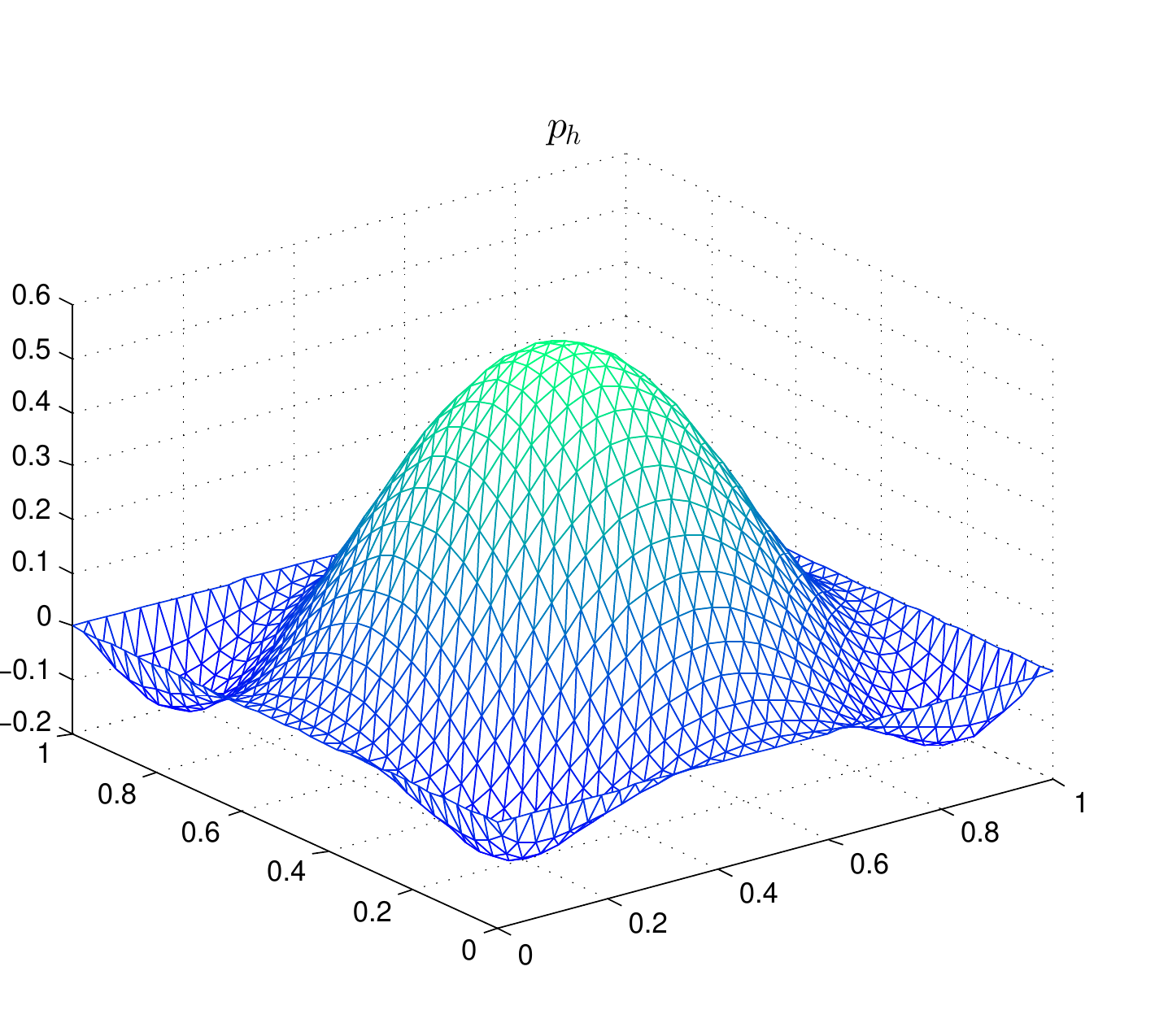}
                \caption{The adjoint state $\bar p_h$.}
        \end{subfigure}
        
        \begin{subfigure}[h!]{0.5\textwidth}
                \includegraphics[trim = 40mm 80mm 30mm 70mm, clip, width=\textwidth]{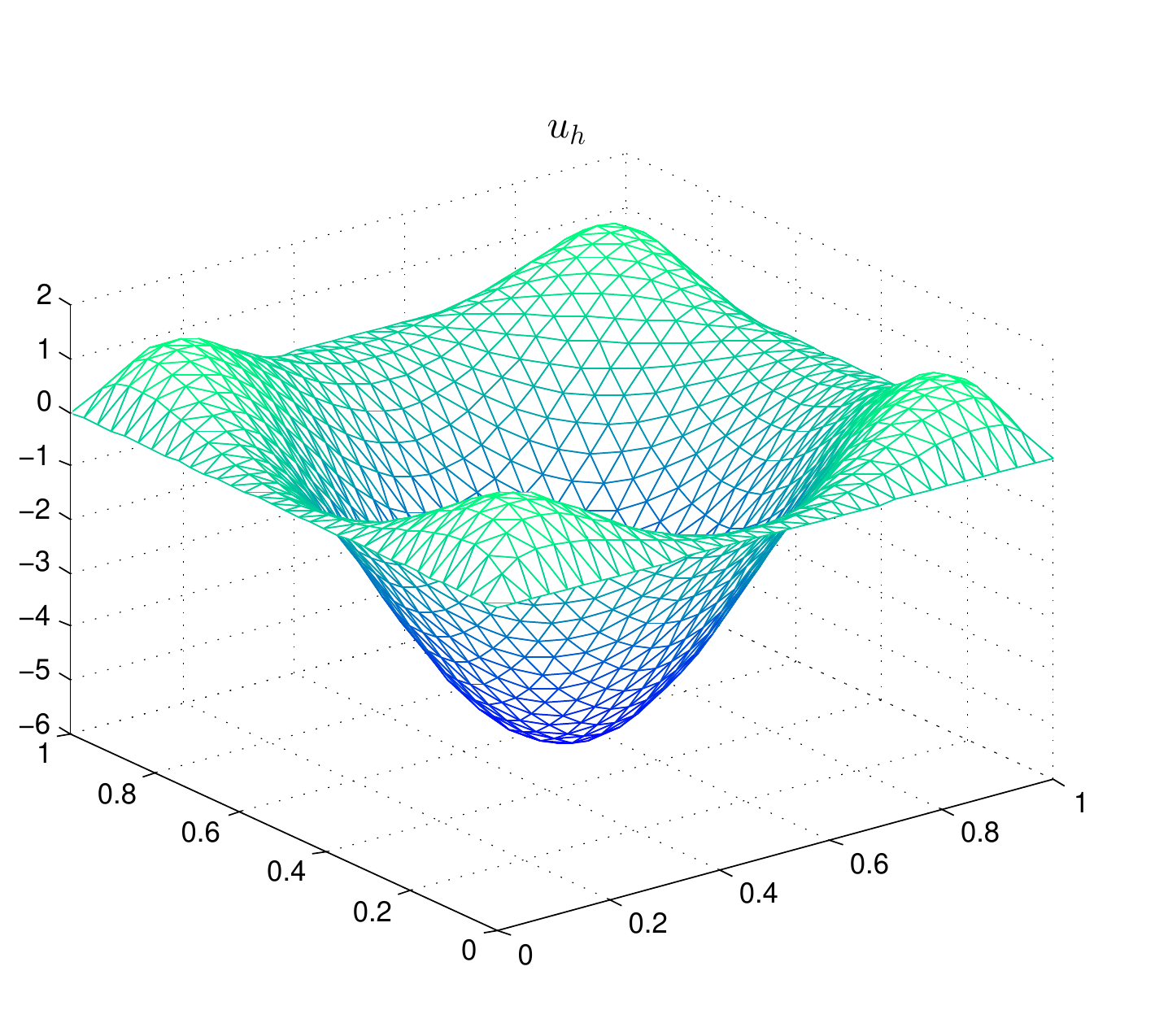}
                \caption{The optimal control $\bar u_h$.}
        \end{subfigure}
        
        \caption{Example~\ref{example: y power 3} Case~3 with choice~\textbf{A2} for $y_0$: The values of $\|\bar p_h\|_{L^4}$, $\eta(\alpha)$ and $J(\bar u_h)$ vs. $\alpha$. The optimal state $\bar y_h$, the optimal control $\bar u_h$ and the adjoint state $\bar p_h$ for $\alpha=10^{-1}$.}
        \label{figure: example y3 case(constrained state) choiceA2}
\end{figure}

\noindent
\textbf{Case~4} 
The following example is taken from  \cite[Section 7]{neitzel2014finite}. In particular, $\phi(s)=s^3$ and
\begin{align*}
u_b &=-u_a=\infty, \\
y_b &=\infty, \\
y_a(x)&=-\frac{2}{3}+\min \bigg( \frac{1}{2}(x_1 + x_2), \frac{1}{2}(1+ x_1 - x_2),  \frac{1}{2}(1- x_1 + x_2),1- \frac{1}{2}(x_1 + x_2) \bigg),\\
y_0 &=-1 \\
\alpha &=10^{-3}.
\end{align*}

The numerical findings for this case are given in Table~\ref{table: example y3 case(NeitzelExample)} and they are illustrated graphically in Figure~\ref{figure: example y3 case(NeitzelExample)}. It is clear that $\bar u_h$ is a global minimum for the given values of $\alpha$.
\begin{table}[p]
\caption{Example~\ref{example: y power 3}  Case~4: The values of  $\|\bar p_h\|_{L^4} $, $\eta(\alpha)$ and $J(\bar u_h)$ for different values of $\alpha$.}
\label{table: example y3 case(NeitzelExample)}
\begin{tabular}{ l  c  c  c}
\toprule
$\alpha$ &	  $\|\bar p_h\|_{L^4}$&        $\eta(\alpha)$ &   $J(\bar u_h)$  \\
\midrule[1pt]

1.0e-06 &	  1.961933031441e-04 &	 	 6.776197632762e-03 &	 	 2.143984056211e-01   \\
1.0e-05 &	  7.663887131231e-04 &	 	 1.606889689070e-02 &	 	 2.410556714493e-01   \\
1.0e-04 &	  2.844056064106e-03 &	 	 3.810535956559e-02 &	 	 2.890783107664e-01   \\
1.0e-03 &	  1.055630139945e-02 &	 	 9.036204771862e-02 &	 	 3.690000948128e-01   \\
1.0e-02 &	  2.397197977885e-02 &	 	 2.142821839497e-01 &	 	 4.449373232494e-01   \\
1.0e-01 &	  4.706175447556e-02 &	 	 5.081431366100e-01 &	 	 4.917394785652e-01   \\
1.0e+00 &	  4.818113594926e-02 &	 	 1.204997272869e+00 &	 	 4.991551130306e-01   \\
1.0e+01 &	  4.829535470702e-02 &	 	 2.857498848277e+00 &	 	 4.999153188201e-01   \\
1.0e+02 &	  4.830679945384e-02 &	 	 6.776197632762e+00 &	 	 4.999915299530e-01   \\
1.0e+03 &	  4.830794415727e-02 &	 	 1.606889689070e+01 &	 	 4.999991529760e-01   \\

\bottomrule 
\end{tabular}
\end{table}

\begin{figure}[p]
        \centering
        \begin{subfigure}[h!]{0.5\textwidth}
                \includegraphics[trim = 40mm 80mm 30mm 70mm, clip, width=\textwidth]{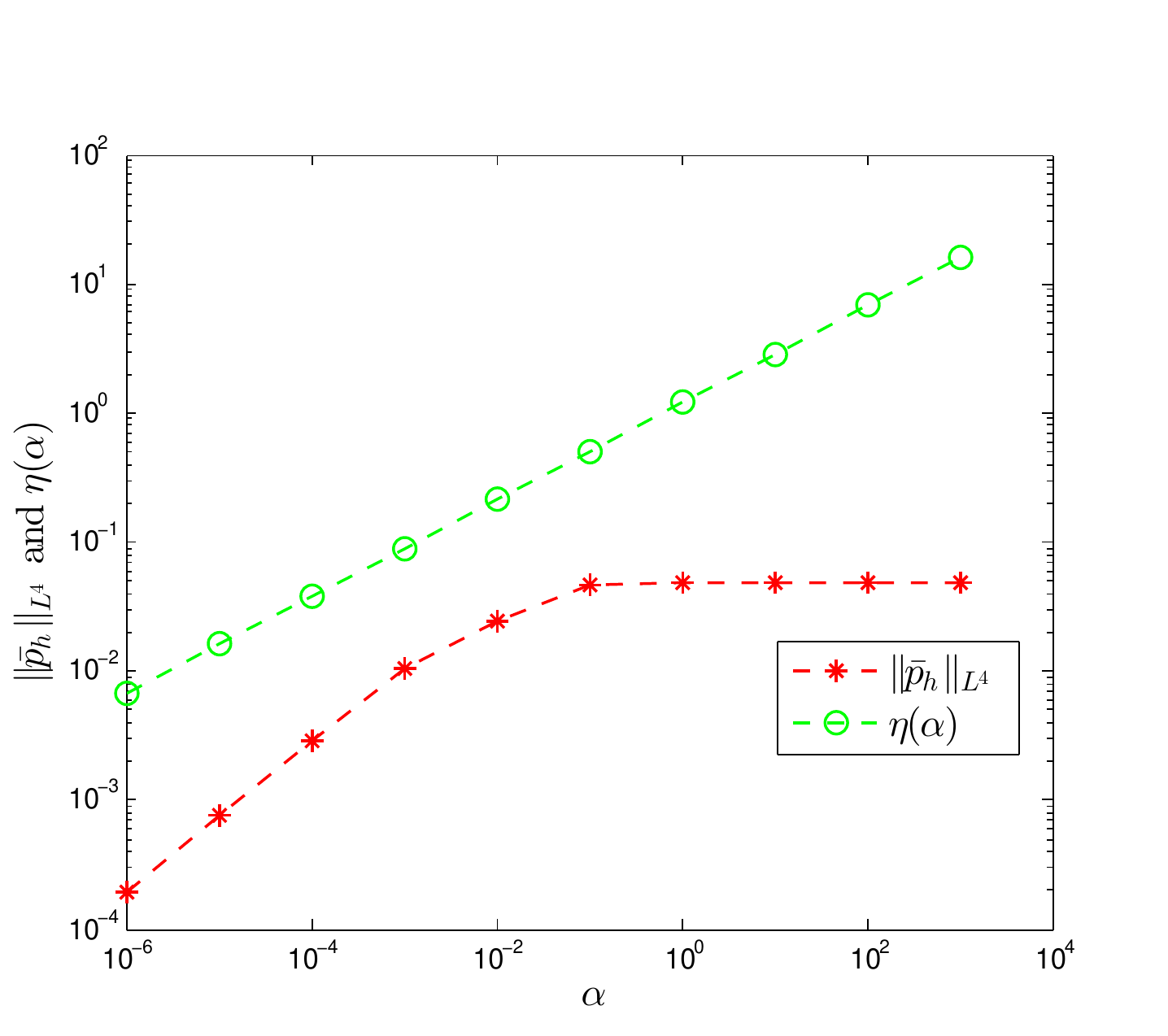}
                \caption{$\|\bar p_h\|_{L^4}$ and $\eta(\alpha)$ vs. $\alpha$.}
        \end{subfigure}%
        ~ 
        \begin{subfigure}[h!]{0.5\textwidth}
                \includegraphics[trim = 35mm 80mm 30mm 70mm, clip, width=\textwidth]{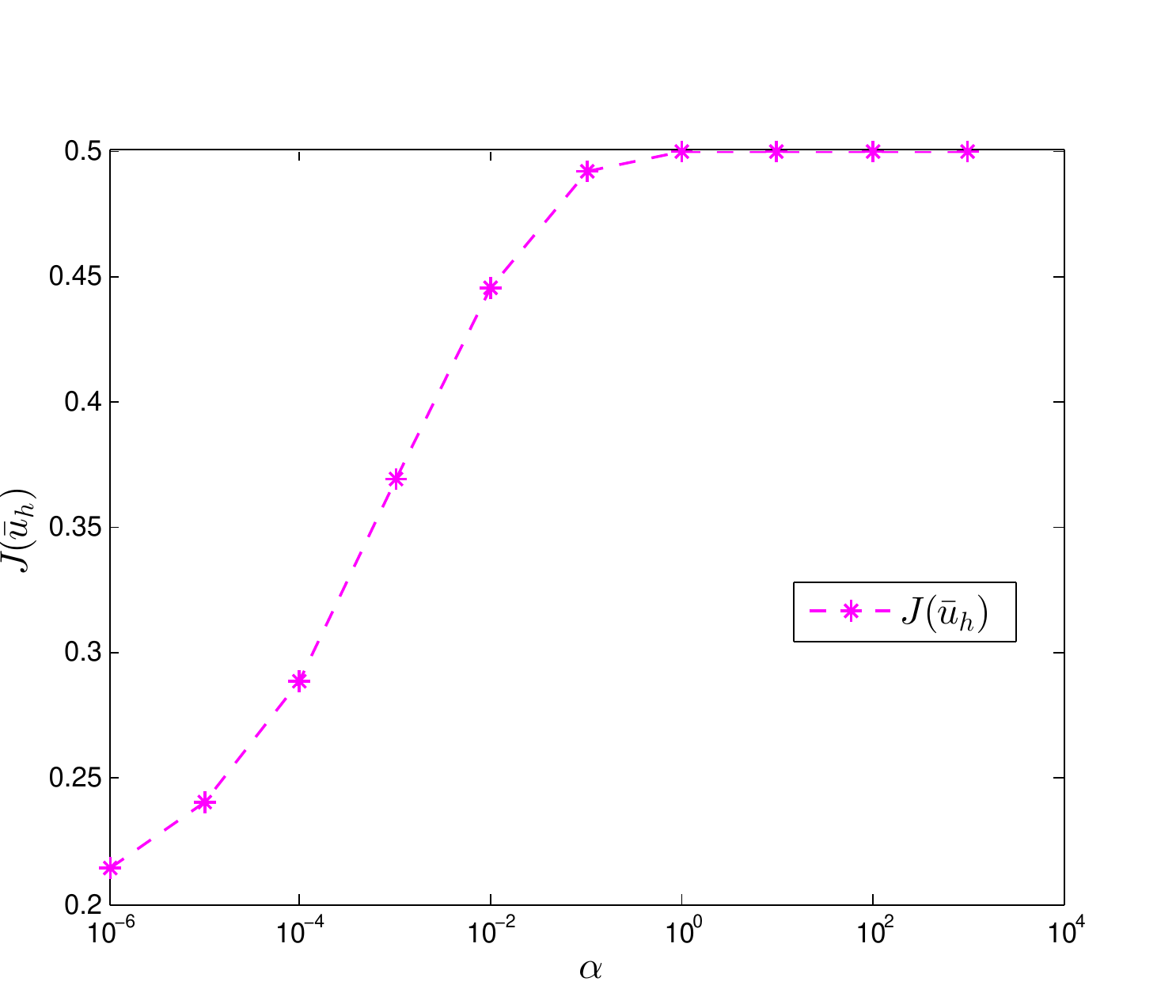}
                \caption{$J(\bar u_h)$  vs. $\alpha$.}
        \end{subfigure}
        
        \begin{subfigure}[h!]{0.5\textwidth}
                  \includegraphics[trim = 40mm 80mm 30mm 70mm, clip, width=\textwidth]{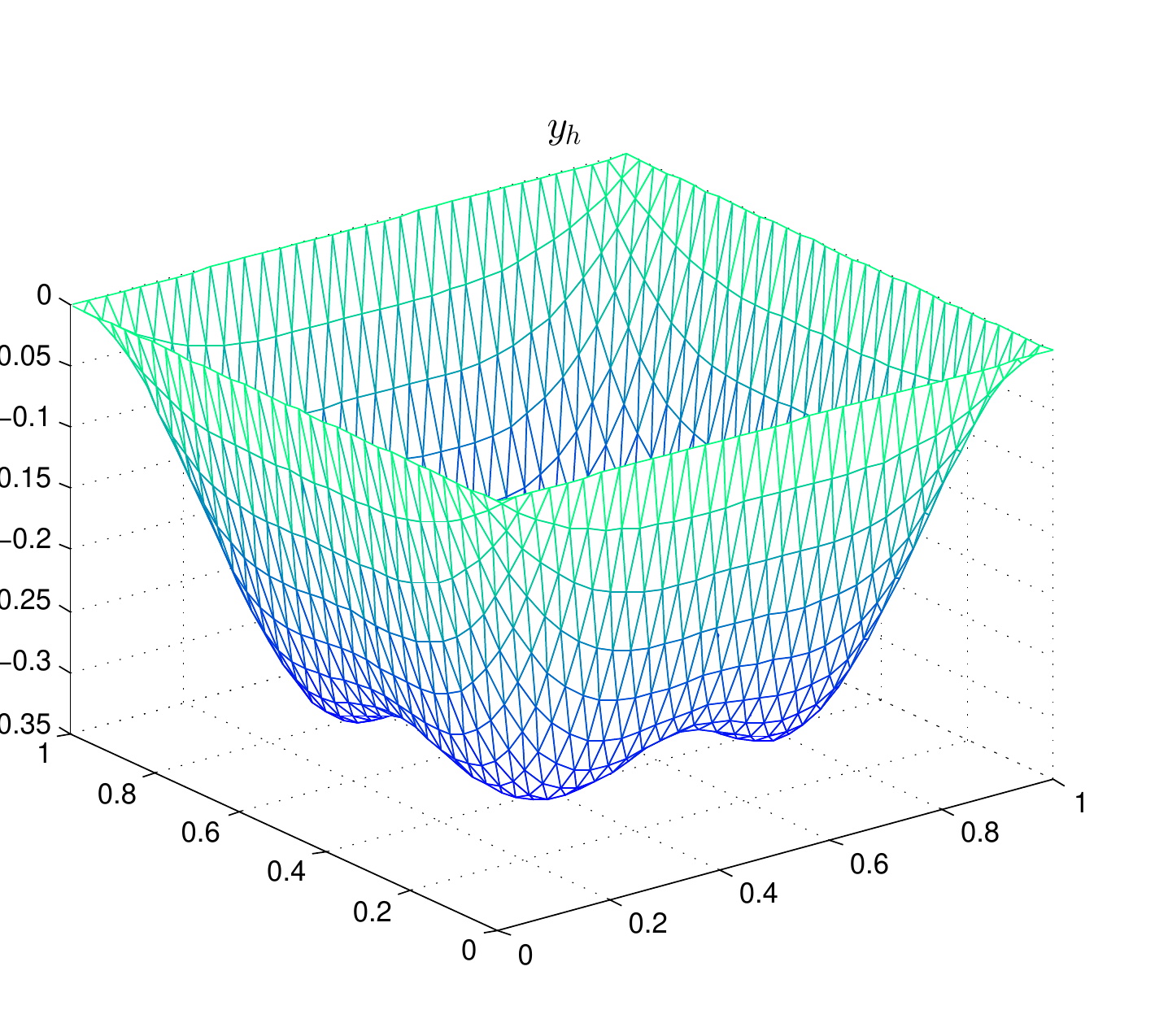}
                \caption{The optimal state $\bar y_h$.}
        \end{subfigure}~
        \begin{subfigure}[h!]{0.5\textwidth}
                  \includegraphics[trim = 40mm 80mm 30mm 70mm, clip, width=\textwidth]{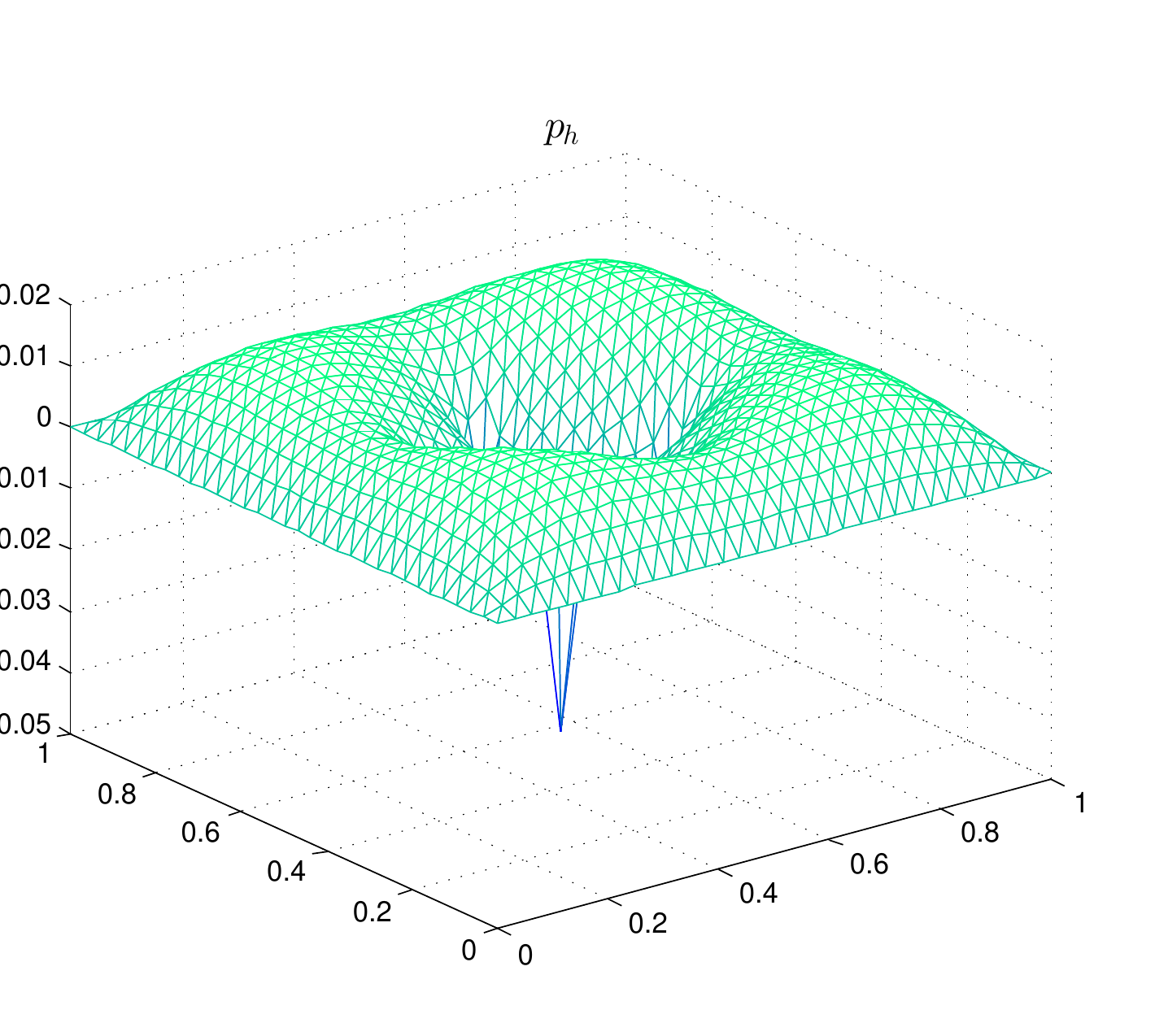}
                \caption{The adjoint state $\bar p_h$.}
        \end{subfigure}
        
        \begin{subfigure}[h!]{0.5\textwidth}
                  \includegraphics[trim = 40mm 80mm 30mm 70mm, clip, width=\textwidth]{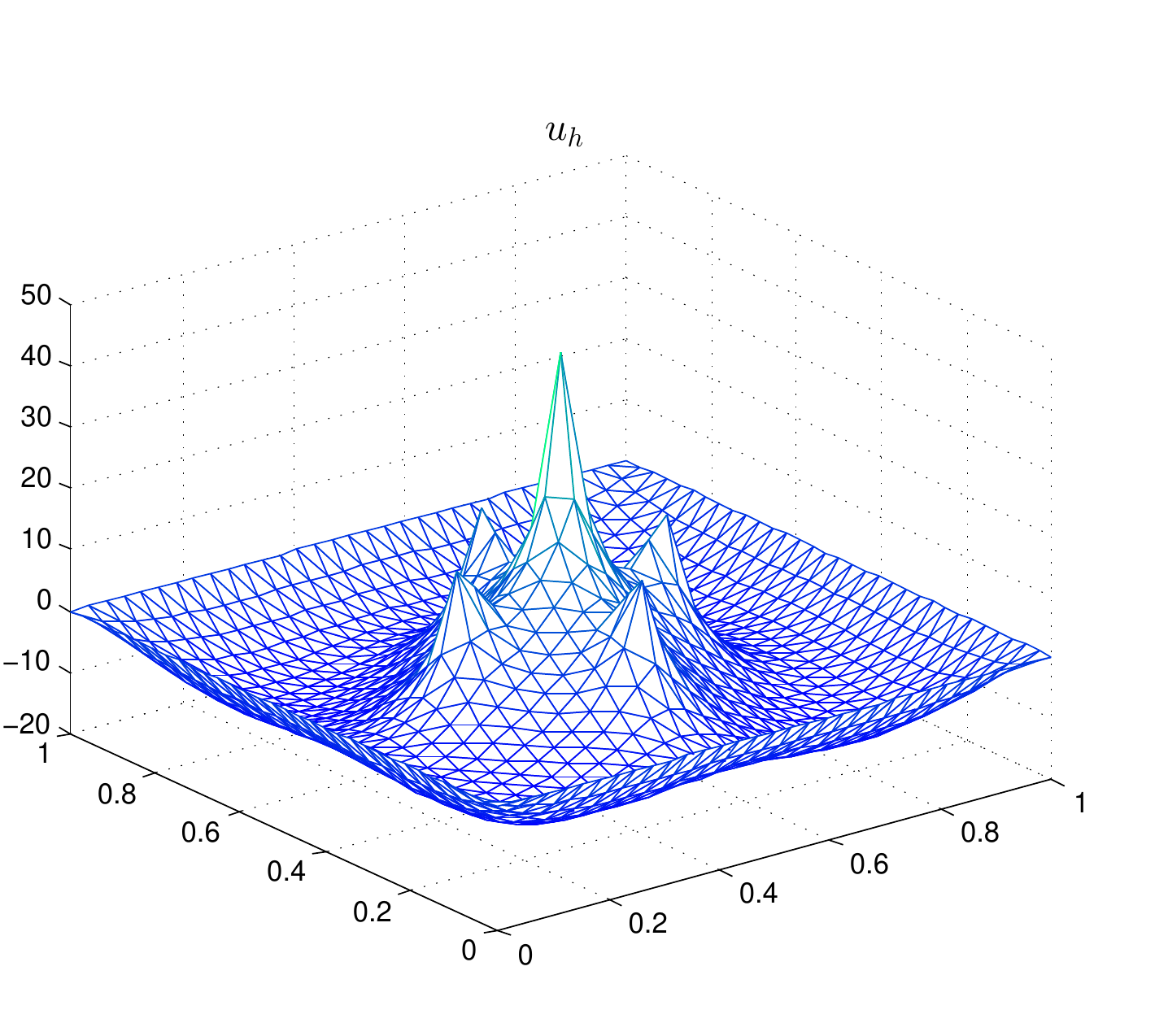}
                \caption{The optimal control $\bar u_h$.}
        \end{subfigure}~
        \begin{subfigure}[h!]{0.5\textwidth}
                  \includegraphics[trim = 40mm 80mm 30mm 70mm, clip, width=\textwidth]{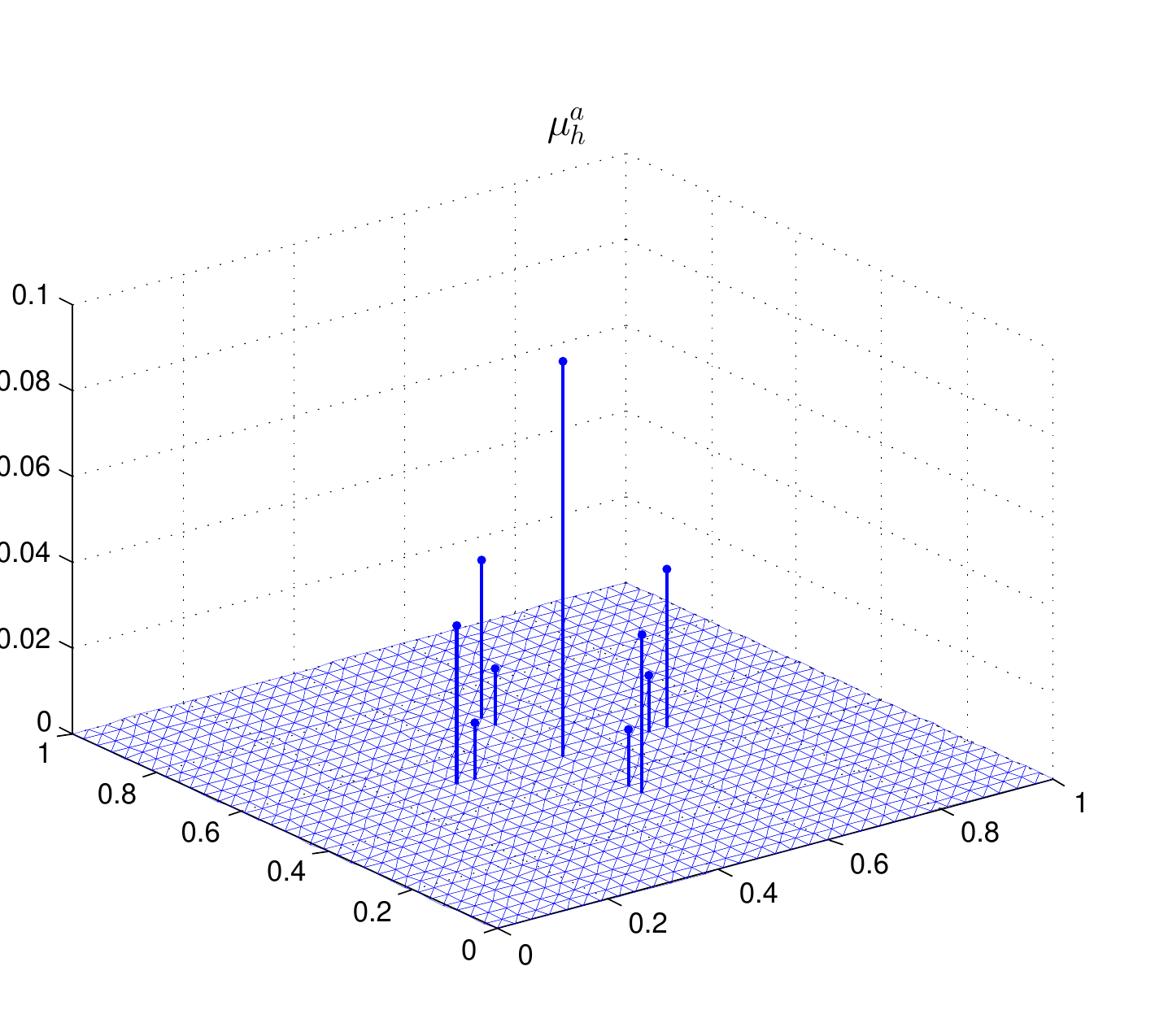}
                \caption{The multiplier $\bar \mu^a_h$.}
        \end{subfigure}

        \caption{Example~\ref{example: y power 3} Case~4: The values of $\|\bar p_h\|_{L^4}$, $\eta(\alpha)$ and $J(\bar u_h)$ vs. $\alpha$. The optimal state $\bar y_h$, the optimal control $\bar u_h$, the adjoint state $\bar p_h$ and the multiplier $\bar \mu^a_h$ for $\alpha=10^{-3}$.}
        \label{figure: example y3 case(NeitzelExample)}
\end{figure}


\end{example}

\newpage
\begin{example}
\label{example: y power 5}
In this example we define $\phi(s):=s^5$. We see that Assumption 1 is satisfied with 
\[
r=\frac{4}{3} \quad \mbox{ and } \quad M=\frac{20}{5^{\frac{3}{4}}}.
\]
Hence, in view of Theorem~\ref{theorem: 2} we have $q=6$ and a control $\bar u_h$ obtained from solving (\ref{equation: 4})-(\ref{inequality: 2}) is a global minimum if the associated adjoint state $\bar p_h$ satisfies 
\[
\|\bar p_h\|_{L^6(\Omega)} \leq \frac{11^{\frac{11}{24}}}{ 13^{\frac{13}{24}} 2^{\frac{1}{6}}   \sqrt{3}}   C^{-\frac{1}{2}}_6 \alpha^{\frac{11}{24}},
\]
where $C^{-\frac{1}{2}}_6 \approx 1.271251384316953$ is the constant from Lemma~\ref{lemma: 2}. For this example we consider the following three cases. 
We abbreviate
\[
\eta(\alpha):= \eta(\alpha,\frac{4}{3})= \frac{11^{\frac{11}{24}}}{ 13^{\frac{13}{24}} 2^{\frac{1}{6}}   \sqrt{3}}   C^{-\frac{1}{2}}_6 \alpha^{\frac{11}{24}}.
\] 
\noindent
\textbf{Case~1} (unconstrained problem)
In this case we set
\begin{align*}
u_b &=-u_a=\infty, \\
y_b &=-y_a=\infty. \\
\end{align*}
The values of $\|\bar p_h\|_{L^6} $, $\eta(\alpha)$ and $J(\bar u_h)$ for different values of $\alpha$ with choice~\textbf{A1} for $y_0$ are given in Table~\ref{table: example y5 case(unconstrained) choiceA1}.  The findings are illustrated graphically in Figure~\ref{figure: example y5 case(unconstrained) choiceA1}. We see that $\bar u_h$ is a global minimum for all values of $\alpha$ since $\|\bar p_h\|_{L^6} $ is less than its corresponding  $\eta(\alpha)$. On the other hand, with choice~\textbf{A2} for $y_0$ we can claim that $\bar u_h$ is a global minimum only for approximately $\alpha$ greater than $1$ as it  can be seen from Figure~\ref{figure: example y5 case(unconstrained) choiceA2}. The numerical values are provided in Table~\ref{table: example y5 case(unconstrained) choiceA2}.

\begin{table}[p]
\caption{Example~\ref{example: y power 5}  Case~1 with choice~\textbf{A1} for $y_0$: The values of  $\|\bar p_h\|_{L^6} $, $\eta(\alpha)$ and $J(\bar u_h)$ for different values of $\alpha$.}
\label{table: example y5 case(unconstrained) choiceA1}
\begin{tabular}{ l  c  c  c}
\toprule
$\alpha$ &	  $\|\bar p_h\|_{L^6}$&        $\eta(\alpha)$ &   $J(\bar u_h)$  \\
\midrule[1pt]

1.0e-06 &	  1.179795342411e-04 &	 	 8.697974773247e-04 &	 	 3.663839269975e-03   \\ 
1.0e-05 &	  1.040717291260e-03 &	 	 2.498914960443e-03 &	 	 3.314332555914e-02   \\
1.0e-04 &	  6.486412414763e-03 &	 	 7.179344781194e-03 &	 	 1.967178952607e-01   \\
1.0e-03 &	  1.467650352720e-02 &	 	 2.062614866979e-02 &	 	 4.320253853445e-01   \\
1.0e-02 &	  1.672495487678e-02 &	 	 5.925861229879e-02 &	 	 4.922543706340e-01   \\
1.0e-01 &	  1.696149575588e-02 &	 	 1.702490943800e-01 &	 	 4.992144828609e-01   \\
1.0e+00 &	  1.698552077353e-02 &	 	 4.891230660460e-01 &	 	 4.999213370331e-01   \\
1.0e+01 &	  1.698792705311e-02 &	 	 1.405243150394e+00 &	 	 4.999921325890e-01   \\
1.0e+02 &	  1.698816771892e-02 &	 	 4.037242258255e+00 &	 	 4.999992132478e-01   \\
1.0e+03 &	  1.698819178587e-02 &	 	 1.159893577654e+01 &	 	 4.999999213247e-01   \\

\bottomrule 
\end{tabular}
\end{table}

\begin{figure}[p]
        \centering
        \begin{subfigure}[h!]{0.5\textwidth}
                \includegraphics[trim = 40mm 80mm 30mm 70mm, clip, width=\textwidth]{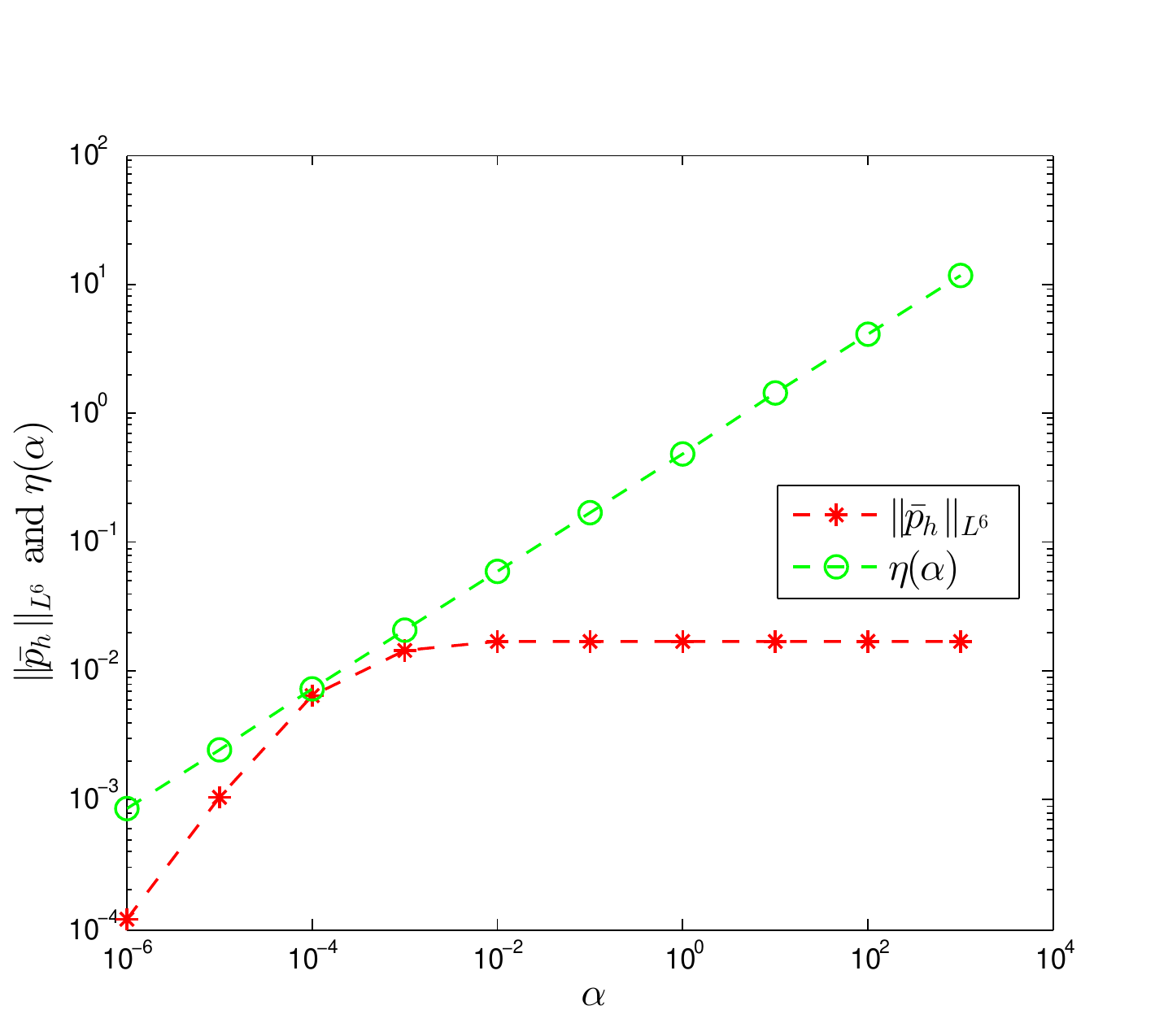}
                \caption{$\|\bar p_h\|_{L^6}$ and $\eta(\alpha)$ vs. $\alpha$.}
        \end{subfigure}%
        ~ 
        \begin{subfigure}[h!]{0.5\textwidth}
                \includegraphics[trim = 35mm 80mm 30mm 70mm, clip, width=\textwidth]{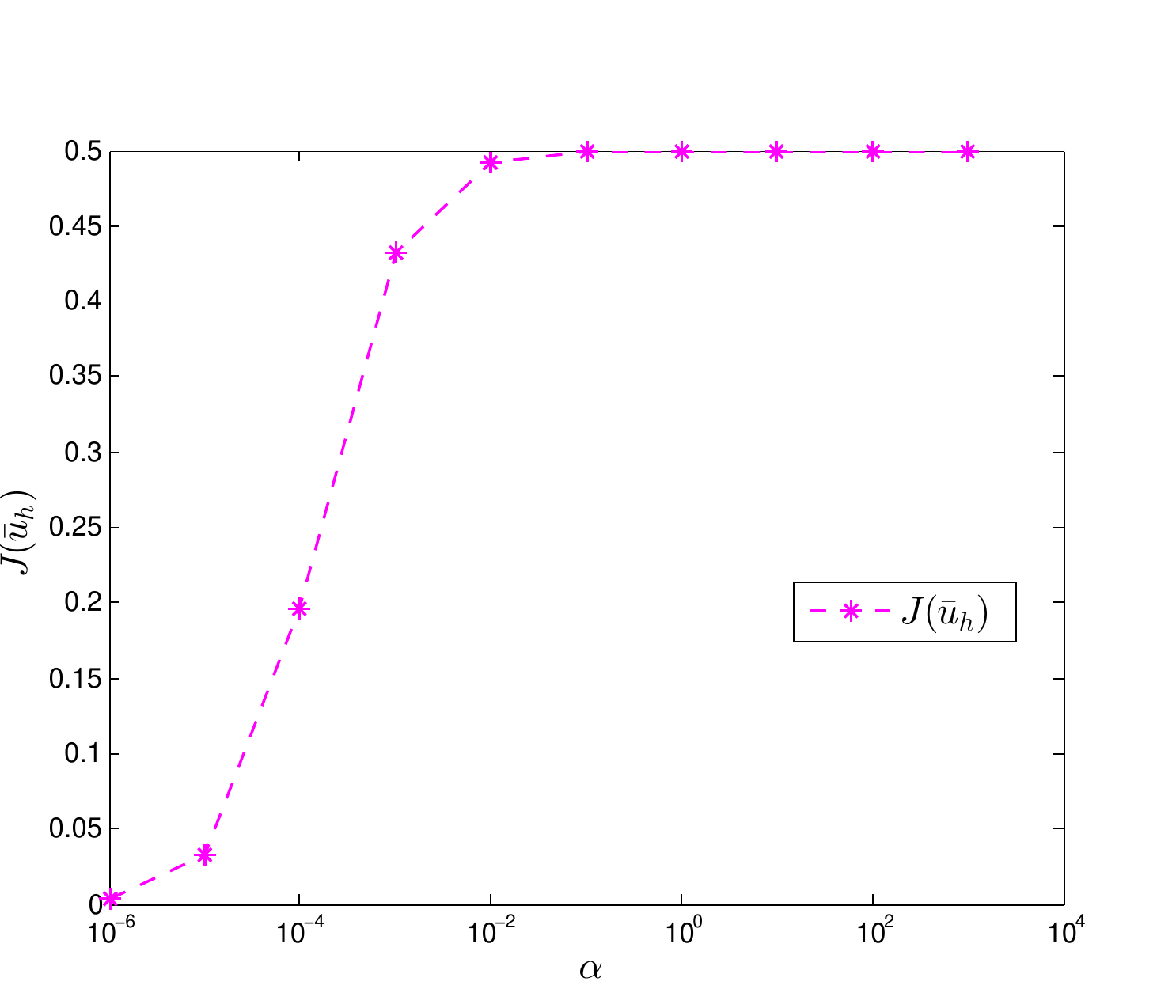}
                \caption{$J(\bar u_h)$  vs. $\alpha$.}
        \end{subfigure}
        
        \begin{subfigure}[h!]{0.5\textwidth}
                  \includegraphics[trim = 40mm 80mm 30mm 70mm, clip, width=\textwidth]{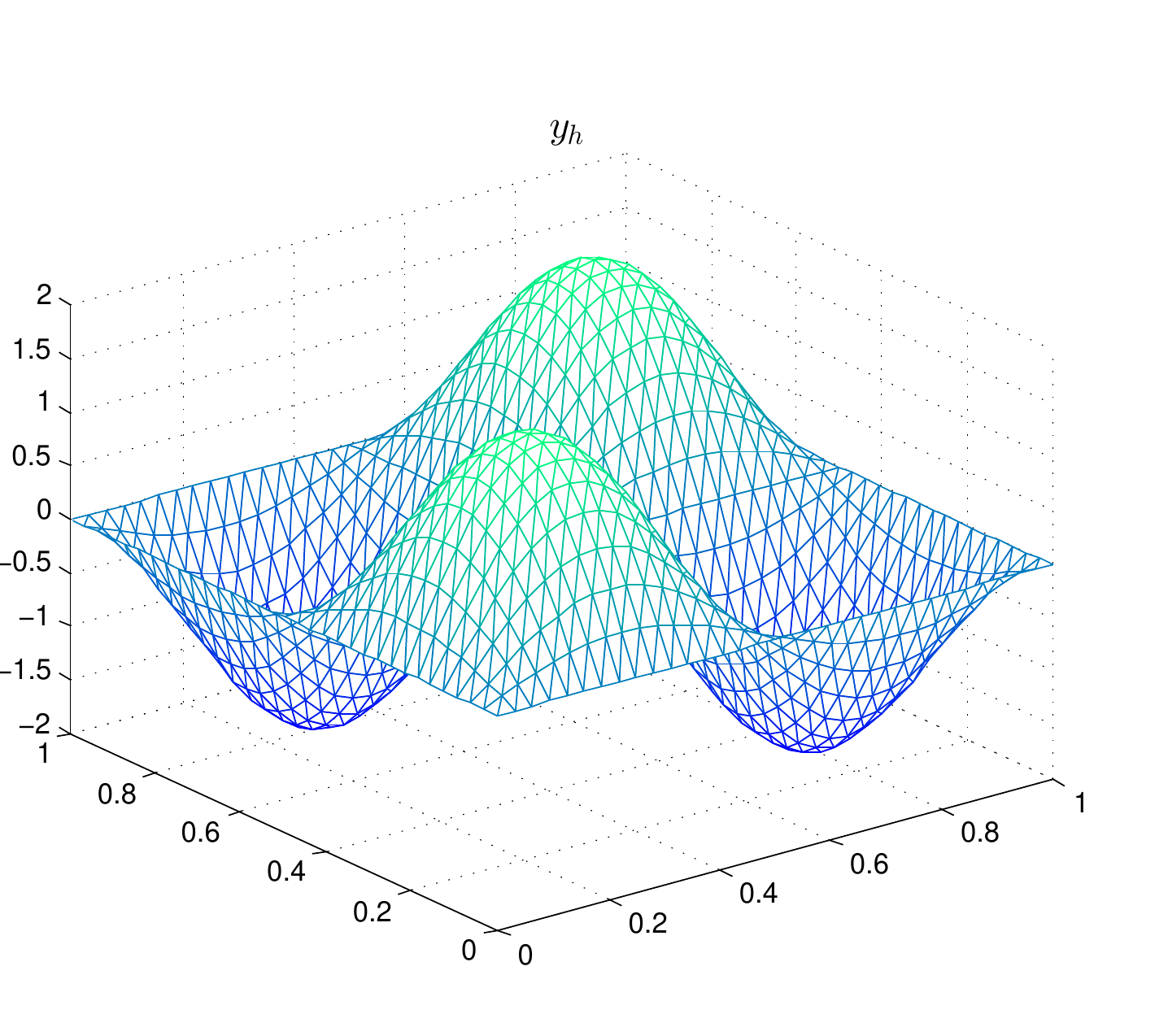}
                \caption{The optimal state $\bar y_h$.}
        \end{subfigure}~
        \begin{subfigure}[h!]{0.5\textwidth}
                  \includegraphics[trim = 40mm 80mm 30mm 70mm, clip, width=\textwidth]{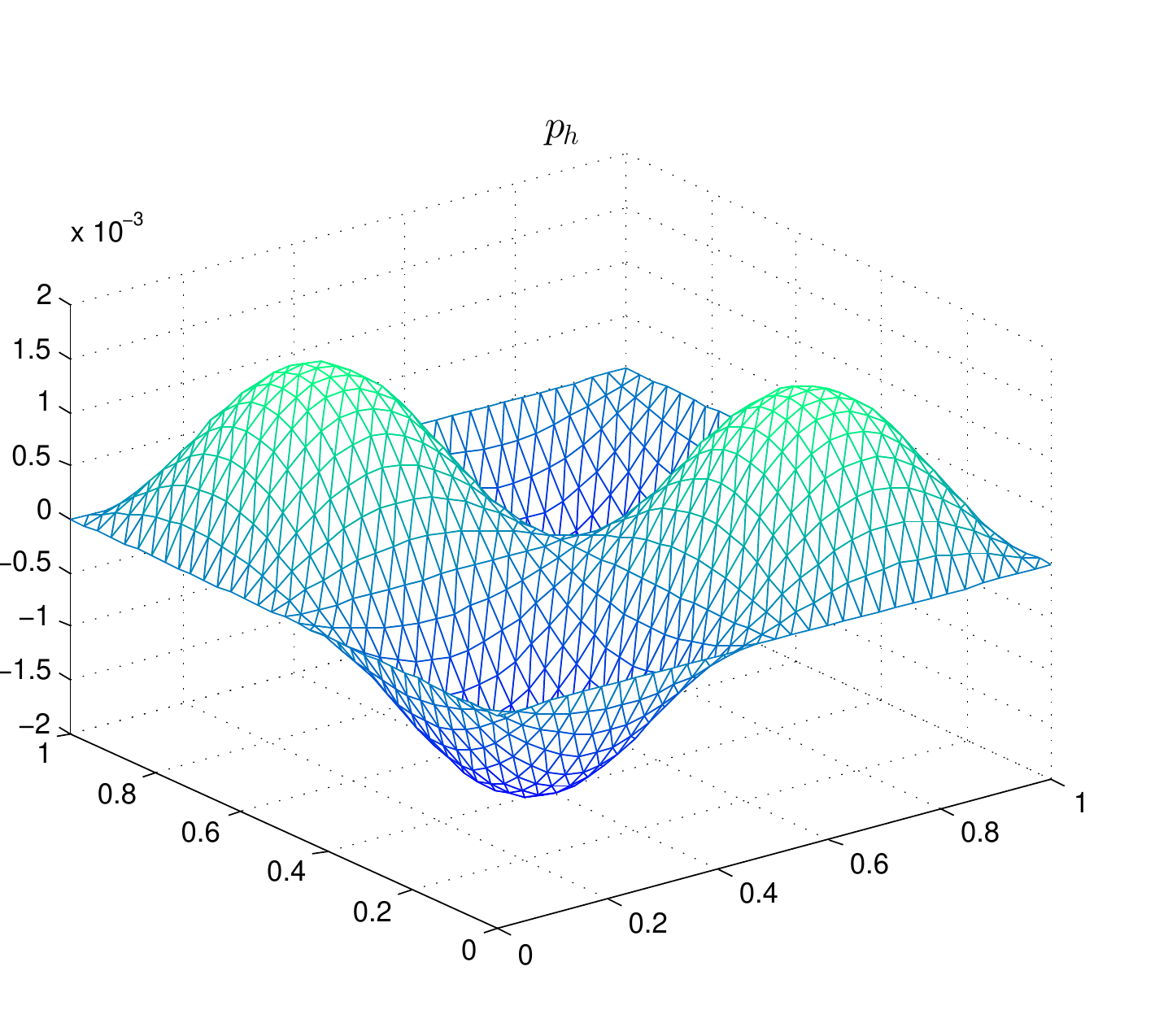}
                \caption{The adjoint state $\bar p_h$.}
        \end{subfigure}
        
        \begin{subfigure}[h!]{0.5\textwidth}
                  \includegraphics[trim = 40mm 80mm 30mm 70mm, clip, width=\textwidth]{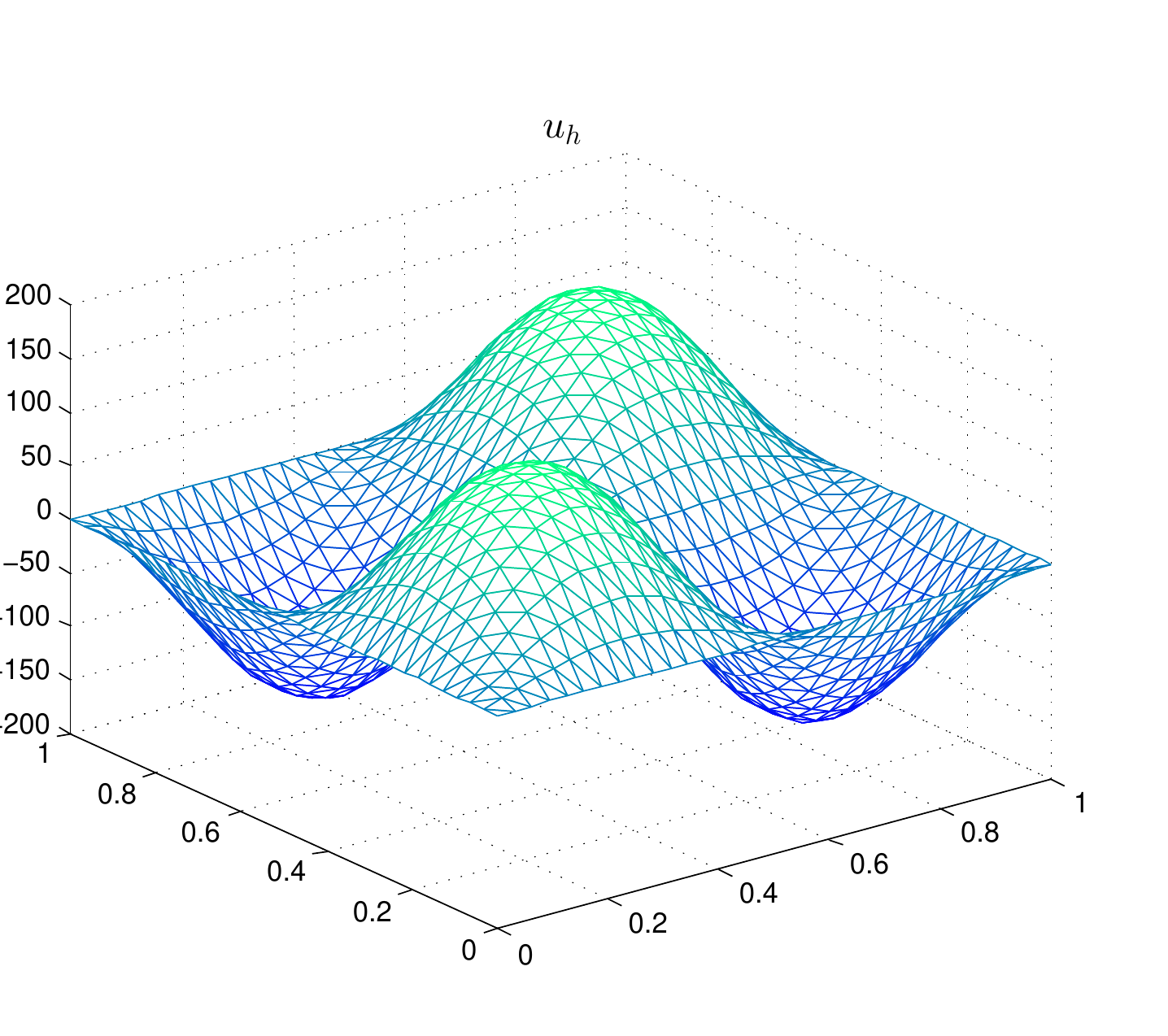}
                \caption{The optimal control $\bar u_h$.}
        \end{subfigure}

        \caption{Example~\ref{example: y power 5} Case~1 with choice~\textbf{A1} for $y_0$: The values of $\|\bar p_h\|_{L^6}$, $\eta(\alpha)$ and $J(\bar u_h)$ vs. $\alpha$. The optimal state $\bar y_h$, the optimal control $\bar u_h$ and the adjoint state $\bar p_h$ for $\alpha=10^{-5}$.}
        \label{figure: example y5 case(unconstrained) choiceA1}
\end{figure}

\begin{table}[p]
\caption{Example~\ref{example: y power 5} Case~1 with choice~\textbf{A2} for $y_0$: The values of  $\|\bar p_h\|_{L^6} $, $\eta(\alpha)$ and $J(\bar u_h)$ for different values of $\alpha$.}
\label{table: example y5 case(unconstrained) choiceA2}
\begin{tabular}{ l  c  c  c}
\toprule
$\alpha$ &	  $\|\bar p_h\|_{L^6}$&        $\eta(\alpha)$ &   $J(\bar u_h)$  \\
\midrule[1pt]

1.0e-06 &	  5.510426875132e-03 &	 	 8.697974773247e-04 &	 	 1.185192313978e+02   \\
1.0e-05 &	  1.587525748968e-02 &	 	 2.498914960443e-03 &	 	 1.331807740335e+02   \\
1.0e-04 &	  4.474831409415e-02 &	 	 7.179344781194e-03 &	 	 1.473322027953e+02   \\
1.0e-03 &	  1.039480114464e-01 &	 	 2.062614866979e-02 &	 	 1.584387338104e+02   \\
1.0e-02 &	  2.428391864045e-01 &	 	 5.925861229879e-02 &	 	 1.626178840362e+02   \\
1.0e-01 &	  3.493646725426e-01 &	 	 1.702490943800e-01 &	 	 1.642025836782e+02   \\
1.0e+00 &	  3.554038724369e-01 &	 	 4.891230660460e-01 &	 	 1.644198119684e+02   \\
1.0e+01 &	  3.560155910725e-01 &	 	 1.405243150394e+00 &	 	 1.644419766411e+02   \\
1.0e+02 &	  3.560769159456e-01 &	 	 4.037242258255e+00 &	 	 1.644441976184e+02   \\
1.0e+03 &	  3.560830499750e-01 &	 	 1.159893577654e+01 &	 	 1.644444197614e+02   \\

\bottomrule 
\end{tabular}
\end{table}

\begin{figure}[p]
        \centering
        \begin{subfigure}[h!]{0.5\textwidth}
                \includegraphics[trim = 40mm 80mm 30mm 70mm, clip, width=\textwidth]{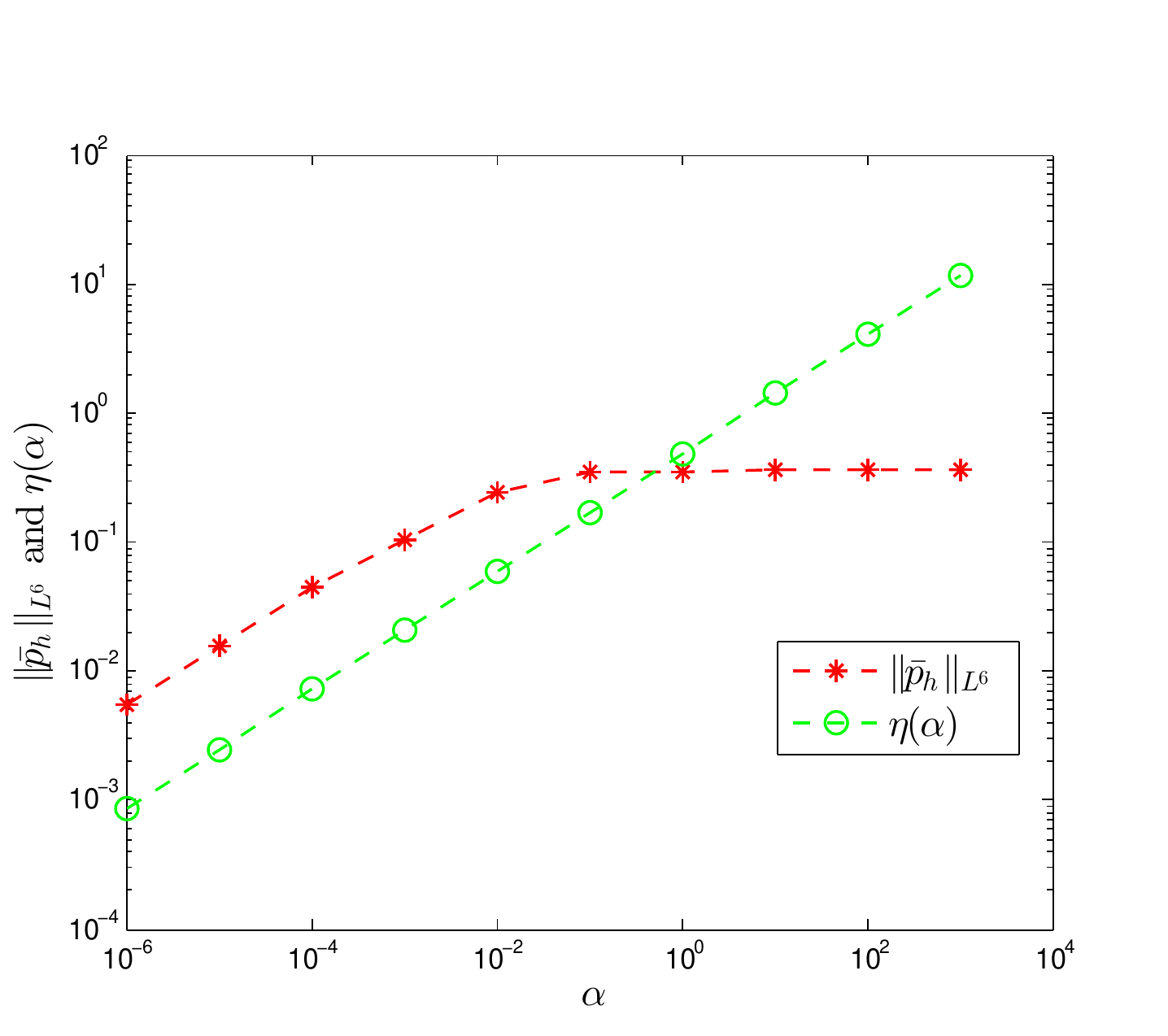}
                \caption{$\|\bar p_h\|_{L^6}$ and $\eta(\alpha)$ vs. $\alpha$.}
        \end{subfigure}%
        ~ 
        \begin{subfigure}[h!]{0.5\textwidth}
                \includegraphics[trim = 35mm 80mm 30mm 70mm, clip, width=\textwidth]{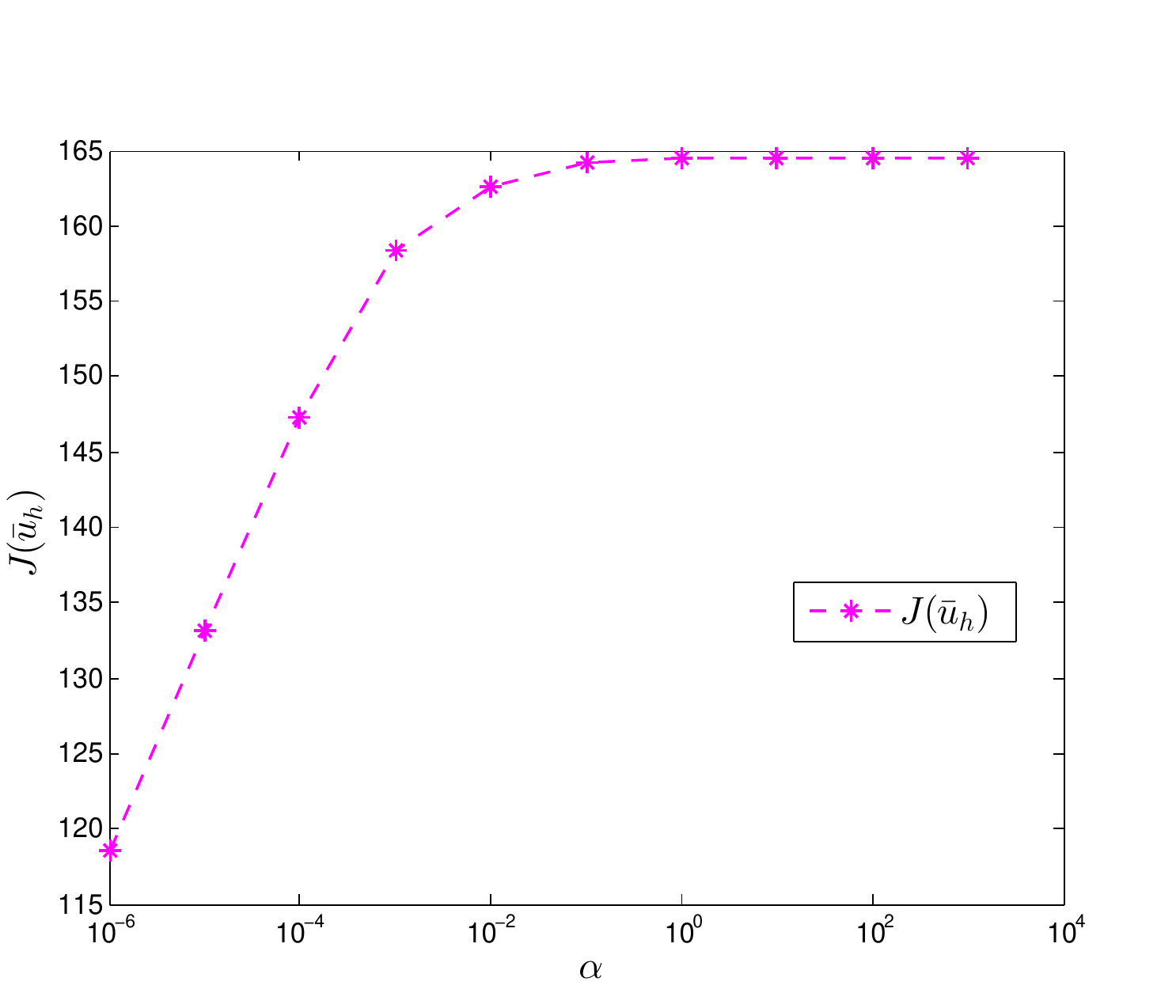}
                \caption{$J(\bar u_h)$  vs. $\alpha$.}
        \end{subfigure}
        
        \begin{subfigure}[h!]{0.5\textwidth}
                \includegraphics[trim = 40mm 80mm 30mm 70mm, clip, width=\textwidth]{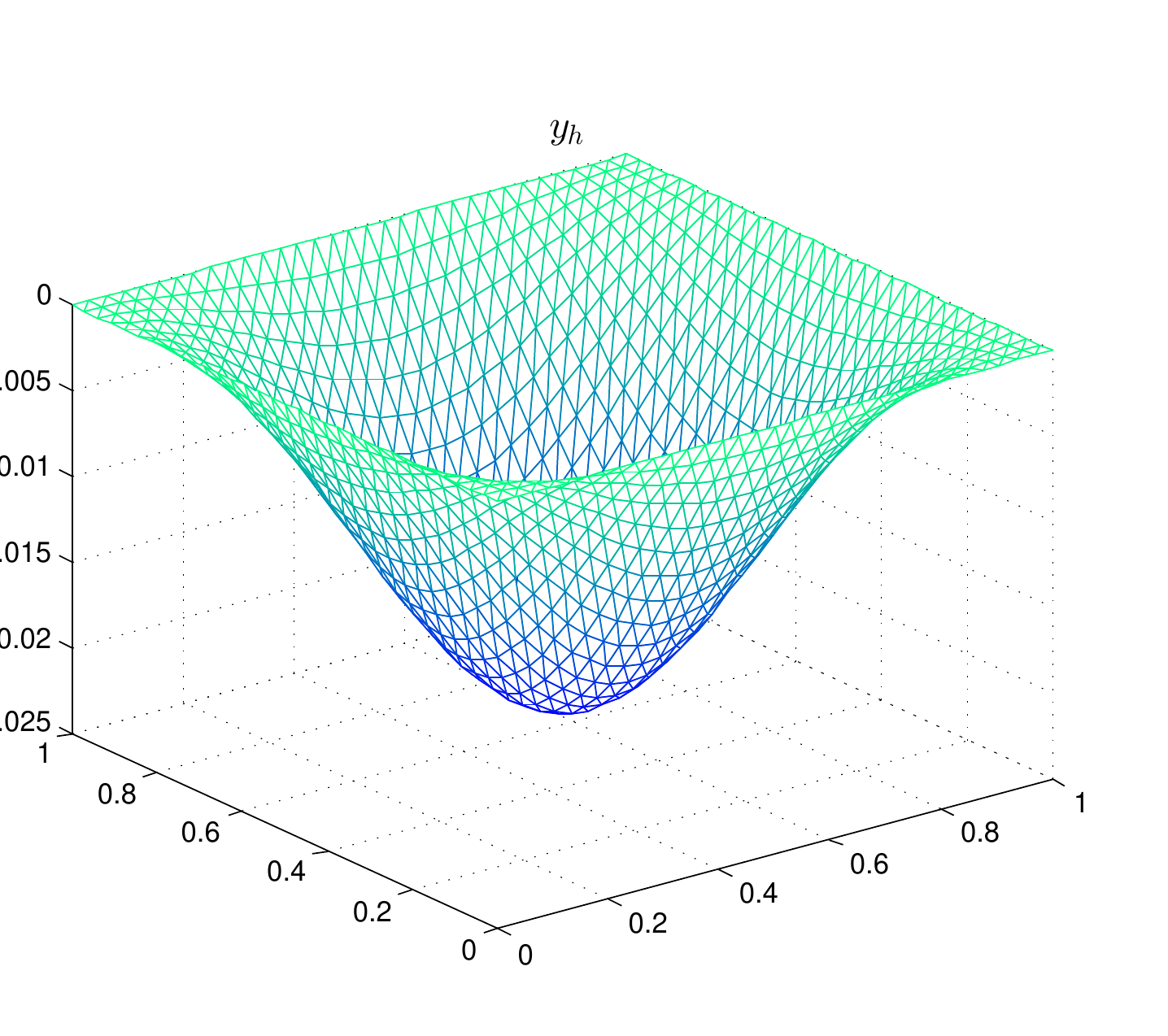}
                \caption{The optimal state $\bar y_h$.}
        \end{subfigure}~
        \begin{subfigure}[h!]{0.5\textwidth}
                \includegraphics[trim = 40mm 80mm 30mm 70mm, clip, width=\textwidth]{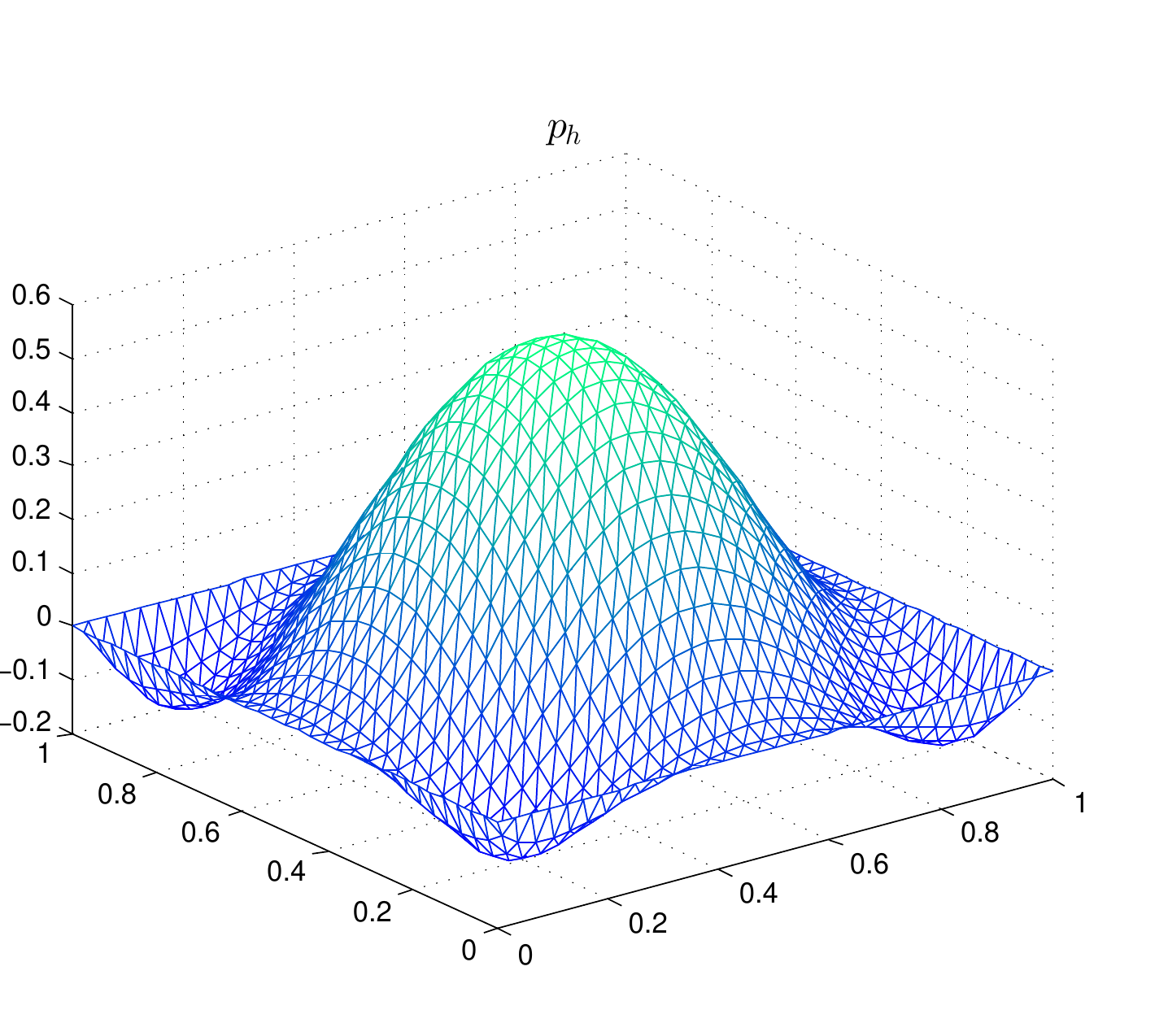}
                \caption{The adjoint state $\bar p_h$.}
        \end{subfigure}
        \begin{subfigure}[h!]{0.5\textwidth}
                \includegraphics[trim = 40mm 80mm 30mm 70mm, clip, width=\textwidth]{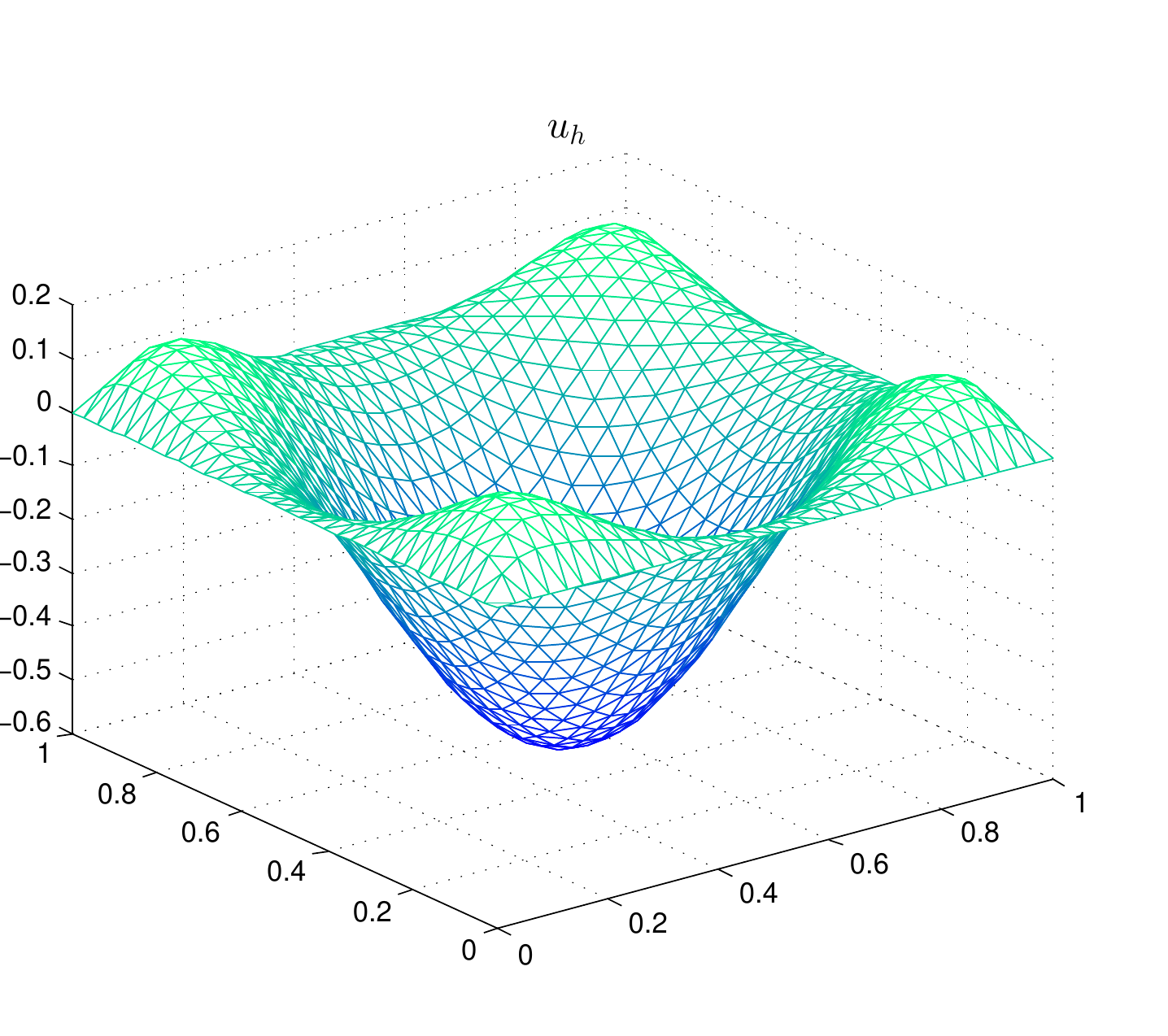}
                \caption{The optimal control $\bar u_h$.}
        \end{subfigure}
        
        \caption{Example~\ref{example: y power 5} Case~1 with choice~\textbf{A2} for $y_0$: The values of $\|\bar p_h\|_{L^6}$, $\eta(\alpha)$ and $J(\bar u_h)$ vs. $\alpha$. The optimal state $\bar y_h$, the optimal control $\bar u_h$ and the adjoint state $\bar p_h$ for $\alpha=1$.}
        \label{figure: example y5 case(unconstrained) choiceA2}
\end{figure}


\noindent
\textbf{Case~2} (constrained control)
In this case we consider constraints only on the control, we set
\begin{align*}
u_a & = -5, \\
u_b & = 5, \\
y_b &=-y_a=\infty.
\end{align*}
Table~\ref{table: example y5 case(constrained control) choiceA1} shows the values of  $\|\bar p_h\|_{L^6} $, $\eta(\alpha)$ and $J(\bar u_h)$ computed for different values of $\alpha$ with choice~\textbf{A1} for $y_0$. The graphical illustration of these findings are shown in Figure~\ref{figure: example y5 case(constrained control) choiceA1}. We see that $\bar u_h$ is a global minimum for $\alpha$ approximately greater than $10^{-3}$. The numerical results associated with the choice~\textbf{A2} are given in Table~\ref{table: example y5 case(constrained control) choiceA2} and illustrated in Figure~\ref{figure: example y5 case(constrained control) choiceA2}. In this case $\bar u_h$ is a global minimum for $\alpha$ approximately greater than $1$.

\begin{table}[p]
\caption{Example~\ref{example: y power 5} Case~2 with choice~\textbf{A1} for $y_0$: The values of  $\|\bar p_h\|_{L^6} $, $\eta(\alpha)$ and $J(\bar u_h)$ for different values of $\alpha$.}
\label{table: example y5 case(constrained control) choiceA1}
\begin{tabular}{ l  c  c  c}
\toprule
$\alpha$ &	  $\|\bar p_h\|_{L^6}$&        $\eta(\alpha)$ &   $J(\bar u_h)$  \\
\midrule[1pt]

1.0e-06 &	  1.613825290585e-02 &	 	 8.697974773247e-04 &	 	 4.507855415302e-01   \\
1.0e-05 &	  1.613824266503e-02 &	 	 2.498914960443e-03 &	 	 4.508885528139e-01   \\
1.0e-04 &	  1.613816501602e-02 &	 	 7.179344781194e-03 &	 	 4.519051721159e-01   \\
1.0e-03 &	  1.615565078678e-02 &	 	 2.062614866979e-02 &	 	 4.612661359991e-01   \\
1.0e-02 &	  1.672495487678e-02 &	 	 5.925861229879e-02 &	 	 4.922543706340e-01   \\
1.0e-01 &	  1.696149575588e-02 &	 	 1.702490943800e-01 &	 	 4.992144828609e-01   \\
1.0e+00 &	  1.698552077353e-02 &	 	 4.891230660460e-01 &	 	 4.999213370331e-01   \\
1.0e+01 &	  1.698792705311e-02 &	 	 1.405243150394e+00 &	 	 4.999921325890e-01   \\
1.0e+02 &	  1.698816771892e-02 &	 	 4.037242258255e+00 &	 	 4.999992132478e-01   \\
1.0e+03 &	  1.698819178587e-02 &	 	 1.159893577654e+01 &	 	 4.999999213247e-01   \\

\bottomrule 
\end{tabular}
\end{table}

\begin{figure}[p]
        \centering
        \begin{subfigure}[h!]{0.5\textwidth}
                \includegraphics[trim = 40mm 80mm 30mm 70mm, clip, width=\textwidth]{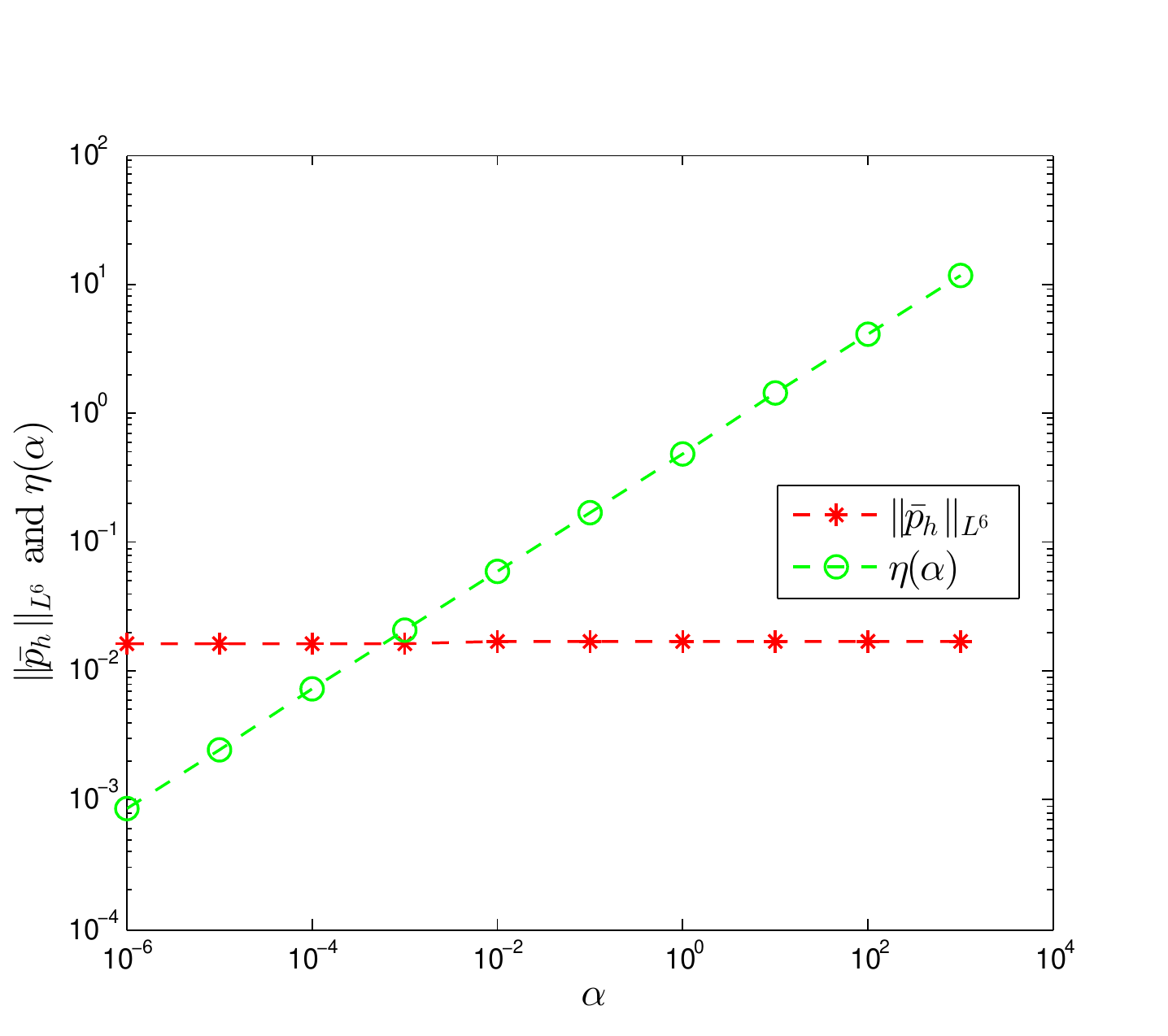}
                \caption{$\|\bar p_h\|_{L^6}$ and $\eta(\alpha)$ vs. $\alpha$.}
        \end{subfigure}%
        ~ 
        \begin{subfigure}[h!]{0.5\textwidth}
                \includegraphics[trim = 30mm 80mm 30mm 70mm, clip, width=\textwidth]{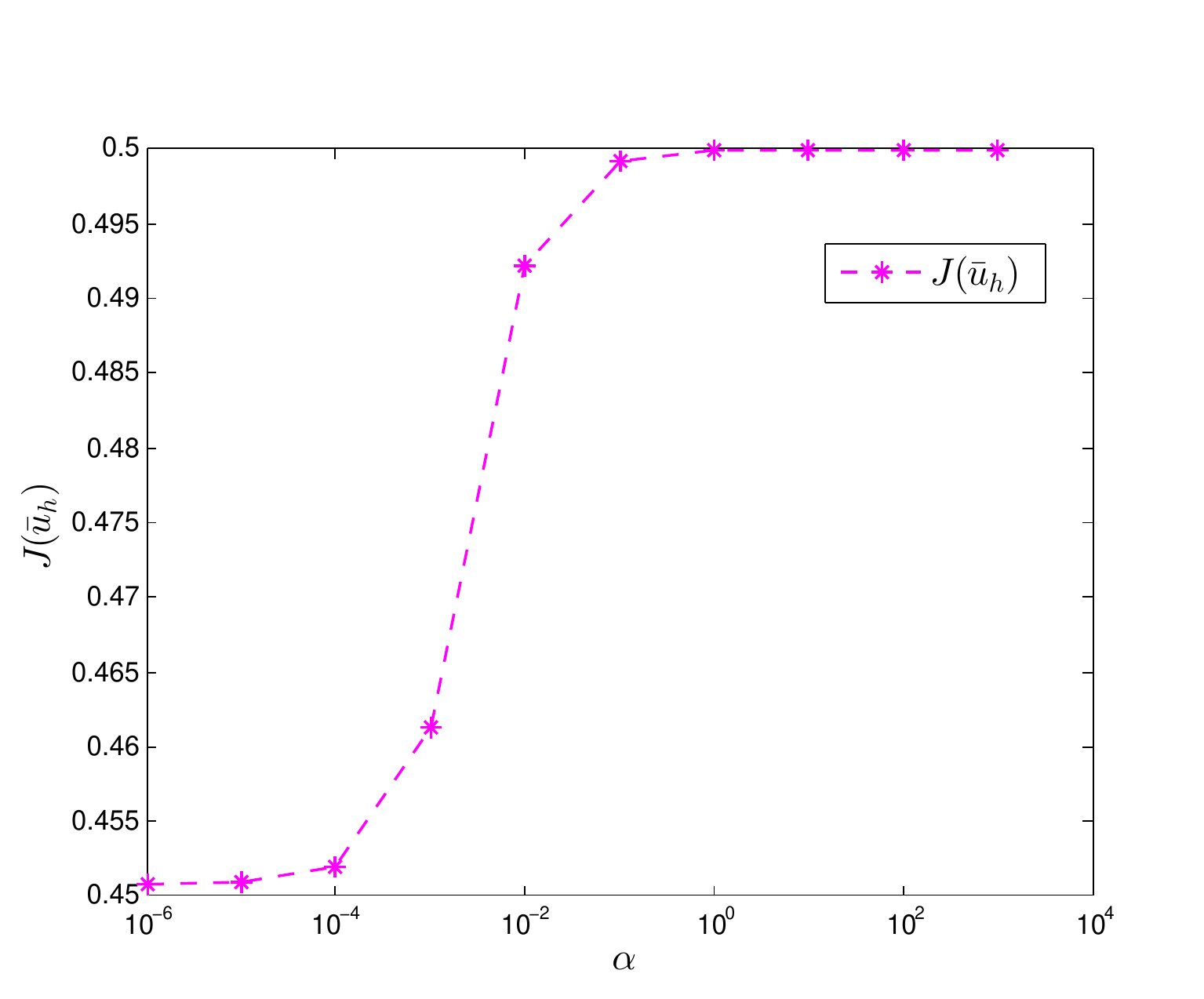}
                \caption{$J(\bar u_h)$  vs. $\alpha$.}
        \end{subfigure}
        
        \begin{subfigure}[h!]{0.5\textwidth}
                \includegraphics[trim = 40mm 80mm 30mm 70mm, clip, width=\textwidth]{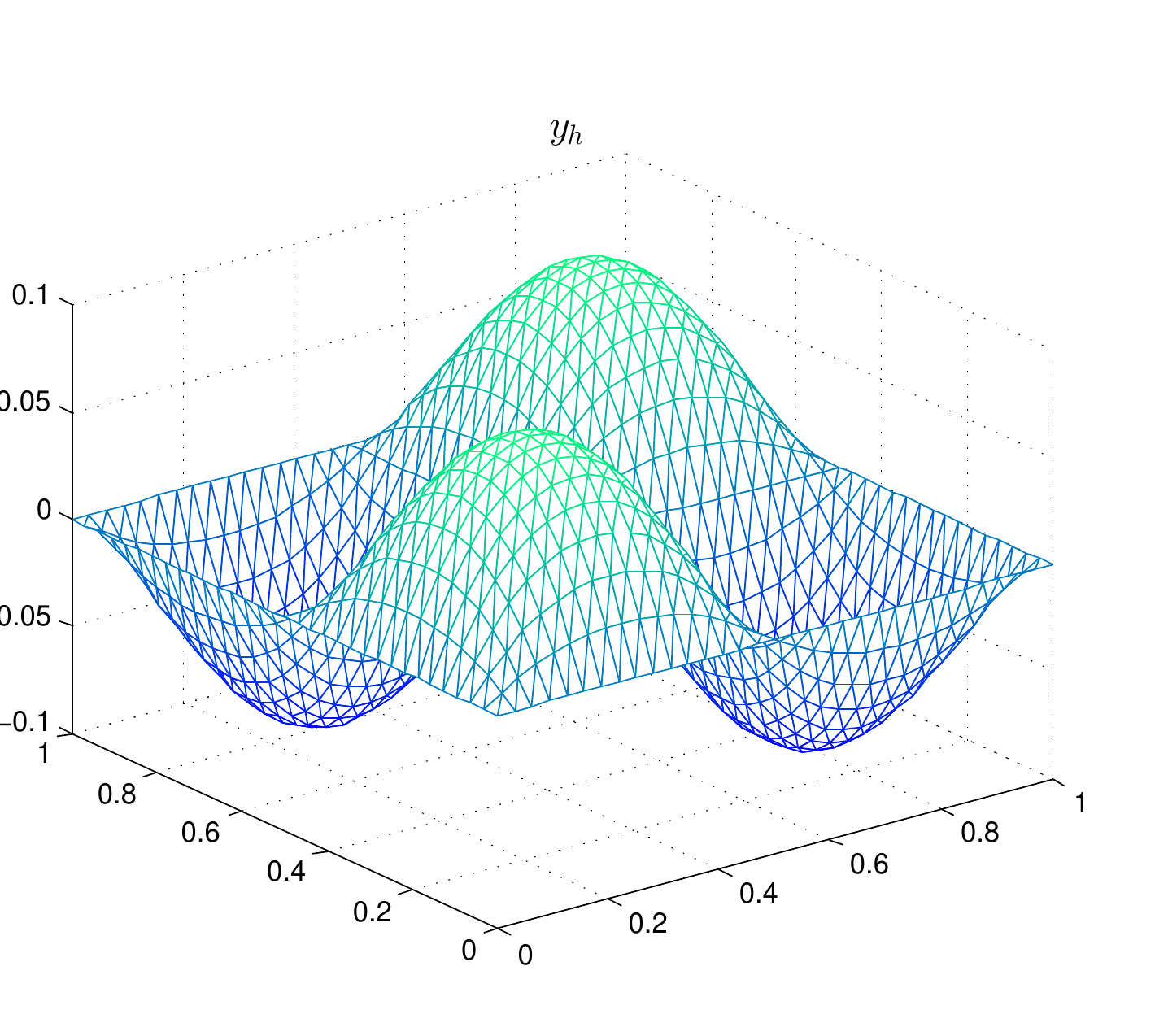}
                \caption{The optimal state $\bar y_h$.}
        \end{subfigure}~
        \begin{subfigure}[h!]{0.5\textwidth}
                \includegraphics[trim = 40mm 80mm 30mm 70mm, clip, width=\textwidth]{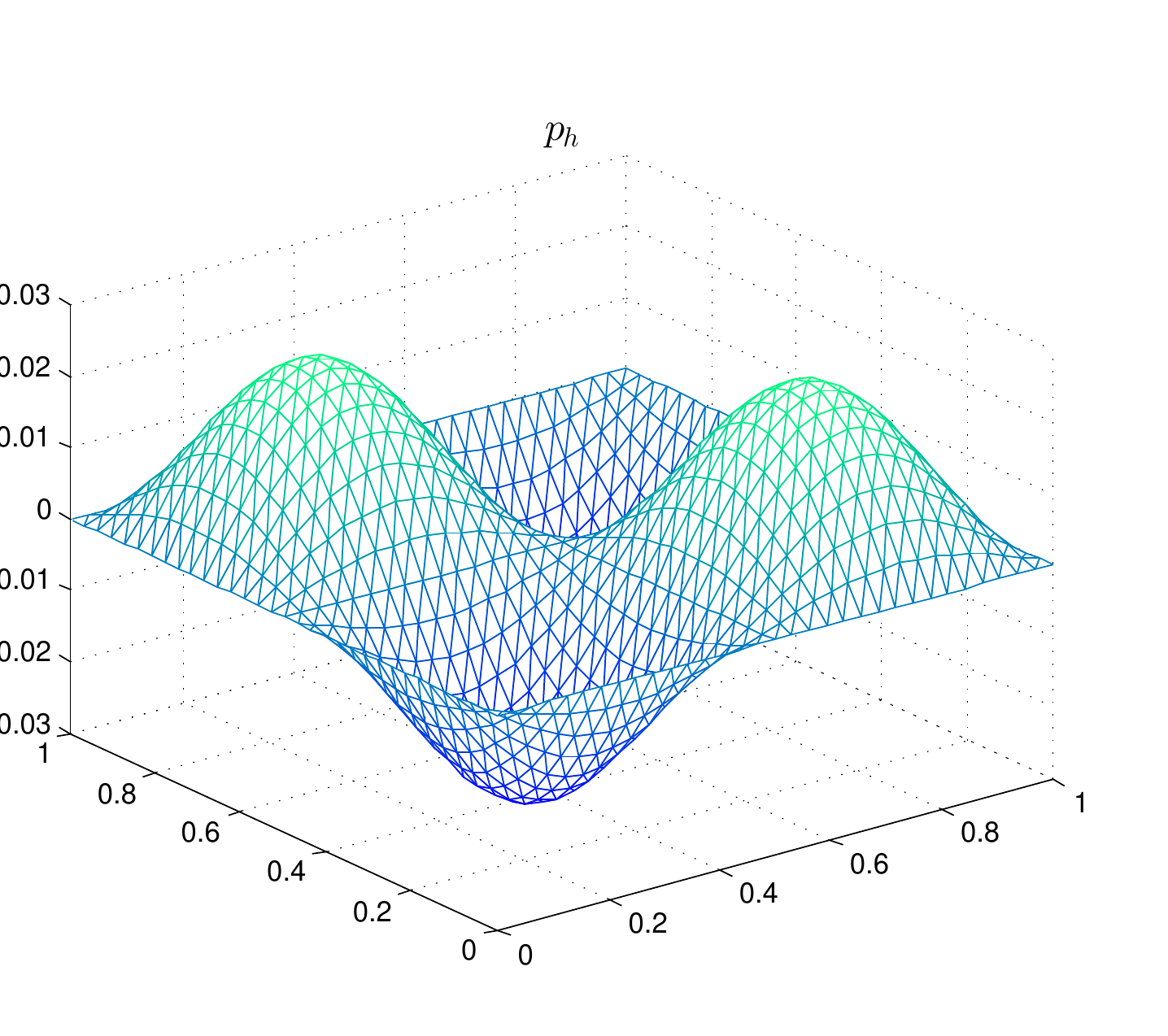}
                \caption{The adjoint state $\bar p_h$.}
        \end{subfigure}
        
        \begin{subfigure}[h!]{0.5\textwidth}
                \includegraphics[trim = 40mm 80mm 30mm 70mm, clip, width=\textwidth]{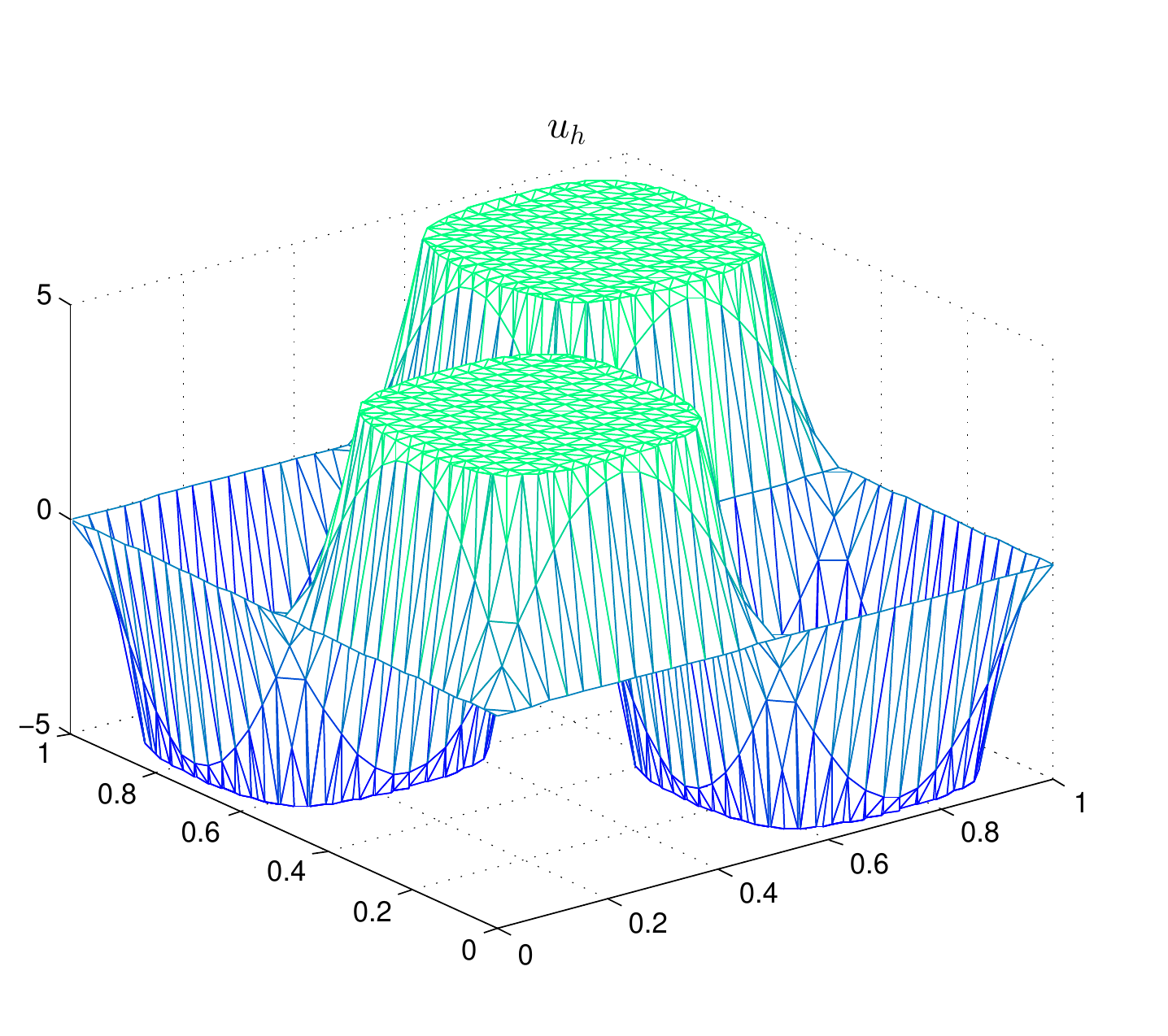}
                \caption{The optimal control $\bar u_h$.}
        \end{subfigure}~
        \begin{subfigure}[h!]{0.5\textwidth}
                \includegraphics[trim = 40mm 80mm 30mm 70mm, clip, width=\textwidth]{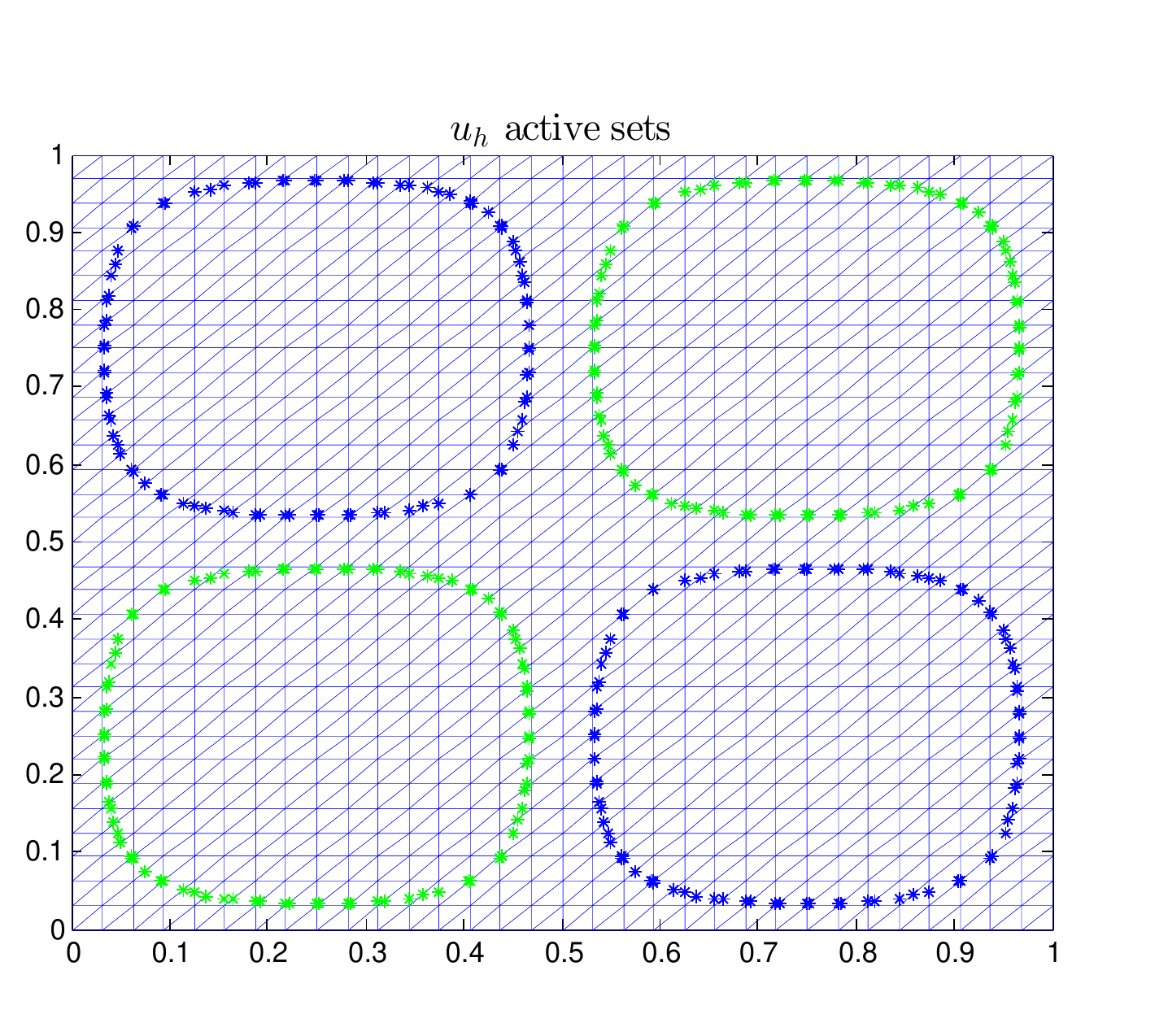}
                \caption{The control active sets ($\bar u_h=5$ inside the green circles and $\bar u_h=-5$ inside the blue ones).}
        \end{subfigure}

        \caption{Example~\ref{example: y power 5} Case~2 with choice~\textbf{A1} for $y_0$: The values of $\|\bar p_h\|_{L^6}$, $\eta(\alpha)$ and $J(\bar u_h)$ vs. $\alpha$. The optimal state $\bar y_h$, the optimal control $\bar u_h$, the control active sets, and the adjoint state $\bar p_h$ for $\alpha=10^{-3}$.}
        \label{figure: example y5 case(constrained control) choiceA1}
\end{figure}

\begin{table}[p]
\caption{Example~\ref{example: y power 5} Case~2 with choice~\textbf{A2} for $y_0$: The values of  $\|\bar p_h\|_{L^6} $, $\eta(\alpha)$ and $J(\bar u_h)$ for different values of $\alpha$.}
\label{table: example y5 case(constrained control) choiceA2}
\begin{tabular}{ l  c  c  c}
\toprule
$\alpha$ &	  $\|\bar p_h\|_{L^6}$&        $\eta(\alpha)$ &   $J(\bar u_h)$  \\
\midrule[1pt]

1.0e-06 &	  3.456649663660e-01 &	 	 8.697974773247e-04 &	 	 1.636832040856e+02   \\
1.0e-05 &	  3.456649990198e-01 &	 	 2.498914960443e-03 &	 	 1.636833073745e+02   \\
1.0e-04 &	  3.456663172695e-01 &	 	 7.179344781194e-03 &	 	 1.636843379602e+02   \\
1.0e-03 &	  3.456602557101e-01 &	 	 2.062614866979e-02 &	 	 1.636944643396e+02   \\
1.0e-02 &	  3.457537810584e-01 &	 	 5.925861229879e-02 &	 	 1.637855203878e+02   \\
1.0e-01 &	  3.494672249476e-01 &	 	 1.702490943800e-01 &	 	 1.642029145907e+02   \\
1.0e+00 &	  3.554038724369e-01 &	 	 4.891230660460e-01 &	 	 1.644198119684e+02   \\
1.0e+01 &	  3.560155910725e-01 &	 	 1.405243150394e+00 &	 	 1.644419766411e+02   \\
1.0e+02 &	  3.560769159456e-01 &	 	 4.037242258255e+00 &	 	 1.644441976184e+02   \\
1.0e+03 &	  3.560830499750e-01 &	 	 1.159893577654e+01 &	 	 1.644444197614e+02   \\

\bottomrule 
\end{tabular}
\end{table}

\begin{figure}[p]
        \centering
        \begin{subfigure}[h!]{0.5\textwidth}
                \includegraphics[trim = 40mm 80mm 30mm 70mm, clip, width=\textwidth]{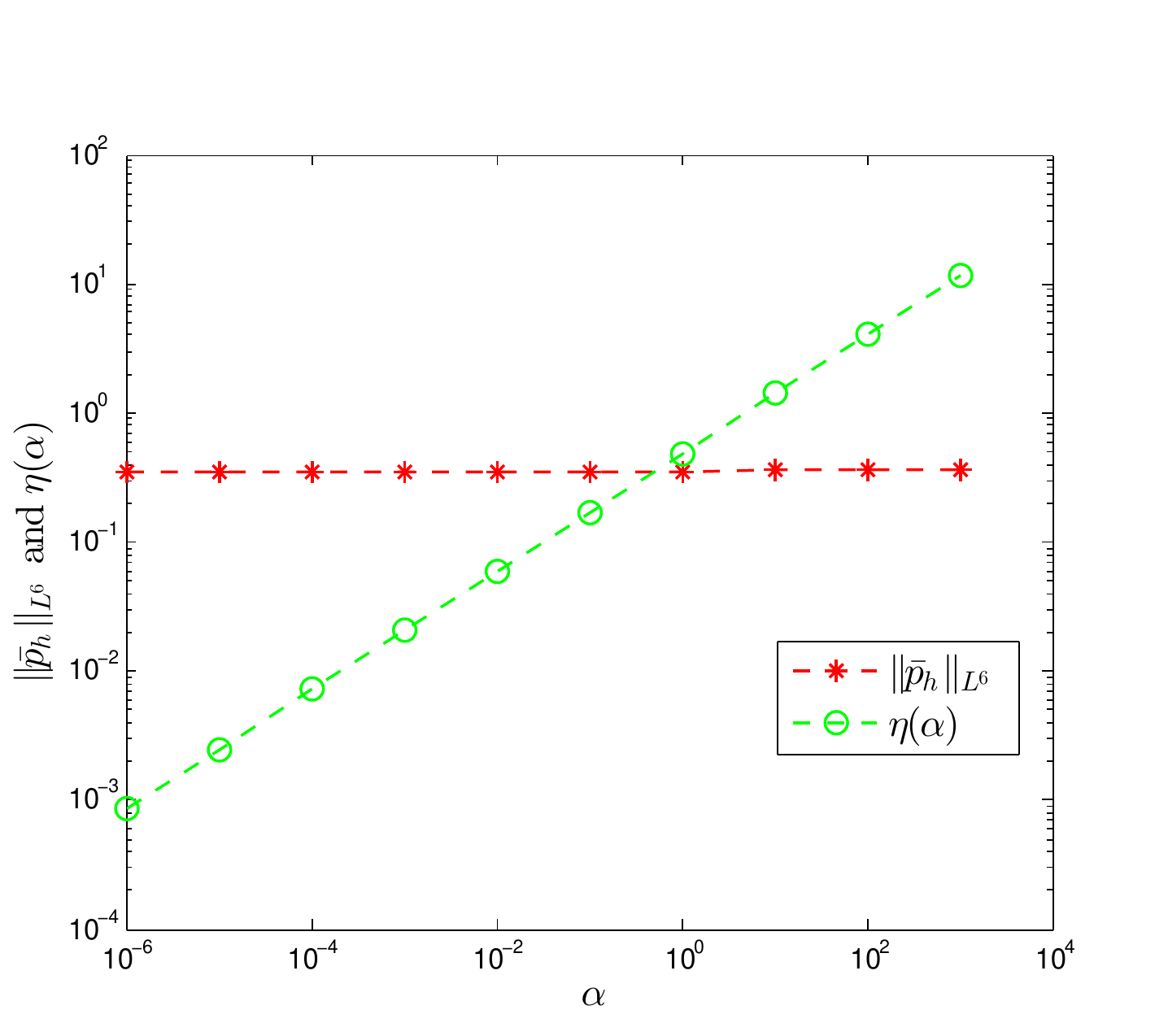}
                \caption{$\|\bar p_h\|_{L^6}$ and $\eta(\alpha)$ vs. $\alpha$.}
        \end{subfigure}%
        ~ 
        \begin{subfigure}[h!]{0.5\textwidth}
                \includegraphics[trim = 30mm 80mm 30mm 70mm, clip, width=\textwidth]{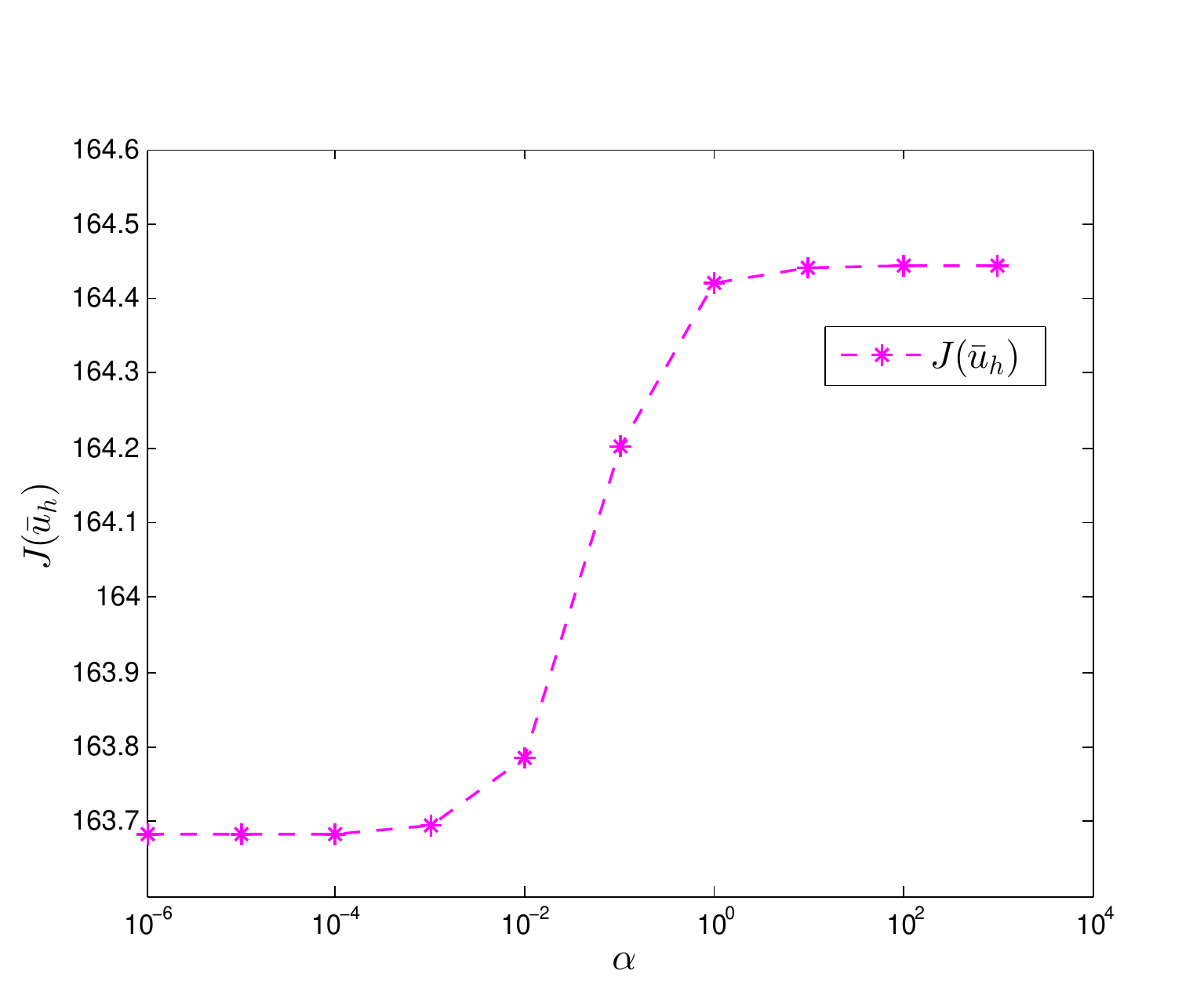}
                \caption{$J(\bar u_h)$  vs. $\alpha$.}
        \end{subfigure}
        
        \begin{subfigure}[h!]{0.5\textwidth}
                \includegraphics[trim = 40mm 80mm 30mm 70mm, clip, width=\textwidth]{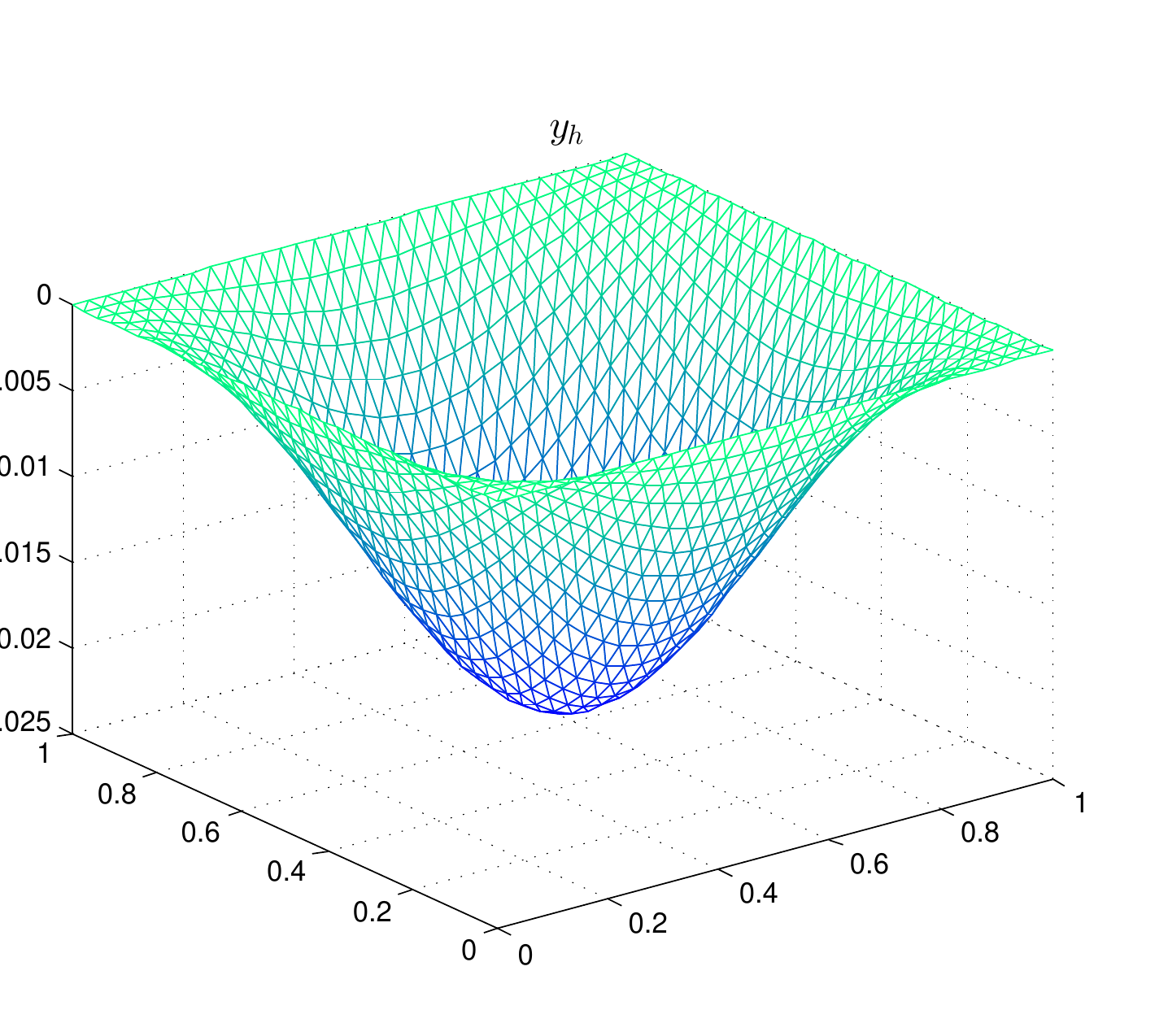}
                \caption{The optimal state $\bar y_h$.}
        \end{subfigure}~
        \begin{subfigure}[h!]{0.5\textwidth}
                \includegraphics[trim = 40mm 80mm 30mm 70mm, clip, width=\textwidth]{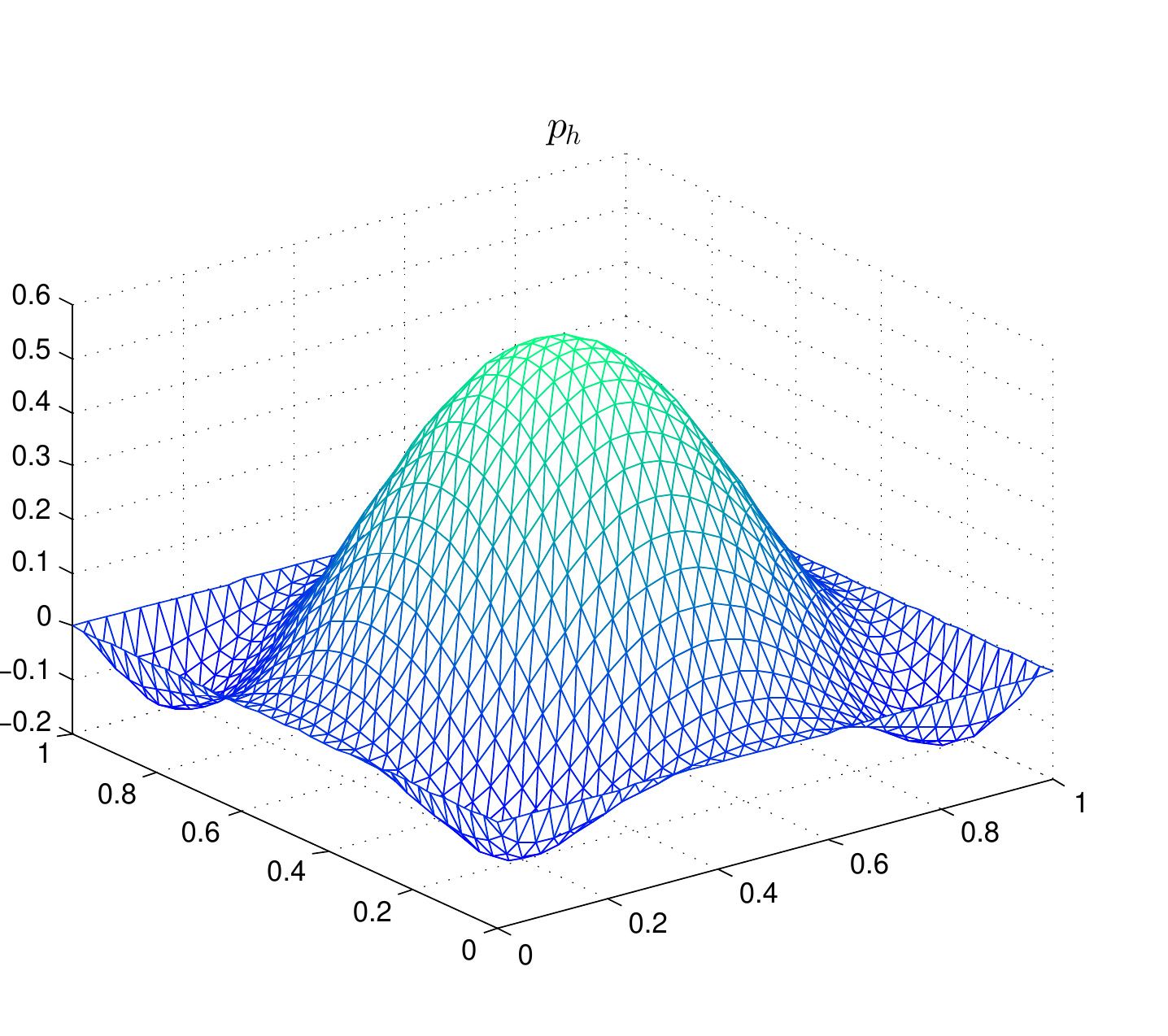}
                \caption{The adjoint state $\bar p_h$.}
        \end{subfigure}
        
        \begin{subfigure}[h!]{0.5\textwidth}
                \includegraphics[trim = 40mm 80mm 30mm 70mm, clip, width=\textwidth]{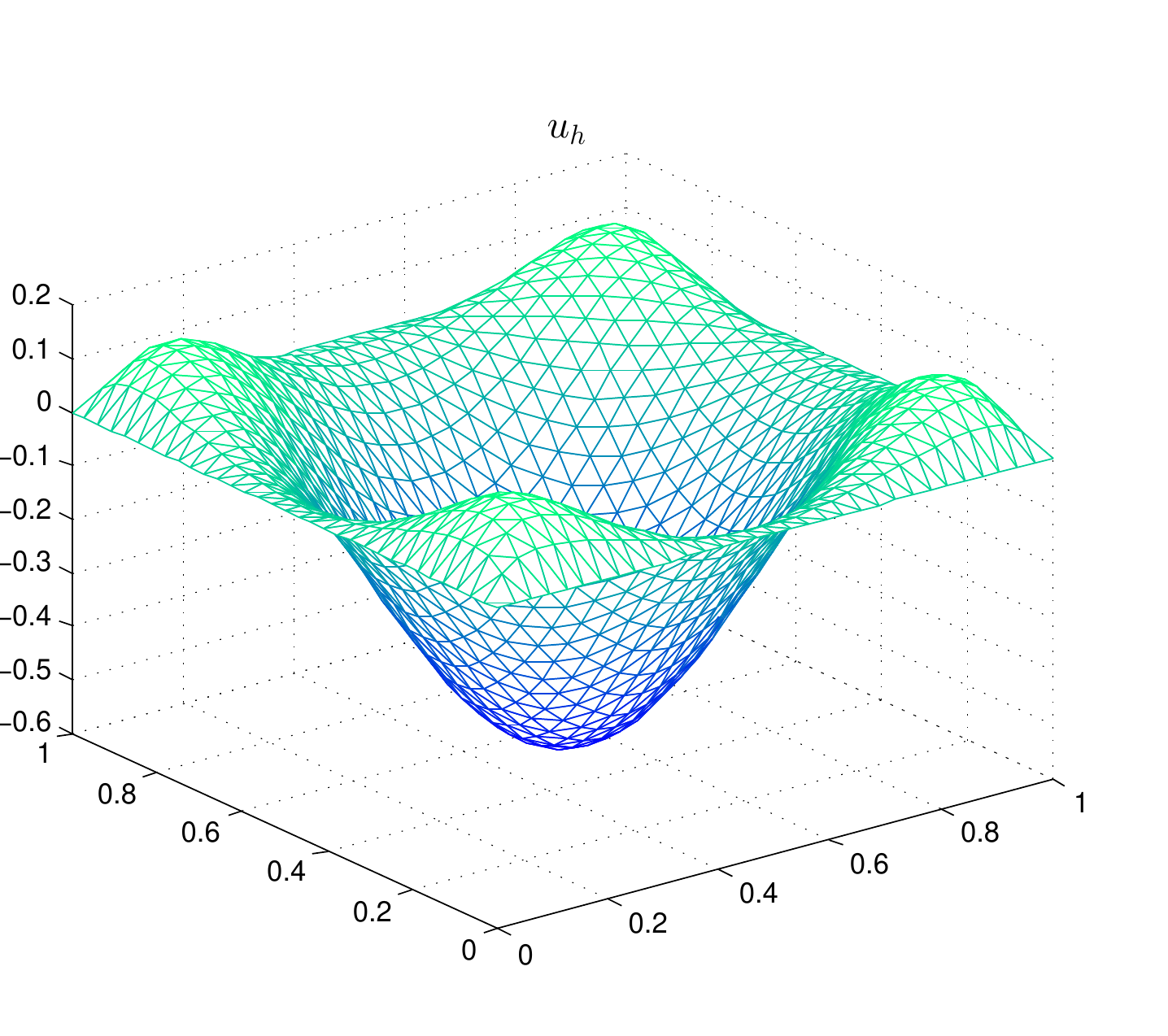}
                \caption{The optimal control $\bar u_h$.}
        \end{subfigure}
        
        \caption{Example~\ref{example: y power 5} Case~2 with choice~\textbf{A2} for $y_0$: The values of $\|\bar p_h\|_{L^6}$, $\eta(\alpha)$ and $J(\bar u_h)$ vs. $\alpha$. The optimal state $\bar y_h$, the optimal control $\bar u_h$ and the adjoint state $\bar p_h$ for $\alpha=1$.}
        \label{figure: example y5 case(constrained control) choiceA2}
\end{figure}


\noindent
\textbf{Case~3} (constrained state)
In this case we consider constrains only on the state, we set
\begin{align*}
u_b & = -u_a=\infty, \\
y_a &=-1, \\
y_b &=1.
\end{align*}
The numerical findings associated with choice~\textbf{A1} are provided in Table~\ref{table: example y5 case(constrained state) choiceA1} and illustrated in Figure~\ref{figure: example y5 case(constrained state) choiceA1}. We see that $\bar u_h$ is a global minimum for all values of $\alpha$. For the choice~\textbf{A2},  the results are given in Table~\ref{table: example y5 case(constrained state) choiceA2} and illustrated in Figure~\ref{figure: example y5 case(constrained state) choiceA2}. We see that $\bar u_h$ is a global minimum only for $\alpha$ approximately greater than $1$.

\begin{table}[p]
\caption{Example~\ref{example: y power 5} Case~3 with choice~\textbf{A1} for $y_0$:  The values of  $\|\bar p_h\|_{L^6} $, $\eta(\alpha)$ and $J(\bar u_h)$ for different values of $\alpha$.}
\label{table: example y5 case(constrained state) choiceA1}
\begin{tabular}{ l  c  c  c}
\toprule
$\alpha$ &	  $\|\bar p_h\|_{L^6}$&        $\eta(\alpha)$ &   $J(\bar u_h)$  \\
\midrule[1pt]

1.0e-06 &	  1.293594798095e-04 &	 	 8.697974773247e-04 &	 	 6.247856764953e-02   \\
1.0e-05 &	  8.673961098825e-04 &	 	 2.498914960443e-03 &	 	 8.936458658379e-02   \\
1.0e-04 &	  5.421978025542e-03 &	 	 7.179344781194e-03 &	 	 2.033602173575e-01   \\
1.0e-03 &	  1.467650352720e-02 &	 	 2.062614866979e-02 &	 	 4.320253853445e-01   \\
1.0e-02 &	  1.672495487678e-02 &	 	 5.925861229879e-02 &	 	 4.922543706340e-01   \\
1.0e-01 &	  1.696149575588e-02 &	 	 1.702490943800e-01 &	 	 4.992144828609e-01   \\
1.0e+00 &	  1.698552077353e-02 &	 	 4.891230660460e-01 &	 	 4.999213370331e-01   \\
1.0e+01 &	  1.698792705311e-02 &	 	 1.405243150394e+00 &	 	 4.999921325890e-01   \\
1.0e+02 &	  1.698816771892e-02 &	 	 4.037242258255e+00 &	 	 4.999992132478e-01   \\
1.0e+03 &	  1.698819178587e-02 &	 	 1.159893577654e+01 &	 	 4.999999213247e-01   \\

\bottomrule 
\end{tabular}
\end{table}

\begin{figure}[p]
        \centering
        \begin{subfigure}[h!]{0.5\textwidth}
                \includegraphics[trim = 40mm 80mm 30mm 70mm, clip, width=\textwidth]{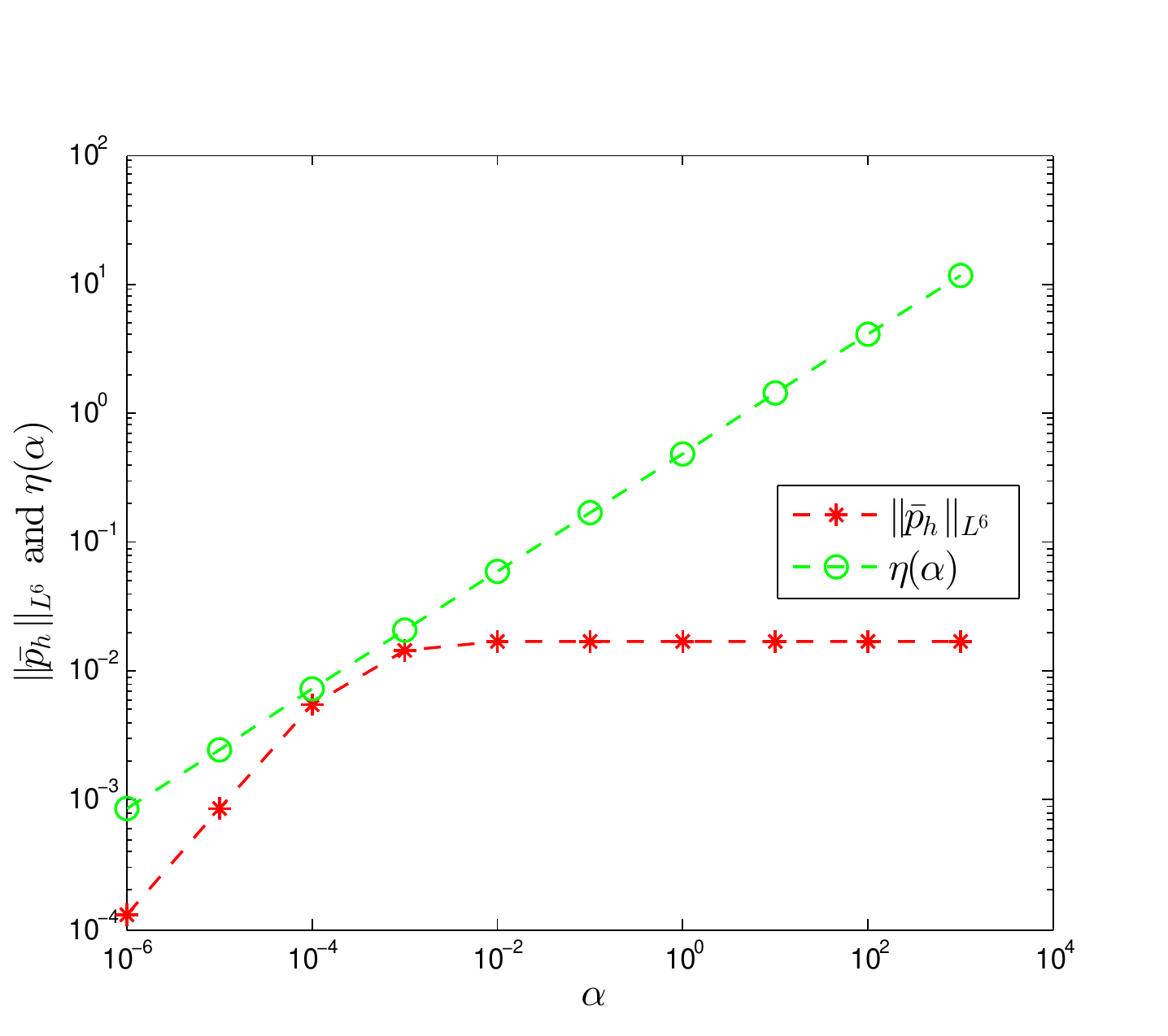}
                \caption{$\|\bar p_h\|_{L^6}$ and $\eta(\alpha)$ vs. $\alpha$.}
        \end{subfigure}%
        ~ 
        \begin{subfigure}[h!]{0.5\textwidth}
                \includegraphics[trim = 30mm 80mm 30mm 70mm, clip, width=\textwidth]{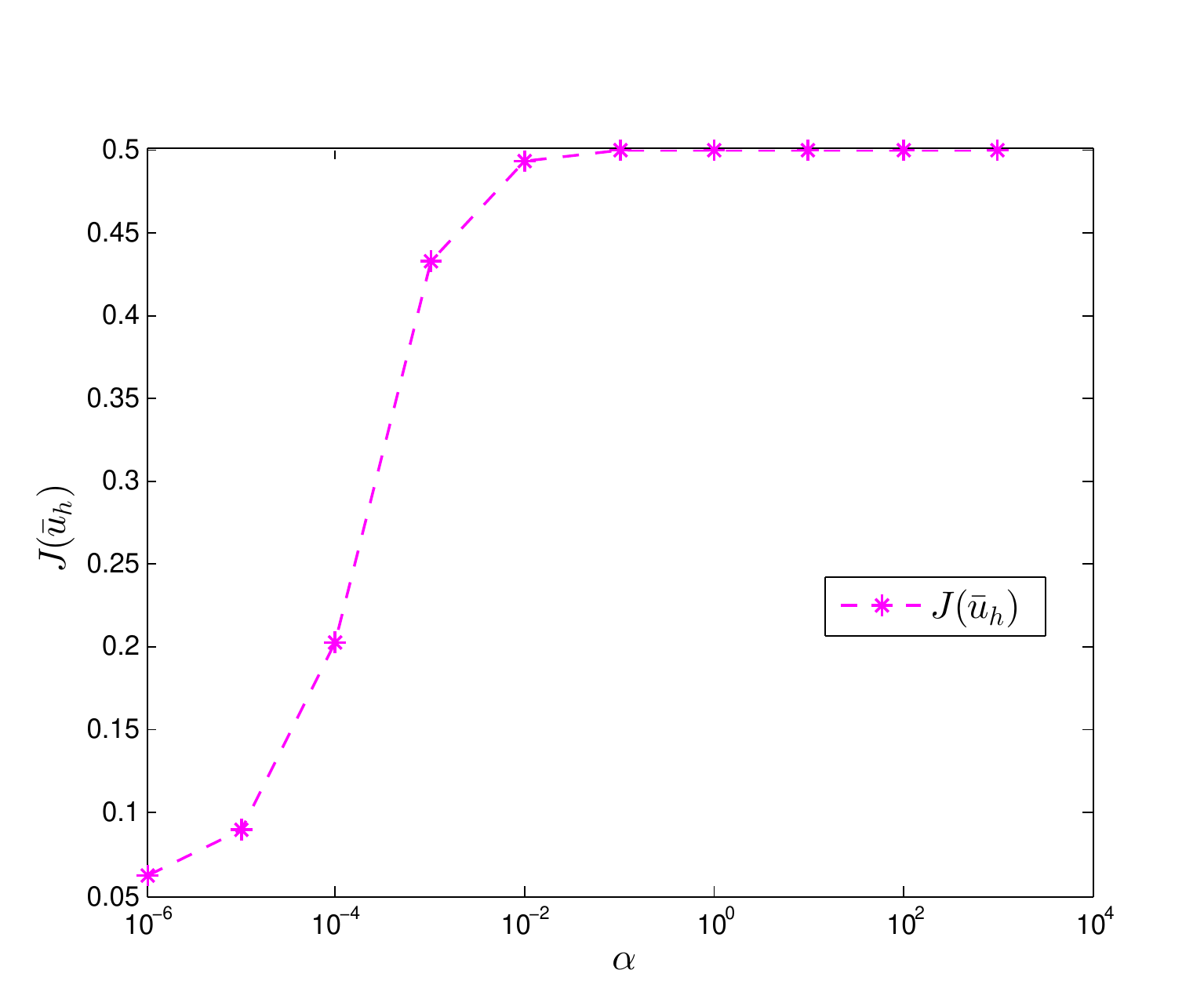}
                \caption{$J(\bar u_h)$  vs. $\alpha$.}
        \end{subfigure}
        
        \begin{subfigure}[h!]{0.5\textwidth}
                \includegraphics[trim = 40mm 80mm 30mm 70mm, clip, width=\textwidth]{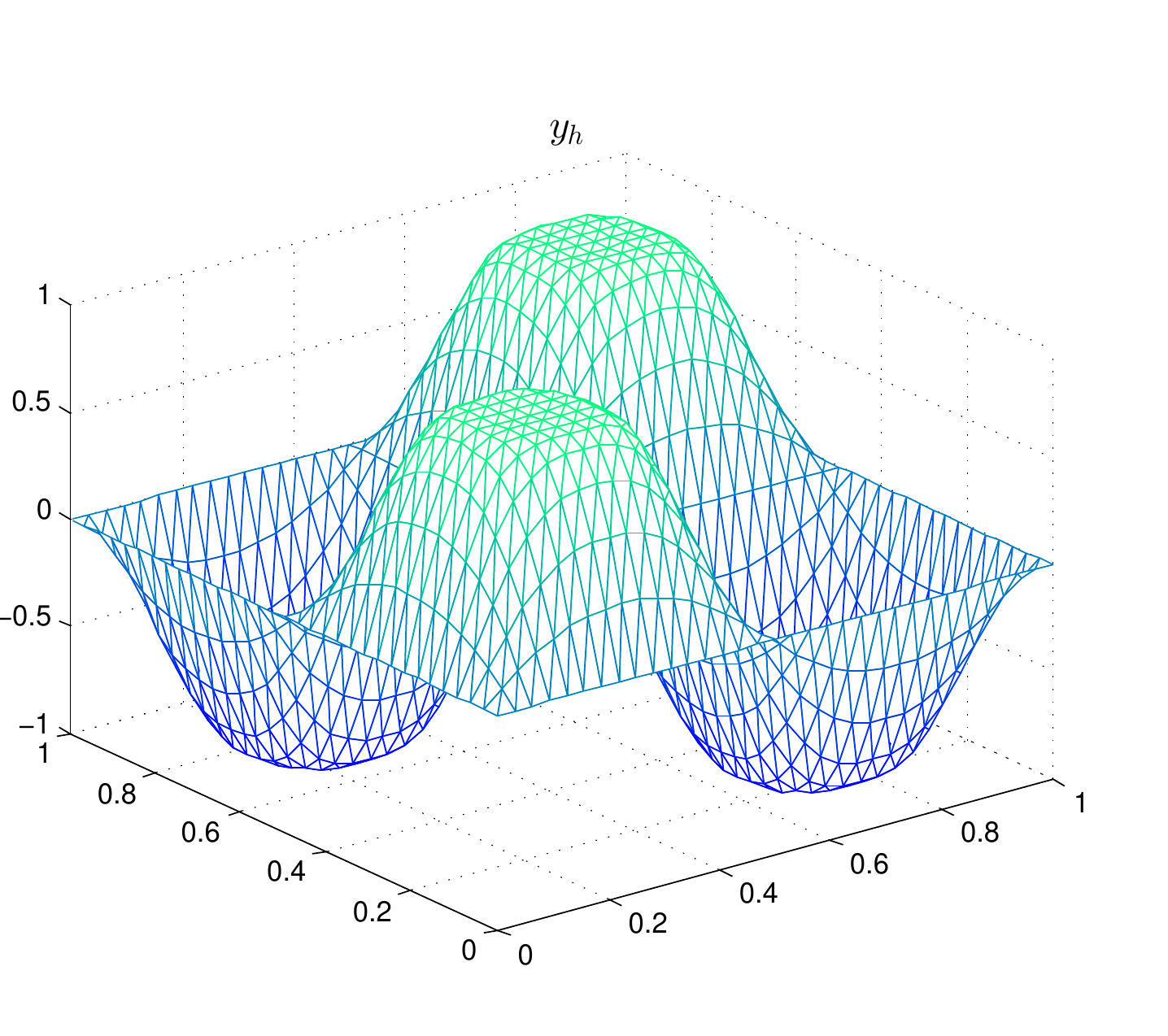}
                \caption{The optimal state $\bar y_h$.}
        \end{subfigure}~
        \begin{subfigure}[h!]{0.5\textwidth}
                \includegraphics[trim = 40mm 80mm 30mm 70mm, clip, width=\textwidth]{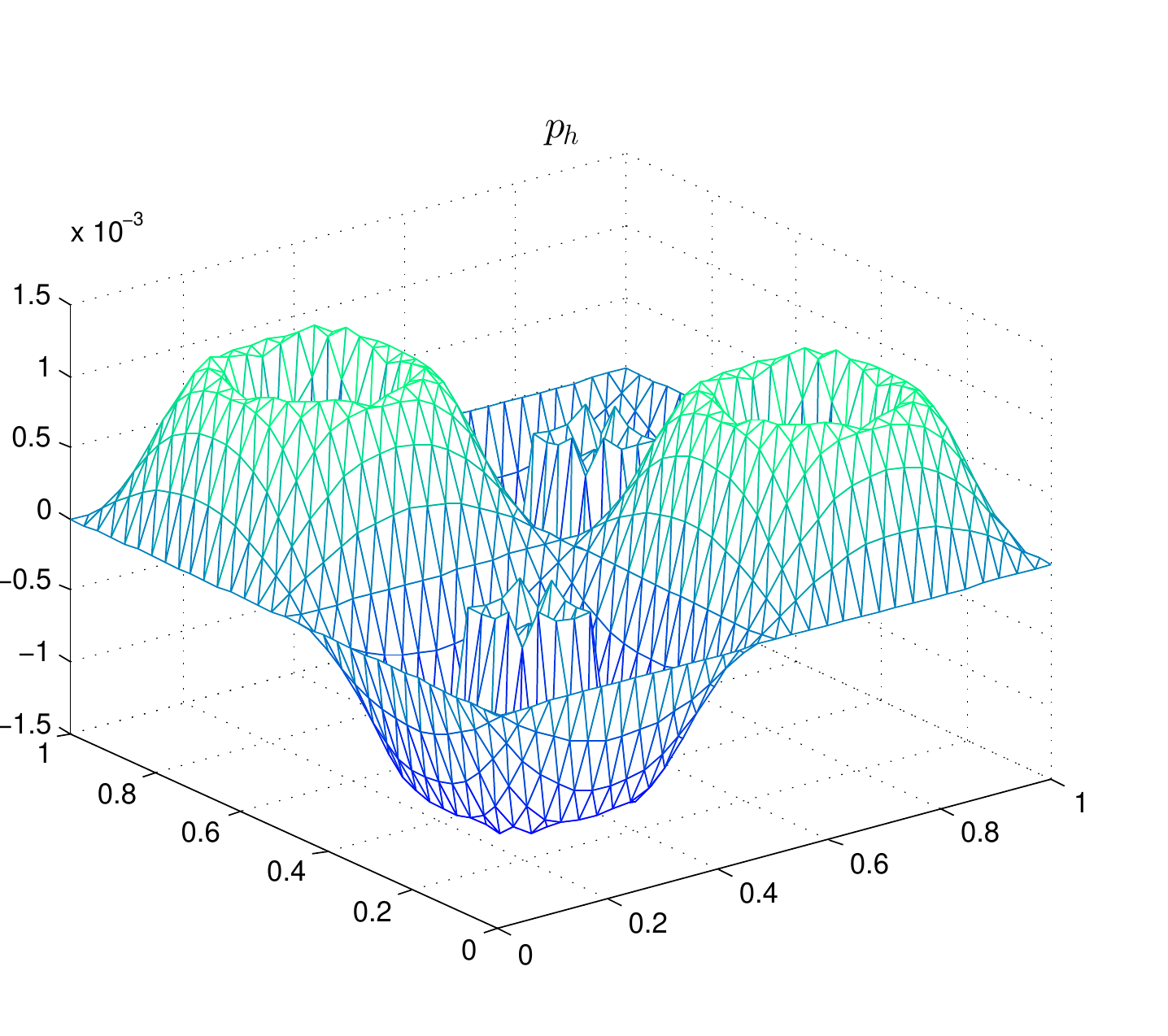}
                \caption{The adjoint state $\bar p_h$.}
        \end{subfigure}
        
        \begin{subfigure}[h!]{0.5\textwidth}
                \includegraphics[trim = 40mm 80mm 30mm 70mm, clip, width=\textwidth]{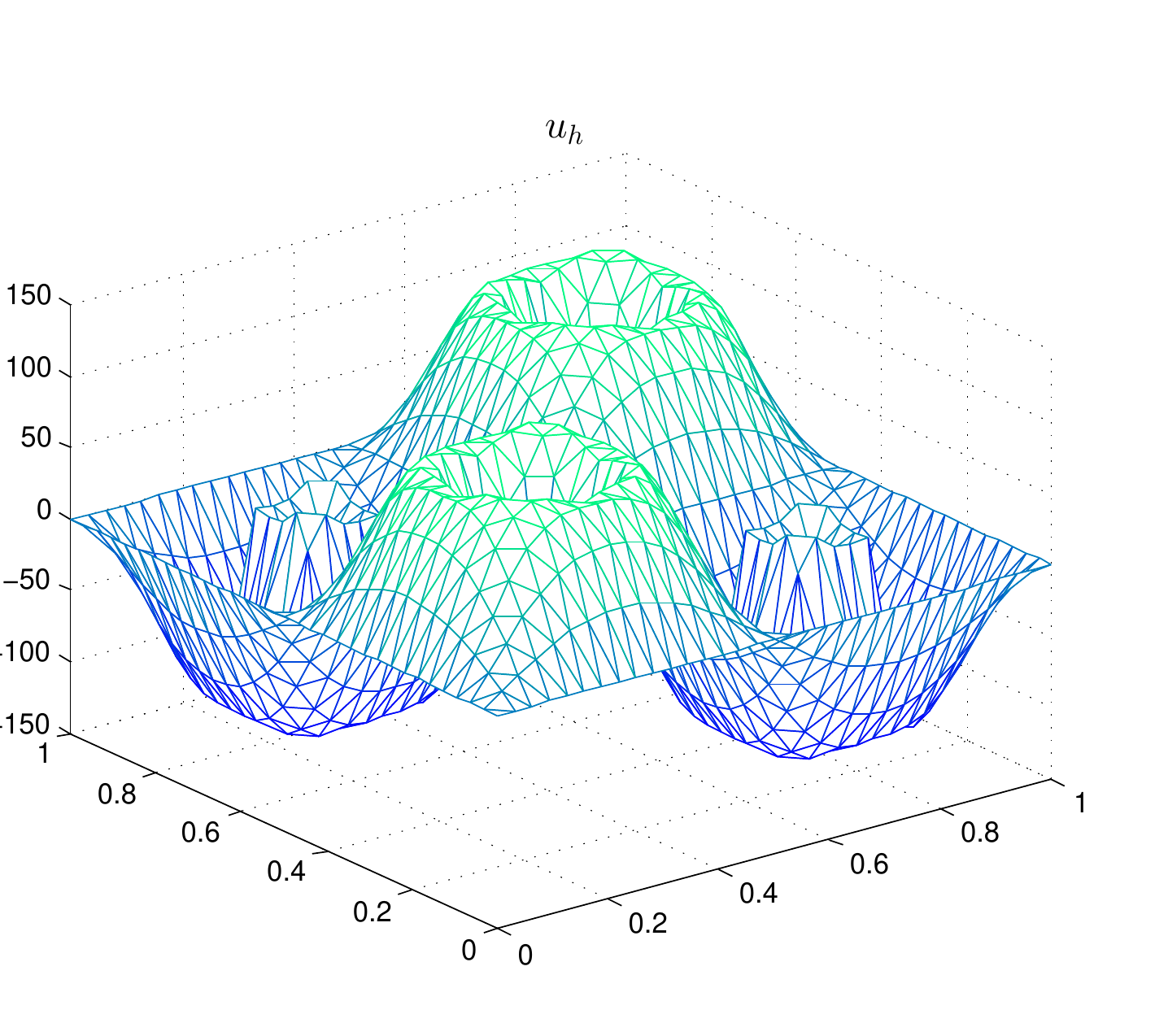}
                \caption{The optimal control $\bar u_h$.}
        \end{subfigure}~
        \begin{subfigure}[h!]{0.5\textwidth}
                \includegraphics[trim = 40mm 80mm 30mm 70mm, clip, width=\textwidth]{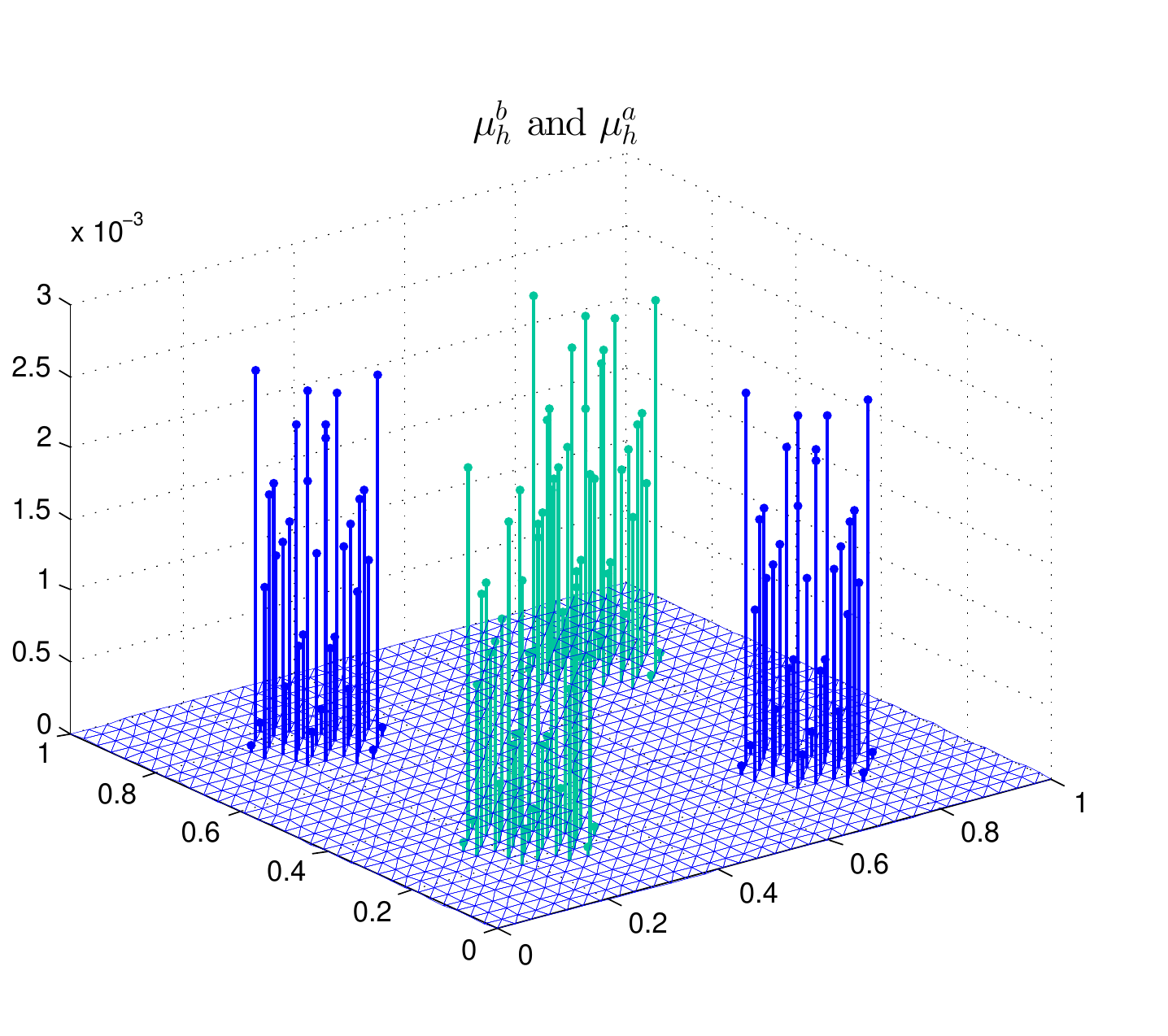}
                \caption{The multipliers $\bar \mu^a_h$ (in blue) and $\bar \mu^b_h$ (in green).}
        \end{subfigure}

        \caption{Example~\ref{example: y power 5} Case~3 with choice~\textbf{A1} for $y_0$: The values of $\|\bar p_h\|_{L^6}$, $\eta(\alpha)$ and $J(\bar u_h)$ vs. $\alpha$. The optimal state $\bar y_h$, the multipliers $\bar \mu^a_h$,$\bar \mu^b_h$, the optimal control $\bar u_h$ and the adjoint state $\bar p_h$ for $\alpha=10^{-5}$.}
        \label{figure: example y5 case(constrained state) choiceA1}
\end{figure}

\begin{table}[p]
\caption{Example~\ref{example: y power 5} Case~3 with choice~\textbf{A2} for $y_0$: The values of  $\|\bar p_h\|_{L^6} $, $\eta(\alpha)$ and $J(\bar u_h)$ for different values of $\alpha$.}
\label{table: example y5 case(constrained state) choiceA2}
\begin{tabular}{ l  c  c  c}
\toprule
$\alpha$ &	  $\|\bar p_h\|_{L^6}$&        $\eta(\alpha)$ &   $J(\bar u_h)$  \\
\midrule[1pt]

1.0e-06 &	  1.139290773221e-03 &	 	 8.697974773247e-04 &	 	 1.525635040951e+02   \\
1.0e-05 &	  8.200728224157e-03 &	 	 2.498914960443e-03 &	 	 1.536016384574e+02   \\
1.0e-04 &	  2.474482888749e-02 &	 	 7.179344781194e-03 &	 	 1.559116076253e+02   \\
1.0e-03 &	  9.716506658549e-02 &	 	 2.062614866979e-02 &	 	 1.600204462920e+02   \\
1.0e-02 &	  1.800129125912e-01 &	 	 5.925861229879e-02 &	 	 1.627566303073e+02   \\
1.0e-01 &	  3.493646725426e-01 &	 	 1.702490943800e-01 &	 	 1.642025836782e+02   \\
1.0e+00 &	  3.554038724369e-01 &	 	 4.891230660460e-01 &	 	 1.644198119684e+02   \\
1.0e+01 &	  3.560155910725e-01 &	 	 1.405243150394e+00 &	 	 1.644419766411e+02   \\
1.0e+02 &	  3.560769159456e-01 &	 	 4.037242258255e+00 &	 	 1.644441976184e+02   \\
1.0e+03 &	  3.560830499750e-01 &	 	 1.159893577654e+01 &	 	 1.644444197614e+02   \\

\bottomrule 
\end{tabular}
\end{table}

\begin{figure}[p]
        \centering
        \begin{subfigure}[h!]{0.5\textwidth}
                \includegraphics[trim = 40mm 80mm 30mm 70mm, clip, width=\textwidth]{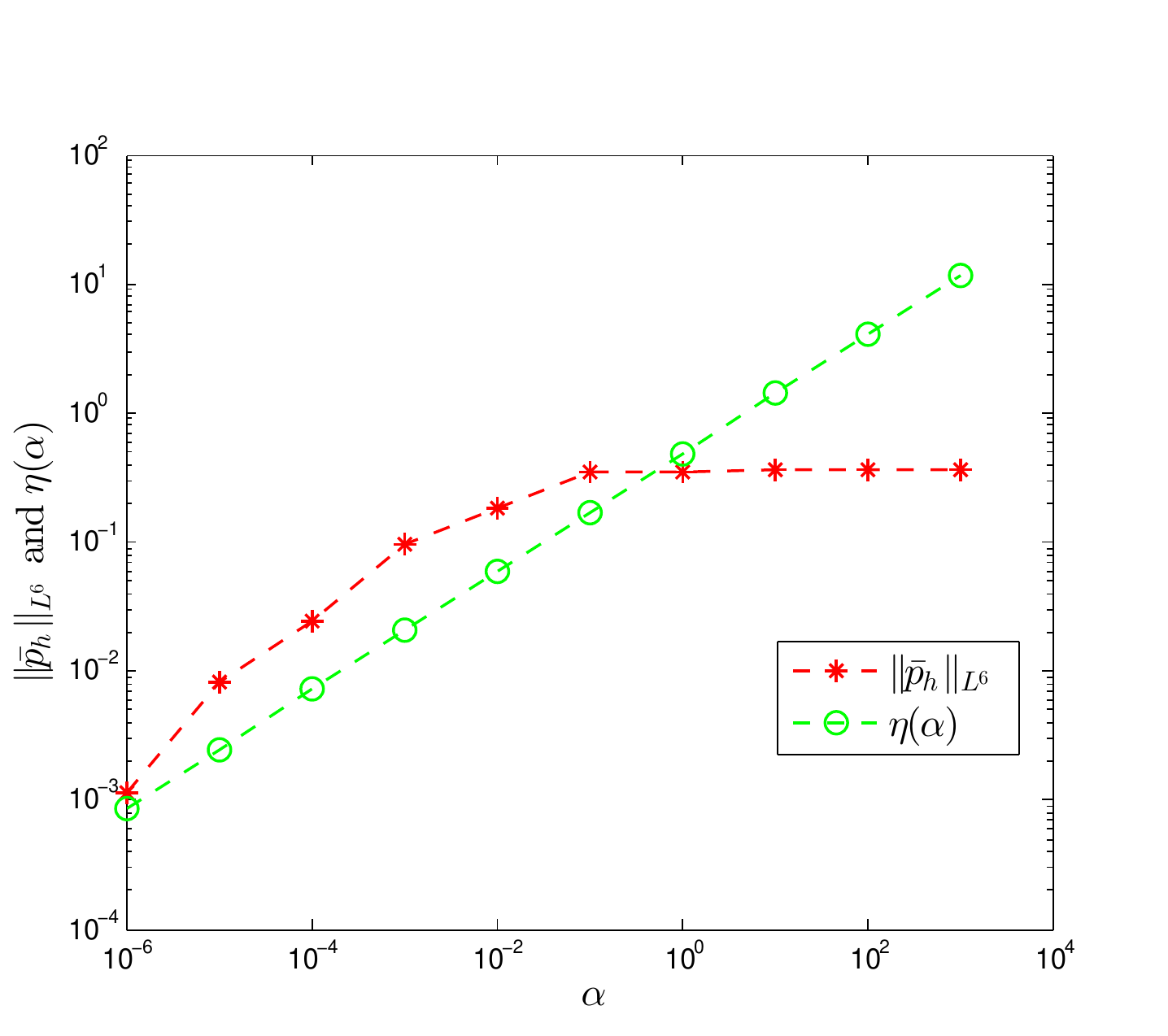}
                \caption{$\|\bar p_h\|_{L^6}$ and $\eta(\alpha)$ vs. $\alpha$.}
        \end{subfigure}%
        ~ 
        \begin{subfigure}[h!]{0.5\textwidth}
                \includegraphics[trim = 30mm 80mm 30mm 70mm, clip, width=\textwidth]{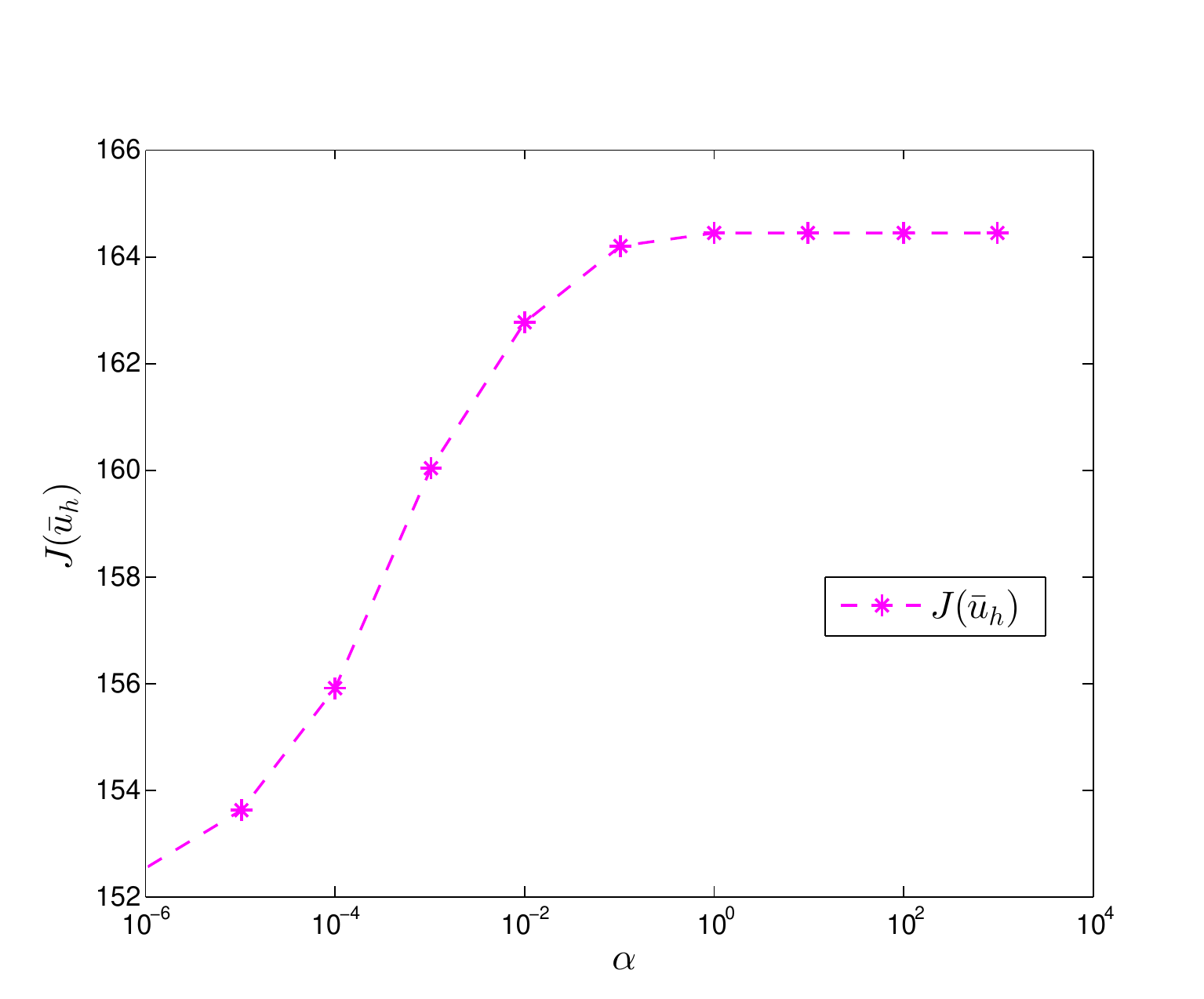}
                \caption{$J(\bar u_h)$  vs. $\alpha$.}
        \end{subfigure}
        
        \begin{subfigure}[h!]{0.5\textwidth}
                \includegraphics[trim = 40mm 80mm 30mm 70mm, clip, width=\textwidth]{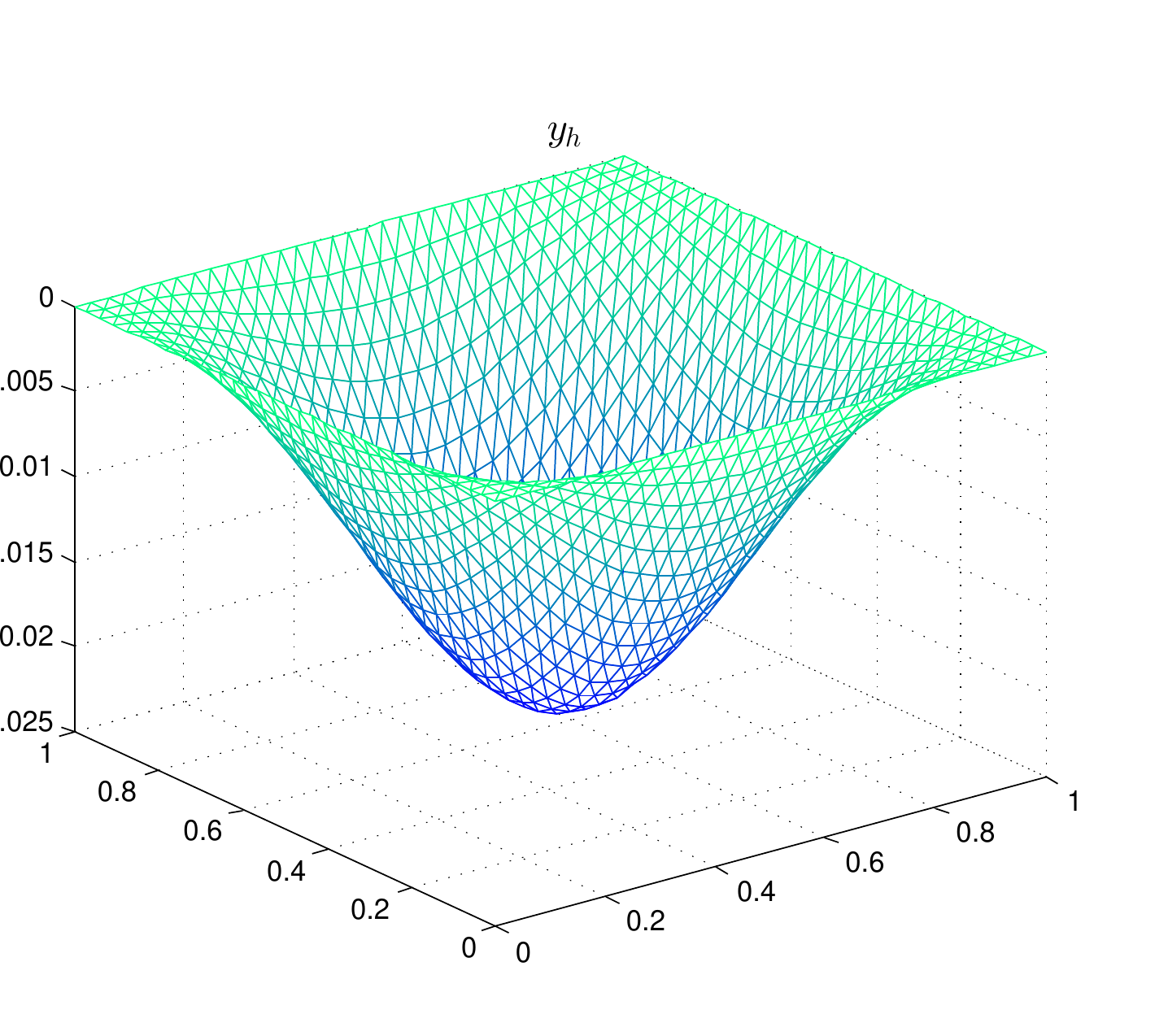}
                \caption{The optimal state $\bar y_h$.}
        \end{subfigure}~
        \begin{subfigure}[h!]{0.5\textwidth}
                \includegraphics[trim = 40mm 80mm 30mm 70mm, clip, width=\textwidth]{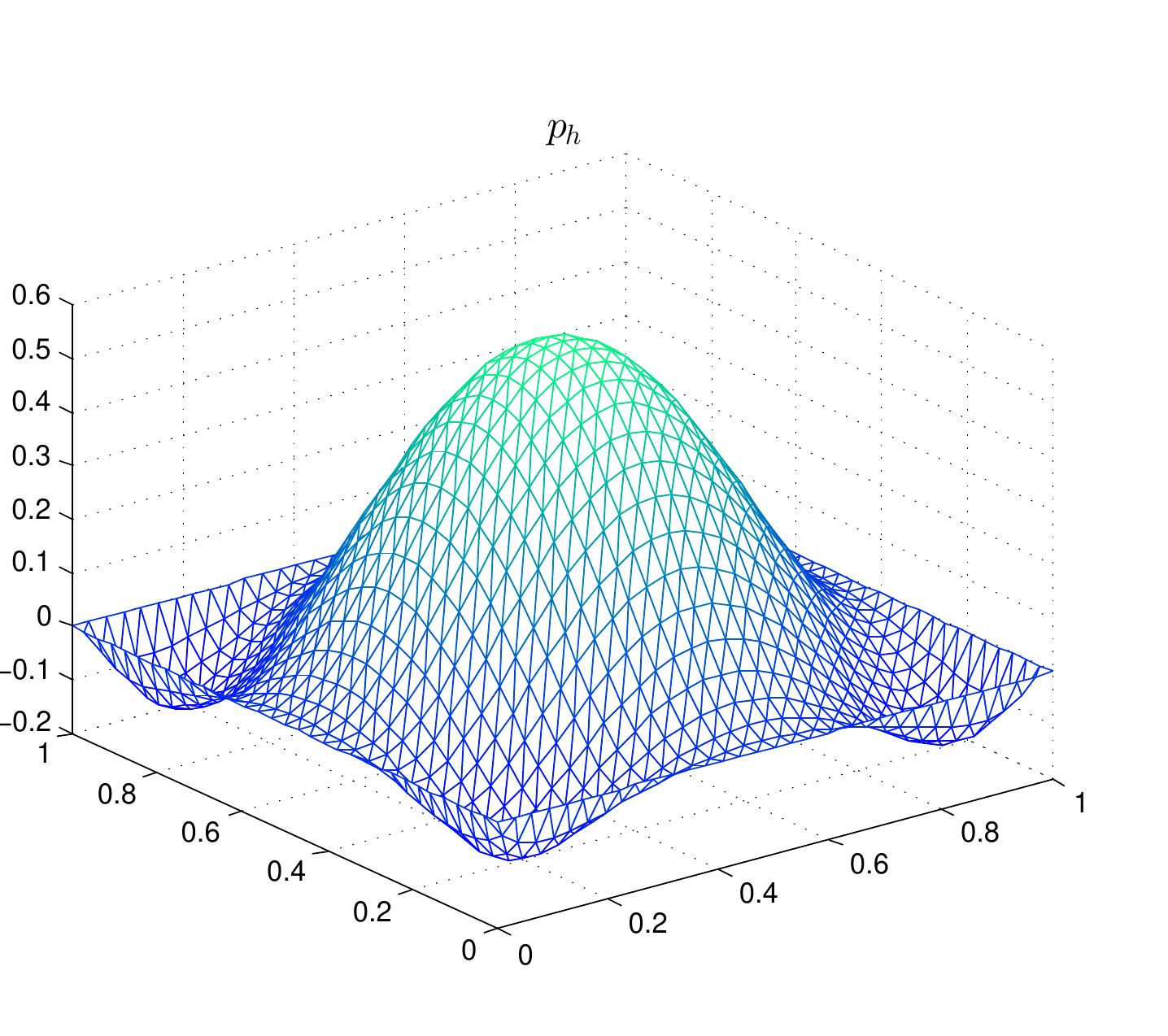}
                \caption{The adjoint state $\bar p_h$.}
        \end{subfigure}
        
        \begin{subfigure}[h!]{0.5\textwidth}
                \includegraphics[trim = 40mm 80mm 30mm 70mm, clip, width=\textwidth]{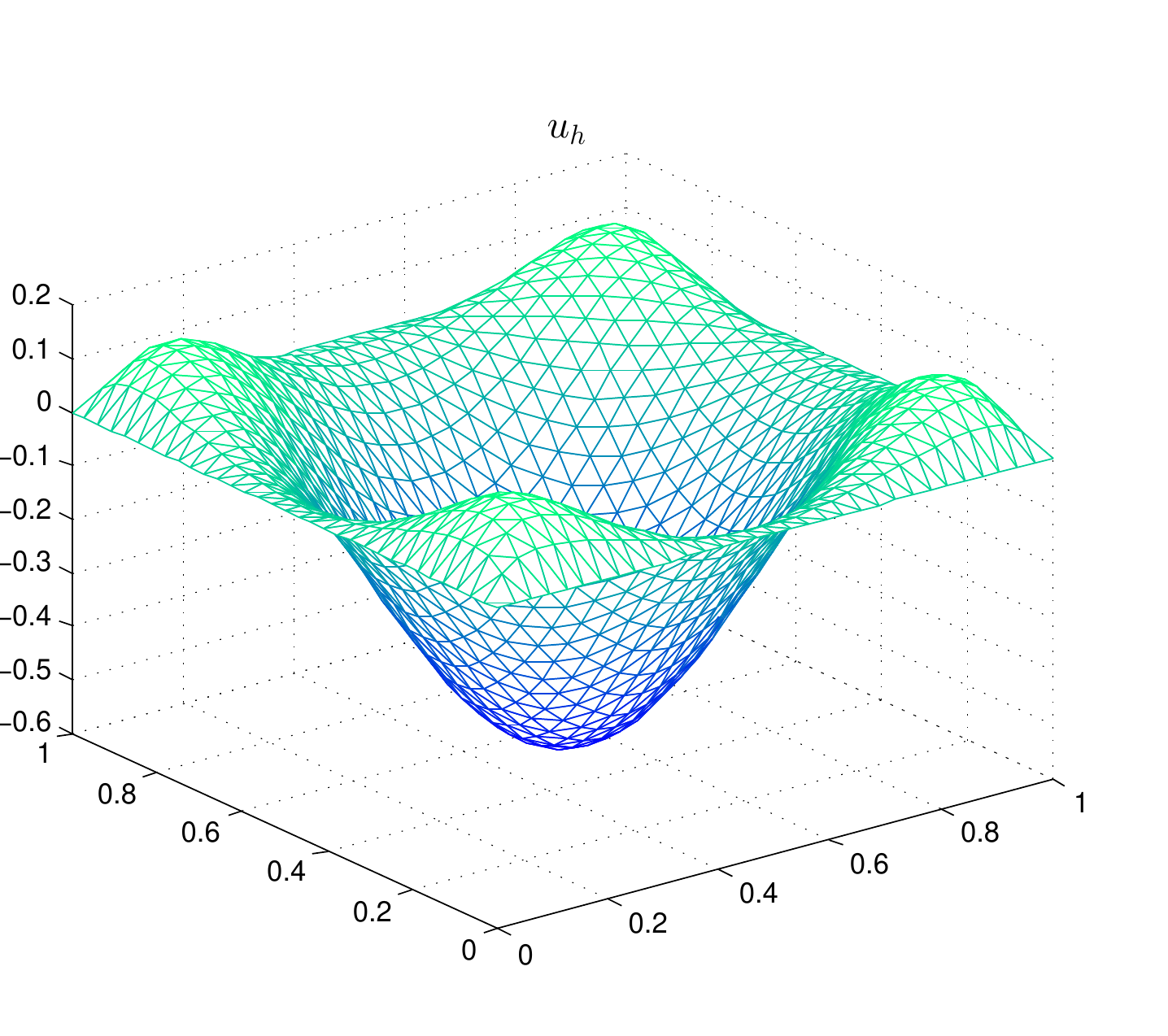}
                \caption{The optimal control $\bar u_h$.}
        \end{subfigure}
        
        \caption{Example~\ref{example: y power 5} Case~3 with choice~\textbf{A2} for $y_0$: The values of $\|\bar p_h\|_{L^6}$, $\eta(\alpha)$ and $J(\bar u_h)$ vs. $\alpha$. The optimal state $\bar y_h$, the optimal control $\bar u_h$ and the adjoint state $\bar p_h$ for $\alpha=1$.}
        \label{figure: example y5 case(constrained state) choiceA2}
\end{figure}

\end{example}

\section{Appendix}

\begin{lemma}
\label{lemma:0}
We have for $a,b \geq0, \lambda,\mu >0$ that
\begin{displaymath}
a^{\lambda} b^{\mu} \leq \frac{\lambda^{\lambda} \mu^{\mu}}{(\lambda+\mu)^{\lambda+\mu}} (a+b)^{\lambda+\mu}.
\end{displaymath}
\end{lemma}
\begin{proof}
Apply Young's inequality $ xy \leq \tfrac{1}{P} \, x^P + \tfrac{1}{Q} \,  y^Q \, x,y \geq 0, \,  \tfrac{1}{P}+\tfrac{1}{Q}=1$ \\
to $P=\frac{\lambda+ \mu}{\lambda}$, $Q=\frac{\lambda+ \mu}{\mu}$ and 
$x= \bigl( P a \bigr)^{\frac{1}{P}}, \, y= \bigl( Q b \bigr)^{\frac{1}{Q}}$.
\end{proof}

\begin{lemma}
\label{lemma: 1}
Suppose that Assumption 1 holds. Then we have for $a,b \in \mathbb{R}$
\begin{equation*}
\Big|  \int_0^1 \phi'\big(t a + (1-t)b \big)-\phi'(b) \, dt  \Big| \leq | a - b| L_r \Big(  \int_0^1 \phi'\big(t a+ (1-t)b \big) \, dt  \Big)^{\frac{1}{r}},
\end{equation*}
where
\[
L_r:= M \big(\frac{r-1}{2r-1} \big)^{\frac{r-1}{r}}.
\]
\end{lemma}
\begin{proof}
We start by noticing that
\begin{align*}
\int_0^1 \phi'\big(t a+ (1-t)b \big)-\phi'(b) \, dt  &= \int_0^1 \int_0^t \phi''\big(\tau a+ (1-\tau)b \big)(a-b) \, d\tau \, dt \\
& = (a-b)  \int_0^1  (1-t)\phi''\big(t a+ (1-t)b \big) \, dt. 
\end{align*}
Therefore, taking the absolute value and using  Assumption 1 we get 
\begin{eqnarray*}
\lefteqn{ \hspace{-1.5cm}
\Big|  \int_0^1 \phi'\big(t a+ (1-t)b \big)-\phi'(b) \, dt  \Big| \leq | a - b | M \int_0^1  (1-t)\phi'\big(t a+ (1-t)b \big)^{\frac{1}{r}} \, dt } \\
& \leq & | a - b | M \| 1-t \|_{L^{r'}(0,1)} \Big(  \int_0^1 \phi'\big(t a+ (1-t)b \big) \, dt  \Big)^{\frac{1}{r}},
\end{eqnarray*}
where $\tfrac{1}{r}+\tfrac{1}{r'}=1$. It is easy to see that
\[
\| 1-t \|_{L^{r'}(0,1)} = \big(\frac{1}{r'+1} \big)^{\frac{1}{r'}}=\big(\frac{r-1}{2r-1} \big)^{\frac{r-1}{r}}.
\]
Denoting $M \| 1-t \|_{L^{r'}(0,1)}$ by $L_r$ completes the proof.
\end{proof}

\begin{theorem} \label{lemma: 2} ({\bf Gagliardo--Nirenberg interpolation inequality}) \\
For  $2 \leq q < \infty$ we define $\theta=1-\frac{2}{q}$ as well as
\begin{displaymath}
GN_q:= \sup_{f \in H^1(\mathbb{R}^2), f \neq 0} \frac{\Vert f \Vert_{L^q(\mathbb{R}^2)}}{\Vert f \Vert_{L^2(\mathbb{R}^2)}^{1-\theta} \Vert \nabla f \Vert_{L^2(\mathbb{R}^2)}^{\theta}}.
\end{displaymath}
Then $GN_q \leq C_q:= \min(C^{(1)}_q,C^{(2)}_q,C^{(3)}_q)$, where
\begin{eqnarray}
C^{(1)}_q &=& \bigl( \theta C_{2,2 \theta} \bigr)^{-\theta}, \quad \mbox{ if } q \geq 4; \label{c1} \\[3mm]
C^{(2)}_q & = & \frac{1}{\sqrt{\theta^{\theta} (1- \theta)^{1 -\theta}}} \bigl( 2 \pi B(1,\frac{2(1-\theta)}{2 \theta}) \bigr)^{\theta/2} k_B(\frac{4}{2+2 \theta}); 
\label{c2} \\[4mm]
C^{(3)}_q & = &   \bigl(  \frac{1}{\pi} \bigr)^{\frac{q-2}{2q}}  \prod_{j=2}^{\infty} \bigl( \frac{2^j}{2^j +q-2} \bigr)^{\frac{2^j+2-q}{2^j q}}. \label{c3}
\end{eqnarray}
Here,
\begin{eqnarray*}
C_{2,s} & = & 2^{1/s} \bigl( \frac{2-s}{s-1} \bigr)^{(s-1)/s} \Bigl( 2 \pi B(\frac{2}{s},3-\frac{2}{s}) \Bigr)^{1/2}, \; 1< s < 2; \; C_{2,1}=2 \sqrt{\pi}; \\
B(a,b) & = & \frac{\Gamma(a) \Gamma(b)}{\Gamma(a+b)}, \quad a,b>0 \\
k_B(p) & = &  \bigl( \frac{p}{2 \pi} \bigr)^{1/p} \bigl( \frac{p'}{2 \pi} \bigr)^{- 1/p'}, \quad \frac{1}{p} + \frac{1}{p'}=1.
\end{eqnarray*} 
\end{theorem}
\begin{proof} The bounds (\ref{c1}) and (\ref{c2}) can be found in the paper \cite{veling2002lower} by Veling. We remark that 
$GN_q=\lambda_{2,\theta}^{-1}$, where $\lambda_{2,\theta}$ is defined in \cite[(1.7)]{veling2002lower}. The estimate (\ref{c1}) is
\cite[(1.31)]{veling2002lower} (note that $\theta \geq \frac{1}{2} \Leftrightarrow q \geq 4$), while (\ref{c2}) is
\cite[(1.42),(1.43)]{veling2002lower}, where the latter bound has been proved by Nasibov in \cite{nasibov1990}. \\[3mm]
Let us now turn to the proof of (\ref{c3}). To begin, 
we claim that for all $k \in \mathbb{N}_0$
\begin{equation}  \label{induction}
\displaystyle
\Vert f \Vert_{L^q} \leq \bigl( \frac{1}{\pi} \bigr)^{\frac{1}{2}(1 - \frac{q_k}{q})}  
\prod_{j=2}^{k+1} \bigl( \frac{2^j}{2^j +q-2} \bigr)^{\frac{2^j+2-q}{2^j q}}
 \Vert f \Vert_{L^{q_k}}^{\frac{q_k}{q}} \Vert \nabla f \Vert_{L^2}^{1-\frac{q_k}{q}},
 \end{equation}
 where
 \begin{displaymath}
 q_k = 2^{-k} \bigl( q+ 2( 2^k -1) \bigr).
 \end{displaymath}
 The inequality clearly holds for $k=0$. Suppose that (\ref{induction}) is true for some $k \in \mathbb{N}_0$. We infer from
 Theorem 1 in \cite{DD02} for the case $d=2$ that
 \begin{equation}  \label{delpino}
 \displaystyle
 \Vert f \Vert_{L^{2p}} \leq A  \Vert f \Vert_{L^{p+1}}^{1-\theta} \Vert \nabla f \Vert_{L^2}^{\theta}, \qquad 1<p<\infty.
\end{equation}
Here, 
\begin{displaymath}
A=  \bigl( \frac{y(p-1)^2}{4 \pi} \bigr)^{\frac{\theta}{2}} \bigl( \frac{2y-2}{2y} \bigr)^{\frac{1}{2p}}  \bigl( \frac{\Gamma(y)}{
\Gamma(y-1)} \bigr)^{\frac{\theta}{2}} \quad \mbox{ with } \quad 
\theta = \frac{2(p-1)}{4p}, \quad y= \frac{p+1}{p-1}.
\end{displaymath}
Using the formula for $y$  and observing that $\Gamma(y)= (y-1) \Gamma(y-1)$, the expression for $A$ can be simplified to
\begin{displaymath}
A = \bigl( \frac{1}{\pi} \bigr)^{\frac{\theta}{2}} \bigl( \frac{p+1}{2} \bigr)^{\frac{\theta}{2} - \frac{1}{2p}}.
\end{displaymath}
We apply (\ref{delpino}) for $p=\frac{1}{2} q_k$ and obtain
\begin{equation}  \label{lqk1}
\displaystyle
\Vert f \Vert_{L^{q_k}} \leq A  \Vert f \Vert_{L^{\frac{1}{2} q_k+1}}^{1-\theta} \Vert \nabla f \Vert_{L^2}^{\theta},
\end{equation}
where
\begin{displaymath}
A=  \bigl( \frac{1}{\pi} \bigr)^{\frac{\theta}{2}} \bigl( \frac{\frac{1}{2} q_k+1}{2} \bigr)^{\frac{\theta}{2} - \frac{1}{q_k}} \quad
\mbox{ and } \quad \theta = \frac{q_k -2}{2 q_k}.
\end{displaymath}
Since $\frac{1}{2} q_k+1= q_{k+1}$ we find that
\begin{displaymath}
A=   \bigl( \frac{1}{\pi} \bigr)^{\frac{\theta}{2}} \bigl( \frac{q_{k+1}}{2} \bigr)^{\frac{\theta}{2} - \frac{1}{q_k}} \quad \mbox{ and }
\quad \theta = 1 - \frac{q_{k+1}}{q_k},
\end{displaymath}
which, inserted into (\ref{lqk1}) yields
\begin{equation} \label{lqk2}
\displaystyle
\Vert f \Vert_{L^{q_k}} \leq \bigl( \frac{1}{\pi} \bigr)^{\frac{\theta}{2}} \bigl( \frac{q_{k+1}}{2} \bigr)^{\frac{\theta}{2} - \frac{1}{q_k}}
  \Vert f \Vert_{L^{q_{k+1}}}^{1-\theta} \Vert \nabla f \Vert^{\theta}.
\end{equation}
Using the induction hypothesis we infer
\begin{eqnarray*}
\lefteqn{
\Vert f \Vert_{L^q} \leq \bigl( \frac{1}{\pi} \bigr)^{\frac{1}{2}(1 - \frac{q_k}{q})+\frac{\theta}{2} \frac{q_k}{q}}  
 \bigl( \frac{q_{k+1}}{2} \bigr)^{(\frac{\theta}{2} - \frac{1}{q_k}) \frac{q_k}{q}} } \\
& & \times \prod_{j=2}^{k+1} \bigl( \frac{2^j}{2^j +q-2} \bigr)^{\frac{2^j+2-q}{2^j q}}
 \Vert f \Vert_{L^{q_{k+1}}}^{(1-\theta)\frac{q_k}{q}} \Vert \nabla f \Vert_{L^2}^{1-\frac{q_k}{q}+\theta \frac{q_k}{q}}.
\end{eqnarray*}
Elementary calculations show that
\begin{eqnarray*}
\frac{1}{2} \bigl(1 - \frac{q_k}{q} \bigr)+\frac{\theta}{2} \frac{q_k}{q} & = & \frac{1}{2} \bigl( 1 - \frac{q_{k+1}}{q} \bigr), \\
\bigl( \frac{q_{k+1}}{2} \bigr)^{(\frac{\theta}{2} - \frac{1}{q_k}) \frac{q_k}{q}} & = & \bigl( \frac{2^{k+2}}{2^{k+2} + q -2} \bigr)^{\frac{2^{k+2}+2-q}{2^{k+2} q}}, \\
(1-\theta) \frac{q_k}{q} & = & \frac{q_{k+1}}{q}, \\
1-\frac{q_k}{q}+\theta \frac{q_k}{q} & = & 1 - \frac{q_{k+1}}{q},
\end{eqnarray*}
which implies (\ref{induction}) for $k+1$. 
The result now follows by sending $k \rightarrow \infty$ in (\ref{induction}) and by observing that $\lim_{k \rightarrow \infty} q_k = 2$. 

\end{proof}

\newpage
\bibliographystyle{plain}

\end{document}